\documentclass[a4paper,reqno]{amsart}
\usepackage{amssymb, amsmath, amsthm, mathrsfs, ascmac,color}
\usepackage[pdftex]{graphicx} 
\usepackage[subrefformat=parens]{subcaption}
\usepackage[top=2.5cm, bottom=2.5cm, left=3cm,right=3cm]{geometry}
\usepackage{comment}
\usepackage[foot]{amsaddr}
\usepackage{caption}
\usepackage{here}

\newtheoremstyle{mystyle}
  {}
  {}
  {\normalfont}
  { }
  {\bfseries}
  {}
  {10pt}
  { }
\theoremstyle{mystyle}
\newtheorem{thm}{Theorem}[section]
\newtheorem{prop}[thm]{Proposition}
\newtheorem{lem}[thm]{Lemma}
\newtheorem{asm}{Assumption}
\newtheorem{rmk}{Remark}
\makeatletter

\@addtoreset{equation}{section}
\makeatother

\newcommand{\ind}{\boldsymbol 1}
\newcommand{\dd}{\mathrm d}
\newcommand{\ee}{\mathrm e}

\newcommand{\EE}{\mathbb E}
\newcommand{\VV}{\mathrm{Var}}

\newcommand{\re}{\mathrm{Re}}
\newcommand{\im}{\mathrm{Im}}

\newcommand{\iu}{\mathrm{i}}
\newcommand{\Fou}{\mathcal F}

\newcommand{\pto}{\stackrel{p}{\longrightarrow}}
\newcommand{\dto}{\stackrel{d}{\longrightarrow}}

\newcommand{\TT}{\mathsf T}
\newcommand{\bs}{\boldsymbol}

\allowdisplaybreaks
\begin{document}
\bibliographystyle{plain}
\title
[Parameter estimation for
linear parabolic SPDEs in two space dimensions] 
{
Parameter estimation for linear parabolic SPDEs 
in two space dimensions based on high frequency data 
}
\date{}
\author[Y. Tonaki]{Yozo Tonaki}
\author[Y. Kaino]{Yusuke Kaino}
\address[Y. Tonaki]
{Graduate School of Engineering Science, Osaka University, 1-3, 
Machikaneyama, Toyonaka, Osaka, 560-8531, Japan}
\address[Y. Kaino]{
Graduate School of Maritime Sciences, Kobe University,
5-1-1, Fukaeminami-machi, Higashinada-ku, Kobe, 658-0022, Japan}
\author[M. Uchida]{Masayuki Uchida}
\address[M. Uchida]{
Graduate School of Engineering Science, 
and Center for Mathematical Modeling and Date Science (MMDS),
Osaka University, 1-3, Machikaneyama, Toyonaka, Osaka, 560-8531, Japan
}


\maketitle
\begin{abstract}
We consider parameter estimation for a linear parabolic second-order 
stochastic partial differential equation (SPDE) 
in two space dimensions driven by two types $Q$-Wiener processes 
based on high frequency data in time and space.
We first estimate the parameters which appear in the coordinate process of the SPDE 
using the minimum contrast estimator based on the thinned data with respect to space, 
and then construct an approximate coordinate process of the SPDE.
Furthermore, we propose estimators of the coefficient parameters of the SPDE 
utilizing the approximate coordinate process based on 
the thinned data with respect to time. 
We also give some simulation results.
\end{abstract}

\section{Introduction}\label{sec1}
We consider the following linear parabolic 
stochastic partial differential equation (SPDE)
in two space dimensions
\begin{align}
\dd X_t^{Q}(y,z)
&=\biggl\{
\theta_2
\biggl(
\frac{\partial^2}{\partial y^2}+\frac{\partial^2}{\partial z^2}
\biggr)
+\theta_1\frac{\partial}{\partial y}
+\eta_1\frac{\partial}{\partial z}
+\theta_0
\biggr\}
X_t^{Q}(y,z)
\dd t
\nonumber
\\
&\qquad+\sigma\dd W_t^Q(y,z),
\quad
(t,y,z)\in[0,1]\times D,
\label{2d_spde}
\\
X_0^{Q}(y,z)
&=\xi(y,z),
\quad (y,z)\in D,
\nonumber
\\
X_t^{Q}(y,z)&=0,
\quad (t,y,z)\in [0,1]\times \partial D,
\nonumber
\end{align}
where $D=[0,1]^2$, 
$W_t^Q$ is a $Q$-Wiener process in a Sobolev space on $D$,
an initial value $\xi$ is independent of $W_t^Q$,
$\theta=(\theta_0,\theta_1,\eta_1,\theta_2)$ and $\sigma$ are unknown parameters
and $(\theta_0,\theta_1,\eta_1,\theta_2,\sigma)\in\mathbb R^3\times(0,\infty)^2$.
Moreover, the parameter space $\Theta$ is a compact convex subset of 
$\mathbb R^3\times(0,\infty)^2$, 
$\theta^*=(\theta_0^*,\theta_1^*,\eta_1^*,\theta_2^*)$ and $\sigma^*$
are true values of the parameters 
and we assume that $(\theta^*,\sigma^*)\in \mathrm{Int}\,\Theta$.
The data are discrete observations 
$\{X_{t_i}^{Q}(y_{j_1},z_{j_2})\}$, 
$i=0,\ldots,N$, $j_1=0,\ldots,M_1$, $j_2=0,\ldots,M_2$
with $t_i=i\Delta_N$, $y_{j_1}=j_1/M_1$ and $z_{j_2}=j_2/M_2$,  
where $\Delta_N=1/N$ and $M:=M_1M_2$.

SPDEs appear in various fields
such as physics, engineering, biology, and economics.
In particular, the linear parabolic SPDE is a fundamental equation,
and its typical example is the stochastic heat equation.
The heat equation is known as the equation which describes the heat conduction 
in an object and the diffusion phenomenon of particles, 
and it is an essential equation which appears in various situations.
For example, in the heat equation with one space dimension, 
we can consider the temperature variability of a thin object such as a wire, 
or a sea surface on a straight line. 
However, it is insufficient to consider practical problems 
with only a heat equation in one space dimension.
Actually, in our daily problems, 
we often deal with two or three dimensional heat phenomena, 
such as temperature variability of a thin steel plate, sea surface, solid, or seawater.
For instance, Piterbarg and Ostrovskii \cite{Piterbarg_Ostrovskii1997}
treated sea surface temperature variability.
Thus, it is important to analyze the SPDEs in two space dimensions 
because it can address more general problems than SPDEs in one space dimension.

Statistical inference for SPDE models based on 
discrete observations has been studied by many researchers, see, for example, 
Markussen \cite{Markussen2003},
Bibinger and Trabs \cite{Bibinger_Trabs2020},
Chong  \cite{Chong2020}, \cite{Chong2019arXiv},
Cialenco et al. \cite{Cialenco_etal2020},
Cialenco and Huang \cite{Cialenco_Huang2020},
Hildebrandt \cite{Hildebrandt2020}, 
Kaino and Uchida \cite{Kaino_Uchida2020} 
and references therein.
Recently, Kaino and Uchida \cite{Kaino_Uchida2021} proposed 
the adaptive maximum likelihood type estimation for 
the coefficient parameters of linear parabolic second-order 
SPDEs in one space dimension with a small noise. 
Hildebrandt and Trabs 
\cite{Hildebrandt_Trabs2021}, \cite{Hildebrandt_Trabs2021arXiv}
studied the estimation for the coefficient parameters of 
linear parabolic and semilinear SPDEs in one space dimension, respectively, 
using a contrast function with double increments in time and space.

In particular, 
Bibinger and Trabs \cite{Bibinger_Trabs2020} 
treated the following linear parabolic SPDE in one space dimension 
\begin{align}
\dd X_t(y)
&=\biggl(
\theta_2\frac{\partial^2}{\partial y^2}
+\theta_1\frac{\partial}{\partial y}
+\theta_0
\biggr)
X_t(y)
\dd t
+\sigma\dd B_t(y),
\quad
(t,y)\in[0,T]\times[0,1],
\label{1d_spde}
\\
X_0(y)
&=\xi(y),
\quad y\in [0,1],
\qquad
X_t(0)=X_t(1)=0,
\quad t\in [0,T],
\nonumber
\end{align}
where $T>0$, $B_t$ is a cylindrical Brownian motion in a Sobolev space on $[0,1]$,
$\xi$ is an initial value,
$\theta_0,\theta_1,\theta_2$ and $\sigma$ are unknown parameters
and $(\theta_0,\theta_1,\theta_2,\sigma)\in\mathbb R^2\times(0,\infty)^2$.
They proposed minimum contrast estimators for 
$\kappa=\theta_1/\theta_2$, $\sigma_0^2=\sigma^2/\theta_2^{1/2}$
in the case where $T$ is fixed.
Since the coordinate process $x_k(t)$, $k\ge1$ of 
the SPDE \eqref{1d_spde} is expressed as 
\begin{equation*}
x_k(t)
=\sqrt2\int_0^1 X_t(y)\sin(\pi k y)\ee^{\kappa y/2} \dd y
\end{equation*}
and is a diffusion process satisfying the stochastic differential equation
\begin{equation}\label{cp0}
\dd x_{k}(t)
=-\lambda_{k} x_{k}(t)\dd t+\sigma\dd w_{k}(t),
\quad
x_k(0)=\sqrt2\int_0^1 \xi(y)\sin(\pi k y)\ee^{\kappa y/2} \dd y
\end{equation}
where $\{w_k\}_{k\ge1}$ is independent real valued standard Brownian motions
and $\lambda_k=-\theta_0+\frac{\theta_1^2}{4\theta_2}+\pi^2k^2\theta_2$,
Kaino and Uchida \cite{Kaino_Uchida2020} constructed 
an approximate coordinate process by a Riemann sum 
\begin{equation*}
\hat x_k(t)=\frac{1}{M}\sum_{j=1}^{M}
\sqrt2X_t(y_j)\sin(\pi k y_j)\ee^{\hat\kappa y_j/2}
\end{equation*}
with observations of space $\{y_j\}_{j=1}^M$ and an estimator $\hat\kappa$ of $\kappa$,
and proposed estimators of $\sigma^2$, $\theta_2$ and $\theta_1$
using the approximate coordinate process based on 
the thinned data with respect to time
and statistical inference for diffusion processes.
Furthermore, they extended the results of \cite{Bibinger_Trabs2020} to the case where
$T$ is large and proposed estimators of 
$\sigma^2$, $\theta_2$, $\theta_1$ and $\theta_0$.
For details of statistical inference for diffusion processes
based on discrete observations,
see 
Genon-Catalot and Jacod \cite{Genon-Catalot_Jacod1993},
Kessler \cite{Kessler1997},
and
Uchida and Yoshida \cite{Uchida_Yoshida2012}, \cite{Uchida_Yoshida2013}.

In this paper, we apply 
the estimation method based on the approximate coordinate process
to SPDEs in two space dimensions 
and consider the estimation for each coefficient parameter of the SPDE \eqref{2d_spde}.
In this case, we need to be careful in setting a noise $W_t^Q$ 
because $X_t^Q$ may not be square integrable for any $t>0$ 
depending on the choice of the $Q$-Wiener process.
Moreover, since the method of \cite{Kaino_Uchida2020} is  
based on the property that the random filed $X_t$ admits a spectral decomposition
\begin{equation*}
X_t(y)=\sum_{k=1}^{\infty}x_k(t)e_k(y)
\end{equation*}
with $e_k(y)=\sqrt{2}\sin(\pi k y)\ee^{-\kappa y/2}$ and $x_k(t)$ in \eqref{cp0},
it is also required to set a noise $W_t^Q$ such that 
the random field $X_t^Q$ in the SPDE \eqref{2d_spde} can be decomposed.
See H\"{u}bner et al. \cite{Hubner_etal1993}
and Cialenco and Glatt-Holtz \cite{Cialenco_Glatt-Holtz2011}
for parameter estimation for SPDEs driven by a $Q$-Wiener process. 
See also the surveys by Lototsky \cite{Lototsky2009} and Cialenco \cite{Cialenco2018}
for statistical inference based on the spectral approach.

This paper is organized as follows.
In Section \ref{sec2}, we present the setting of our model.
We also discuss how to choose a $Q$-Wiener process 
and introduce two types of $Q$-Wiener processes.
In Section \ref{sec3}, we first propose minimum contrast estimators 
of the parameters appearing in the coordinate process of the SPDE 
by using the thinned data with respect to space.
Next, we construct an approximate coordinate process 
by using these minimum contrast estimators, 
and provide estimators of the coefficient parameters of the SPDE 
based on the thinned data with respect to time.
We then show that the estimators of the coefficient parameters 
are asymptotically normal.
In Section \ref{sec4}, we give some simulation studies.
Section \ref{sec5} is devoted to the proofs of our assertions in Section \ref{sec3}.
In order to understand the characteristics of the parameters 
of the SPDE \eqref{2d_spde},
the sample paths with different values of the 
parameters are provided in Appendix section.

\section{Preliminaries}\label{sec2}
Let $(\Omega,\mathscr F, \{{\mathscr F}_t\}_{t\ge0} P)$ 
be a stochastic basis with usual conditions,
and let $\{w_{k,\ell}\}_{k,\ell\in\mathbb N}$
be independent real valued standard Brownian motions on this basis.

Let $A_\theta$ be the differential operator defined by 
\begin{equation*}
-A_\theta
=\theta_2
\biggl(
\frac{\partial^2}{\partial y^2}+\frac{\partial^2}{\partial z^2}
\biggr)
+\theta_1\frac{\partial}{\partial y}
+\eta_1\frac{\partial}{\partial z}
+\theta_0.
\end{equation*}
Note that the SPDE \eqref{2d_spde} is represented as  
\begin{equation*}
\dd X_t^{Q}(y,z)=-A_\theta X_t^{Q}(y,z)\dd t+\sigma\dd W_t^Q(y,z).
\end{equation*}
For $k,\ell\in\mathbb N$,
the eigenfunctions $e_{k,\ell}$ of $A_\theta$ and 
the corresponding eigenvalues $\lambda_{k,\ell}$ are given by 
\begin{align*}
e_{k,\ell}(y,z)
&=2\sin(\pi k y)\sin(\pi \ell z)
\ee^{-\frac{\theta_1}{2\theta_2}y}\ee^{-\frac{\eta_1}{2\theta_2}z},
\quad (y,z)\in D,
\\
\lambda_{k,\ell}
&=-\theta_0+\frac{\theta_1^2+\eta_1^2}{4\theta_2}+\pi^2(k^2+\ell^2)\theta_2.
\end{align*}
We then obtain that 
$A_\theta e_{k,\ell}=\lambda_{k,\ell}e_{k,\ell}$
for $k,\ell\in\mathbb N$.
For real valued functions $f$ and $g$ defined on $D$, let
\begin{equation*}
\langle f,g\rangle_\theta
=\int_0^1\int_0^1
f(y,z)g(y,z)
\ee^{\frac{\theta_1}{\theta_2}y}\ee^{\frac{\eta_1}{\theta_2}z} 
\dd y\dd z,
\quad
\|f\|_\theta=\langle f,f\rangle_\theta^{1/2}.
\end{equation*}
Moreover, set
\begin{equation*}
H_\theta
=\{f:D\to\mathbb R|\, \|f\|_\theta<\infty \text{ and } 
f(y,z)=0\text{ for } (y,z)\in \partial D  \}.
\end{equation*}
We consider two types of $Q$-Wiener processes,
$\{W_t^{Q_1}\}_{t\ge0}$ and $\{W_t^{Q_2}\}_{t\ge0}$, defined as follows.
\begin{equation}\label{QWp_ver1}
\langle W_t^{Q_1},f\rangle_\theta
=\sum_{k,\ell\ge1}\lambda_{k,\ell}^{-\alpha/2}
\langle f,e_{k,\ell}\rangle_\theta w_{k,\ell}(t), 
\end{equation}
\begin{equation}\label{QWp_ver2}
\langle W_t^{Q_2},f\rangle_\theta
=\sum_{k,\ell\ge1}\mu_{k,\ell}^{-\alpha/2}
\langle f,e_{k,\ell}\rangle_\theta w_{k,\ell}(t)
\end{equation}
for $f\in H_\theta$ and $t\ge0$, 
where 
$\mu_{k,\ell}=\pi^2(k^2+\ell^2)+\mu_0$, $\mu_0\in(-2\pi^2,\infty)$
and $\alpha\in(0,1)$.
$\mu_0$ is an unknown parameter (may be known),
the parameter space of $\mu_0$ is a compact convex subset of 
$(-2\pi^2,\infty)$ and the true value $\mu_0^*$ belongs to its interior.
The restriction $\alpha>0$ is for mathematical reasons, 
and the restriction $\alpha<1$ is for statistical inferences.
See Remarks \ref{rmk1}, \ref{rmk2} and \ref{rmk4} 
for details of the $Q$-Wiener processes
and the restriction of $\alpha$.

We assume that $\xi\in H_\theta$ and 
$\lambda_{1,1}^*=-\theta_0^*
+\frac{(\theta_1^*)^2+(\eta_1^*)^2}{4\theta_2^*}+2\pi^2\theta_2^*>0$.
$X_t^{Q}$ is called a mild solution of \eqref{2d_spde} on $D$ 
if it satisfies that for any $t\in[0,1]$, 
\begin{equation*}
X_t^{Q}=
\ee^{-tA_\theta}\xi+\sigma\int_0^t\ee^{-(t-s)A_\theta}\dd W_s^Q
\quad \mathrm{a.s.},
\end{equation*}
where
$\ee^{-tA_\theta}u=\sum_{k,\ell\ge1}\ee^{-\lambda_{k,\ell}t}
\langle u,e_{k,\ell}\rangle_\theta e_{k,\ell}$ for $u\in H_\theta$.
By defining the $Q_1$-Wiener process $W_t^{Q_1}$ in \eqref{QWp_ver1}, 
the random field $X_t^{Q_1}(y,z)$ is expressed as
\begin{equation}\label{sd_ver1}
X_t^{Q_1}(y,z)
=\sum_{k,\ell\ge1}x_{k,\ell}^{Q_1}(t)e_{k,\ell}(y,z),
\end{equation}
where the coordinate process 
\begin{equation*}
x_{k,\ell}^{Q_1}(t)
=\langle X_t^{Q_1},e_{k,\ell}\rangle_\theta
=\ee^{-\lambda_{k,\ell}t}\langle \xi,e_{k,\ell}\rangle_\theta
+\sigma\int_0^t\lambda_{k,\ell}^{-\alpha/2}
\ee^{-\lambda_{k,\ell}(t-s)}\dd w_{k,\ell}(s)
\end{equation*}
is the Ornstein-Uhlenbeck process 
which satisfies the stochastic differential equation
\begin{align*}
\dd x_{k,\ell}^{Q_1}(t)
=-\lambda_{k,\ell} x_{k,\ell}^{Q_1}(t)\dd t
+\sigma\lambda_{k,\ell}^{-\alpha/2}\dd w_{k,\ell}(t),
\quad
x_{k,\ell}^{Q_1}(0)=\langle \xi,e_{k,\ell}\rangle_\theta.
\end{align*}
Note that
\begin{equation}\label{cp2_ver1}
x_{k,\ell}^{Q_1}(t)
=
2\int_0^1\int_0^1
X_t^{Q_1}(y,z)\sin(\pi k y)\sin(\pi \ell z)
\ee^{\frac{\theta_1}{2\theta_2}y}\ee^{\frac{\eta_1}{2\theta_2}z}
\dd y\dd z.
\end{equation}
Similarly,  by defining the $Q_2$-Wiener process $W_t^{Q_2}$ in \eqref{QWp_ver2}, 
the random field $X_t^{Q_2}(y,z)$ is spectrally decomposed as
\begin{equation*}
X_t^{Q_2}(y,z)
=\sum_{k,\ell\ge1}x_{k,\ell}^{Q_2}(t)e_{k,\ell}(y,z),
\end{equation*}
where the coordinate process 
\begin{equation*}
x_{k,\ell}^{Q_2}(t)
=\langle X_t^{Q_2},e_{k,\ell}\rangle_\theta
=\ee^{-\lambda_{k,\ell}t}\langle \xi,e_{k,\ell}\rangle_\theta
+\sigma\int_0^t\mu_{k,\ell}^{-\alpha/2}
\ee^{-\lambda_{k,\ell}(t-s)}\dd w_{k,\ell}(s)
\end{equation*}
is described by the stochastic differential equation
\begin{equation}\label{dp_ver2}
\dd x_{k,\ell}^{Q_2}(t)
=-\lambda_{k,\ell} x_{k,\ell}^{Q_2}(t) \dd t
+\sigma\mu_{k,\ell}^{-\alpha/2}\dd w_{k,\ell}(t),
\quad
x_{k,\ell}^{Q_2}(0)=\langle \xi,e_{k,\ell}\rangle_\theta
\end{equation}
and is also represented as
\begin{equation*}
x_{k,\ell}^{Q_2}(t)
=2\int_0^1\int_0^1
X_t^{Q_2}(y,z)\sin(\pi k y)\sin(\pi \ell z)
\ee^{\frac{\theta_1}{2\theta_2}y}\ee^{\frac{\eta_1}{2\theta_2}z}
\dd y\dd z.
\end{equation*}

\begin{rmk}\label{rmk1}
Consider the case where the SPDE \eqref{2d_spde} 
is driven by a cylindrical Brownian motion $\{B_t\}_{t\ge0}$ defined as 
\begin{equation}\label{cBm}
\langle B_t,f\rangle_\theta
=\sum_{k,\ell\ge1}
\langle f,e_{k,\ell}\rangle_\theta w_{k,\ell}(t), 
\quad
f\in H_\theta,\ t\ge0
\end{equation}
as in the setting of previous studies on parameter estimation 
for SPDEs in one space dimension, in other words, 
consider the SPDE represented by
\begin{equation}\label{2d_spde_ver2}
\dd X_t^{I}(y,z)=-A_\theta X_t^{I}(y,z)\dd t+\sigma\dd W_t^I(y,z),
\end{equation}
where $I$ is the identity operator and $B_t=W_t^I$.
In this case, the coordinate process 
$x_{k,\ell}^I(t)$ of the SPDE \eqref{2d_spde_ver2} is expressed as
$x_{k,\ell}^I(t)
=\ee^{-\lambda_{k,\ell}t}\langle \xi,e_{k,\ell}\rangle_\theta$  
$+\sigma\int_0^t \ee^{-\lambda_{k,\ell}(t-s)}\dd w_{k,\ell}(s)$.
Since there exists a constant $C_1>0$ such that 
$\lambda_{k,\ell}\le C_1(k^2+\ell^2)$ for $k,\ell\ge 1$,
it follows that for any $t>0$,
\begin{align*}
\EE[\|X_t^I\|_\theta^2]
=\sum_{k,\ell\ge1}\EE[x_{k,\ell}^I(t)^2]
&=
\sum_{k,\ell\ge1}
\ee^{-2\lambda_{k,\ell}t}
\EE[\langle \xi,e_{k,\ell}\rangle_\theta^2]
+\sigma^2\sum_{k,\ell\ge1}
\frac{1-\ee^{-2\lambda_{k,\ell}t}}{2\lambda_{k,\ell}}
\\
&\ge \sigma^2\sum_{k,\ell\ge1}\frac{1-\ee^{-2\lambda_{k,\ell}t}}{2\lambda_{k,\ell}}
\\
&\ge C \sum_{k,\ell\ge 1}\frac{1}{\lambda_{k,\ell}}
=\infty,
\end{align*}
and then $\EE[X_t^I(y,z)^2]=\infty$ for some $(y,z)\in D$.
In order to avoid this inconvenience, 
we consider a noise $W_t^Q$ with a damping factor such as 
$\{\lambda_{k,\ell}^{-\alpha/2}\}_{k,\ell\ge1}$ in \eqref{QWp_ver1}.
Let $\mathscr D(A_{\theta}^{-1/2})\supset H_\theta$ 
be the domain of $A_{\theta}^{-1/2}$.
Define the covariance operator $Q_1$ on $\mathscr D(A_\theta^{-1/2})$ by
\begin{equation*}
Q_1 e_{k,\ell}=\lambda_{k,\ell}^{-\alpha}\|e_{k,\ell}\|_{\theta,-1/2}^2 e_{k,\ell},
\end{equation*}
where $\alpha>0$, $\|u\|_{\theta,-1/2}=\|A_\theta^{-1/2}u\|_\theta$, 
and then $\{e_{k,\ell}/\|e_{k,\ell}\|_{\theta,-1/2}\}_{k,\ell\ge1}$
is the complete orthonormal system on $\mathscr D(A_\theta^{-1/2})$ 
and the corresponding eigenvalues of $Q_1$ are 
$\lambda_{k,\ell}^{-\alpha}\|e_{k,\ell}\|_{\theta,-1/2}^2$. 
Noting that
\begin{equation*}
\sum_{k,\ell\ge1}\|A_\theta^{-\alpha/2}e_{k,\ell}\|_{\theta,-1/2}^2
=\sum_{k,\ell\ge1}\frac{1}{\lambda_{k,\ell}^{1+\alpha}}
<\infty
\end{equation*}
for $\alpha>0$, the $Q_1$-Wiener process $\{W_t^{Q_1}\}_{t\ge0}$ 
is well-defined in $\mathscr D(A_\theta^{-1/2})$ and 
it follows that for $f\in H_\theta$,
\begin{equation*}
\langle W^{Q_1}_t,f\rangle_\theta
=\sum_{k,\ell \ge1}
\lambda_{k,\ell}^{-\alpha/2}\|e_{k,\ell}\|_{\theta,-1/2}
\biggl\langle f,\frac{e_{k,\ell}}{\|e_{k,\ell}\|_{\theta,-1/2}}\biggr\rangle_\theta
w_{k,\ell}(t)
=\sum_{k,\ell\ge1}
\lambda_{k,\ell}^{-\alpha/2}
\langle f,e_{k,\ell}\rangle_\theta
w_{k,\ell}(t),
\end{equation*}
and thus we obtain \eqref{QWp_ver1}.
Refer to 
Lord et al. \cite{Lord_etal2014},
Da Prato and Zabczyk \cite{DaPrato_Zabczyk2014}
and
Lototsky and Rozovsky \cite{Lototsky_Rozovsky2017}
for details on the $Q$-Wiener process.
\end{rmk}

\begin{rmk}\label{rmk2}
Unlike the damping factor $\{\lambda_{k,\ell}^{-\alpha/2}\}_{k,\ell\ge1}$
of the $Q_1$-Wiener process,
the damping factor $\{\mu_{k,\ell}^{-\alpha/2}\}_{k,\ell\ge1}$ 
of the $Q_2$-Wiener process 
does not include the parameter $\theta$ of the differential operator $A_\theta$. 
Moreover, by setting $\kappa=\theta_1/\theta_2$, $\eta=\eta_1/\theta_2$ and 
$\zeta=(\frac{\kappa^2+\eta^2}{4}-\mu_0,\kappa,\eta,1)$,
it can be regarded as
\begin{equation*}
\mu_{k,\ell}
=\pi^2(k^2+\ell^2)+\mu_0
=-\biggl(\frac{\kappa^2+\eta^2}{4}-\mu_0\biggr)+\frac{\kappa^2+\eta^2}{4\cdot1}
+\pi^2(k^2+\ell^2)\cdot1,
\end{equation*}
and it holds that 
$\langle \cdot,\cdot\rangle_{\zeta}=\langle \cdot,\cdot\rangle_\theta$,
$H_{\zeta}=H_\theta$ and 
$A_{\zeta} e_{k,\ell}=\mu_{k,\ell} e_{k,\ell}$, where
\begin{equation*}
-A_\zeta
=
\biggl(\frac{\partial^2}{\partial y^2}+\frac{\partial^2}{\partial z^2}\biggr)
+\kappa\frac{\partial}{\partial y}
+\eta\frac{\partial}{\partial z}
+\biggl(\frac{\kappa^2+\eta^2}{4}-\mu_0\biggr).
\end{equation*}
Hence, the $Q_2$-Wiener process can be constructed in the same way. 
Indeed, by choosing $Q_2$ as the covariance operator 
on $\mathscr D(A_{\zeta}^{-1/2})\supset H_{\zeta} (=H_\theta)$ that satisfies 
\begin{equation*}
Q_2 e_{k,\ell}
=\mu_{k,\ell}^{-\alpha}\|e_{k,\ell}\|_{\zeta,-1/2}^2 e_{k,\ell},
\end{equation*}
where $\alpha>0$, we can see that
the $Q_2$-Wiener process $\{W_t^{Q_2}\}_{t\ge0}$ 
is well-defined in $\mathscr D(A_{\zeta}^{-1/2})$ and 
\begin{equation*}
\langle W^{Q_2}_t,f\rangle_\theta
=\langle W^{Q_2}_t,f\rangle_{\zeta}
=\sum_{k,\ell\ge1}
\mu_{k,\ell}^{-\alpha/2}
\langle f,e_{k,\ell}\rangle_{\zeta}
w_{k,\ell}(t)
=\sum_{k,\ell\ge1}
\mu_{k,\ell}^{-\alpha/2}
\langle f,e_{k,\ell}\rangle_\theta
w_{k,\ell}(t)
\end{equation*}
for $f\in H_\theta$ and $t\ge0$.

In this paper, we consider only two types of $Q$-Wiener processes 
given by \eqref{QWp_ver1} or \eqref{QWp_ver2}. 
However, the same argument can be developed by choosing 
a $Q$-Wiener process with a damping factor 
such that the random field $X_t^Q(y,z)$ is spectrally decomposable 
as in \eqref{sd_ver1}.
\end{rmk}

We assume the following conditions 
on the initial value $\xi\in H_\theta$ of the SPDE \eqref{2d_spde}.
\begin{asm}\label{asm}
The initial value $\xi$ satisfies either (i) or (ii), and both (iii) and (iv).
\begin{enumerate}
\item[(i)]
$\EE[\langle \xi,e_{k,\ell}\rangle_\theta]=0$ for all $k,\ell\ge1$
and $\sup_{k,\ell\ge1}\lambda_{k,\ell}^{1+\alpha}
\EE[\langle \xi,e_k\rangle_\theta^2]<\infty$. 

\item[(ii)]
$\EE[\|A_\theta^{(1+\alpha)/2}\xi\|_\theta^2]<\infty$.

\item[(iii)]
$\sup_{k,\ell\ge1}\lambda_{k,\ell}^{1+\alpha}
\EE[\langle \xi,e_{k,\ell}\rangle_\theta^4]<\infty$. 

\item[(iv)]
$\{\langle \xi,e_{k,\ell}\rangle_\theta\}_{k,\ell\ge1}$ are independent.

\end{enumerate}
\end{asm}

\begin{rmk}
From $\|A_\theta^{(1+\alpha)/2}\xi\|_\theta^2
=\sum_{k,\ell\ge1}\lambda_{k,\ell}^{1+\alpha}
\langle\xi,e_{k,\ell}\rangle_\theta^2$, 
(ii) in Assumption \ref{asm} can be replaced by
\begin{equation*}
\sum_{k,\ell\ge1}\lambda_{k,\ell}^{1+\alpha}
\EE[\langle\xi,e_{k,\ell}\rangle_\theta^2]<\infty.
\end{equation*}
Moreover, for a non-random function $\xi(y,z)=\xi_1(y)\xi_2(z)\in H_\theta$, 
if $\xi_1,\xi_2\in C^2([0,1])$ for $0<\alpha<1/2$, 
or if $\xi_1,\xi_2\in C^3([0,1])$ for $1/2\le\alpha<1$,
then $\xi$ satisfies (ii) and (iii). 
Indeed, noting that 
\begin{equation*}
\langle \xi,e_{k,\ell}\rangle_\theta
=2\int_0^1 \xi_1(y)\sin(\pi k y)\ee^{\frac{\theta_1}{2\theta_2}y}\dd y
\int_0^1 \xi_2(z)\sin(\pi \ell z)\ee^{\frac{\eta_1}{2\theta_2}z}\dd z
\end{equation*}
and that for $\xi_1\in C^p([0,1])$ ($p=2,3$) such that $\xi_1(0)=\xi_1(1)=0$,
\begin{equation*}
\Biggl|
\int_0^1 \xi_1(y)\sin(\pi k y)\ee^{\frac{\theta_1}{2\theta_2}y}\dd y
\Biggr|
\le \frac{C_1}{k^p},
\end{equation*}
we obtain that $|\langle \xi,e_{k,\ell}\rangle_\theta|\le C_2/k^p\ell^p$. 
Since
$\lambda_{k,\ell}\le C_3(k^2+\ell^2)$ and 
\begin{equation*}
\sum_{k,\ell\ge1}\lambda_{k,\ell}^{1+\alpha}\langle \xi,e_{k,\ell}\rangle_\theta^2
\le
C\sum_{k,\ell\ge1}\frac{(k^2+\ell^2)^{1+\alpha}}{k^{2p}\ell^{2p}}
\le
C'\sum_{k\ge1}\frac{1}{k^{2(p-1)-2\alpha}}
\sum_{\ell\ge1}\frac{1}{\ell^{2p}},
\end{equation*}
$\sum_{k,\ell\ge1}\lambda_{k,\ell}^{1+\alpha}\langle \xi,e_{k,\ell}\rangle_\theta^2$
converges if $p>\alpha+3/2$, and therefore (ii) is satisfied 
by setting $p=2$ for $0<\alpha<1/2$ and $p=3$ for $1/2\le\alpha<1$.
Similarly, $\xi$ also satisfies (iii).
\end{rmk}

By setting the $Q_1$-Wiener process $W_t^{Q_1}$ by \eqref{QWp_ver1},
there exists a unique mild solution $X_t^{Q_1}$ of the SPDE \eqref{2d_spde}
which satisfies $\sup_{0\le t\le1}\EE[\|X_t^{Q_1}\|_\theta^2]<\infty$ 
under $\lambda_{1,1}^*>0$ and Assumption \ref{asm}.
Indeed, since  
$\sup_{k,\ell\ge1}\lambda_{k,\ell}^{1+\alpha}
\EE[\langle \xi,e_k\rangle_\theta^2]<\infty$ under (i) or (ii) in Assumption \ref{asm}
and that there exists a constant $C_1>0$ such that 
$\lambda_{k,\ell}\ge C_1(k^2+\ell^2)$ for $k,\ell\ge1$,
it holds that
\begin{align*}
\EE[\|X_t^{Q_1}\|_\theta^2]
=\sum_{k,\ell\ge1}\EE[x_{k,\ell}^{Q_1}(t)^2]
&=
\sum_{k,\ell\ge1}
\ee^{-2\lambda_{k,\ell}t}
\EE[\langle \xi,e_{k,\ell}\rangle_\theta^2]
+\sigma^2\sum_{k,\ell\ge1}
\frac{1-\ee^{-2\lambda_{k,\ell}t}}{2\lambda_{k,\ell}^{1+\alpha}}
\\
&\le 
\sum_{k,\ell\ge1}
\frac{1}{\lambda_{k,\ell}^{1+\alpha}}
\lambda_{k,\ell}^{1+\alpha}\EE[\langle \xi,e_{k,\ell}\rangle_\theta^2]
+\sigma^2\sum_{k,\ell\ge1}
\frac{1}{2\lambda_{k,\ell}^{1+\alpha}}
\\
&\le
C\sum_{k,\ell\ge1}
\frac{1}{\lambda_{k,\ell}^{1+\alpha}}
<\infty
\end{align*}
for $\alpha>0$, and the mild solution $X_t^{Q_1}$ satisfies 
$\sup_{0\le t\le1}\EE[\|X_t^{Q_1}\|_\theta^2]<\infty$.
The same is true for $X_t^{Q_2}$.

\section{Main results}\label{sec3}

\subsection{SPDE driven by $Q_1$-Wiener process}\label{sec3.1}
In this subsection, we deal with the SPDE \eqref{2d_spde}
driven by the $Q_1$-Wiener process defined as \eqref{QWp_ver1}.
We first consider the estimation for the parameters 
which appear in the coordinate process \eqref{cp2_ver1} 
based on the thinned  data with respect to space.
Set $\overline m_1\le M_1$ and $\overline m_2\le M_2$
such that $\overline m:=\overline m_1\overline m_2=O(N^\rho)$
for some $0<\rho<1\land2(1-\alpha)$, and let
\begin{align*}
\overline y_{j_1}
&=\biggl\lfloor\frac{M_1}{\overline m_1}\biggr\rfloor\frac{j_1}{M_1},
\quad j_1=0,\ldots,\overline m_1,
\\
\overline z_{j_2}
&=\biggl\lfloor\frac{M_2}{\overline m_2}\biggr\rfloor\frac{j_2}{M_2},
\quad j_2=0,\ldots,\overline m_2.
\end{align*}
For $\delta\in(0,1/2)$, 
there exist $J_1,J_2\ge1$, $m_1,m_2\ge1$ such that 
\begin{align*}
&\overline y_{J_1}<\delta \le \overline y_{J_1+1} < \cdots 
< \overline y_{J_1+m_1}\le 1-\delta < \overline y_{J_1+m_1+1},
\\
&\overline z_{J_2}<\delta \le \overline z_{J_2+1} < \cdots 
< \overline z_{J_2+m_2}\le 1-\delta < \overline z_{J_2+m_2+1},
\end{align*}
and let 
\begin{align*}
\widetilde y_{j_1}
&=\overline y_{J_1+j_1}
=\biggl\lfloor\frac{M_1}{\overline m_1}\biggr\rfloor\frac{J_1+j_1}{M_1},
\quad j_1=1,\ldots, m_1,
\\
\widetilde z_{j_2}
&=\overline z_{J_2+j_2}
=\biggl\lfloor\frac{M_2}{\overline m_2}\biggr\rfloor\frac{J_2+j_2}{M_2},
\quad j_2=1,\ldots, m_2
\end{align*}
and $D_\delta=[\delta,1-\delta]^2\subset D$.
Note that $m:=m_1m_2=O(N^\rho)$ and 
$(\widetilde y_{j_1},\widetilde z_{j_2})\in D_\delta$ 
for any $j_1=1,\ldots, m_1$ and $j_2=1,\ldots, m_2$.

We write $\Delta_i X^{Q}(y,z)=X_{t_i}^{Q}(y,z)-X_{t_{i-1}}^{Q}(y,z)$.
Let $\Gamma(s)=\int_0^\infty x^{s-1}\ee^{-x}\dd x$, $s>0$.

\begin{prop}[]\label{prop1}
Under Assumption \ref{asm}, 
it holds that uniformly in $(y,z)\in D_\delta$,
\begin{equation}\label{prop1-eq1}
\EE[(\Delta_i X^{Q_1})^2(y,z)]
=\Delta_N^\alpha\frac{\sigma^2\Gamma(1-\alpha)}{4\pi\alpha\theta_2}
\ee^{-\frac{\theta_1}{\theta_2}y}\ee^{-\frac{\eta_1}{\theta_2}z}
+r_{N,i}+O(\Delta_N),
\end{equation}
where $\sum_{i=1}^N |r_{N,i}|=O(\Delta_N^{\alpha})$,
and thus
\begin{equation*}
\EE\biggl[
\frac{1}{N\Delta_N^\alpha}\sum_{i=1}^N(\Delta_i X^{Q_1})^2(y,z)
\biggr]
=\frac{\sigma^2\Gamma(1-\alpha)}{4\pi\alpha\theta_2}
\ee^{-\frac{\theta_1}{\theta_2}y}\ee^{-\frac{\eta_1}{\theta_2}z}
+O(\Delta_N^{1-\alpha}).
\end{equation*}
\end{prop}

\begin{rmk}\label{rmk4}
It follows from the proofs of Lemmas \ref{lem3}, \ref{lem4} 
and Proposition \ref{prop1} that
\begin{align}
&\EE[(\Delta_i X^{Q_1})^2(y,z)]
\nonumber
\\
&=\sigma^2
\sum_{k,\ell\ge1}
\frac{1-\ee^{-\lambda_{k,\ell}\Delta_N}}{\lambda_{k,\ell}^{1+\alpha}}
\biggl(
1-\frac{1-\ee^{-\lambda_{k,\ell}\Delta_N}}{2}
\ee^{-2\lambda_{k,\ell}(i-1)\Delta_N }
\biggr)
e_{k,\ell}^2(y,z)
+r_{N,i},
\label{rmk3-eq1}
\end{align}
where $r_{N,i}$ is the sequence in \eqref{prop1-eq1}.
According to Lemmas \ref{lem2} and \ref{lem4}, 
the restriction of $\alpha$ allows us to 
approximate the summation in \eqref{rmk3-eq1} containing $\lambda_{k,\ell}$ 
with unknown parameters by an explicit expression 
such as the main part in \eqref{prop1-eq1}. 
This makes it possible to estimate 
$\sigma^2/\theta_2$, $\theta_1/\theta_2$ and $\eta_1/\theta_2$.
\end{rmk}

Let $s=\sigma^2/\theta_2$, $\kappa=\theta_1/\theta_2$, $\eta=\eta_1/\theta_2$ and
\begin{equation*}
Z_{N}^{Q}(y,z)
=\frac{1}{N\Delta_N^\alpha}\sum_{i=1}^N(\Delta_i X^Q)^2(y,z).
\end{equation*}
By setting 
 the contrast function
\begin{equation*}
U_{N,m}^{(1)}(s,\kappa,\eta)
=\sum_{j_1=1}^{m_1}\sum_{j_2=1}^{m_2}
\biggl(Z_{N}^{Q_1}(\widetilde y_{j_1},\widetilde z_{j_2})
-\frac{\Gamma(1-\alpha)}{4\pi\alpha}s\,
\ee^{-(\kappa \widetilde y_{j_1}+\eta \widetilde z_{j_2})}
\biggr)^2,
\end{equation*}
the minimum contrast estimators of $s$, $\kappa$ and $\eta$ are defined as 
\begin{equation*}
(\hat s,\hat \kappa, \hat \eta)
=\underset{(s,\kappa,\eta)\in\Xi_1}{\mathrm{arginf}}\, U_{N,m}^{(1)}(s,\kappa,\eta),
\end{equation*}
where $\Xi_1$ is a compact convex subset of
$(0,\infty)\times\mathbb R^2$.
We suppose that the true value $(s^*,\kappa^*,\eta^*)$ 
belongs to $\mathrm{Int}\, \Xi_1$.

\begin{thm}[]\label{th1}
Under Assumption \ref{asm}, 
it holds that for any $\gamma<\frac{1}{2}\land(1-\alpha)-\frac{\rho}{2}$,
as $N\to\infty$ and $m\to\infty$,
\begin{equation*}
m^{1/2}N^\gamma
\begin{pmatrix}
\hat s-s^*
\\
\hat \kappa-\kappa^*
\\
\hat \eta-\eta^*
\end{pmatrix}
\pto 0.
\end{equation*}
\end{thm}

\begin{rmk}
Bibinger and Trabs \cite{Bibinger_Trabs2020}, 
which dealt with the parameter estimation for linear parabolic SPDEs 
in one space dimension, showed that 
the estimators have asymptotic normality.
On the other hand, in two space dimensions, 
the evaluation of the remainder in \eqref{prop1-eq1} 
is worse than that in one space dimension 
because of the increase in dimension, 
and the asymptotic normality can not be derived. 
However, it is possible to estimate each coefficient parameter 
even if the estimators do not have asymptotic normality.
Theorem \ref{th1} shows that in two space dimensions, 
the estimators have $m^{1/2}N^\gamma$-consistency
and $m^{1/2}N^\gamma=o(N^{\frac{1}{2}\land(1-\alpha)})$.
In other words, the assertion in Theorem \ref{th1} can be regarded as 
\begin{equation*}
N^{\gamma'}
\begin{pmatrix}
\hat s-s^*
\\
\hat \kappa-\kappa^*
\\
\hat \eta-\eta^*
\end{pmatrix}
\pto 0
\end{equation*}
for any $\gamma'<\frac{1}{2}\land(1-\alpha)$.
\end{rmk}

Once we estimate $(s,\kappa, \eta)$, we can construct an approximation for 
the coordinate process 
\begin{equation*}
x_{k,\ell}^{Q_1}(t)
=\langle X_t^{Q_1},e_{k,\ell}\rangle_\theta
=2\int_0^1\int_0^1
X_t^{Q_1}(y,z)\sin(\pi k y)\sin(\pi \ell z)
\ee^{\frac{\kappa}{2}y}\ee^{\frac{\eta}{2}z}
\dd y\dd z
\end{equation*}
by a Riemann sum. Furthermore, 
noting that the coordinate process $x_{k,\ell}^{Q_1}(t)$ satisfies 
\begin{equation*}
\dd x_{k,\ell}^{Q_1}(t)
=-\lambda_{k,\ell} x_{k,\ell}^{Q_1}(t)\dd t
+\sigma\lambda_{k,\ell}^{-\alpha/2}\dd w_{k,\ell}(t),
\quad
x_{k,\ell}^{Q_1}(0)=\langle \xi,e_{k,\ell}\rangle_\theta,
\end{equation*}
we can estimate the volatility parameter
$\sigma_{k,\ell}:=\sigma\lambda_{k,\ell}^{-\alpha/2}$
by using statistical inference for diffusion processes. Since
\begin{equation*}
\lambda_{1,2}
=\biggl(\frac{\sigma_{1,1}^2}{\sigma_{1,2}^2}\biggr)^{1/\alpha}\lambda_{1,1},
\quad
\lambda_{1,1}
=\biggl(\frac{\sigma^2}{\sigma_{1,1}^2}\biggr)^{1/\alpha}
=\biggl(\frac{s\theta_2}{\sigma_{1,1}^2}\biggr)^{1/\alpha}
\end{equation*}
and $\lambda_{1,2}-\lambda_{1,1}=3\pi^2\theta_2$, 
$\theta_2$ can be expressed by using $s$, $\sigma_{1,1}$ and $\sigma_{1,2}$ as follows.
\begin{equation*}
\theta_2
=
\Biggl\{
\frac{3\pi^2}{s^{1/\alpha}}
\Biggl(
\frac{1}{\sigma_{1,2}^{2/\alpha}}
-\frac{1}{\sigma_{1,1}^{2/\alpha}}
\Biggr)^{-1}
\Biggr\}^{\frac{\alpha}{1-\alpha}}.
\end{equation*}
$\sigma^2$, $\theta_1$, $\eta_1$ and $\theta_0$
can also be expressed by using  
$\theta_2$, $s$, $\kappa$, $\eta$ and $\lambda_{1,1}$ as follows.
\begin{equation*}
\sigma^2=s\theta_2,
\quad 
\theta_1=\kappa\theta_2, 
\quad
\eta_1=\eta\theta_2,
\end{equation*}
\begin{equation*}
\theta_0
=-\lambda_{1,1}
+\biggl(
\frac{\kappa^2+\eta^2}{4}+2\pi^2
\biggr)\theta_2.
\end{equation*}

With the above in mind, 
we construct an approximate coordinate process 
by using the thinned data with respect to time,  
and consider the estimation for each coefficient parameter.
Let $n\le N$ and
\begin{equation*}
\widetilde t_i=
\biggl\lfloor\frac{N}{n}\biggr\rfloor
\frac{i}{N},
\quad i=0,\ldots, n.
\end{equation*}
As an approximation of $x_{k,\ell}^{Q_1}(t)$, we consider
\begin{equation*}
\hat x_{k,\ell}^{Q_1}(\widetilde t_i)
=\frac{2}{M}\sum_{j_1=1}^{M_1}\sum_{j_2=1}^{M_2}
X_{\widetilde t_i}^{Q_1}(y_{j_1},z_{j_2})
\sin(\pi k y_{j_1})\sin(\pi \ell z_{j_2})
\ee^{\frac{\hat\kappa}{2}y_{j_1}}\ee^{\frac{\hat\eta}{2}z_{j_2}},
\quad i=1,\ldots,n.
\end{equation*}
By using the thinned data based on the approximate coordinate process
$\{\hat x_{k,\ell}^{Q_1}(\widetilde t_i)\}_{i=1}^n$,
the estimator of $\sigma_{k,\ell}^2$ is defined as
\begin{equation*}
\hat\sigma_{k,\ell}^2
=\sum_{i=1}^{n} \bigl(\hat x_{k,\ell}^{Q_1}(\widetilde t_i)
-\hat x_{k,\ell}^{Q_1}(\widetilde t_{i-1})\bigr)^2.
\end{equation*}
Moreover, the estimators of $\theta_2$, $\sigma^2$, $\theta_1$, $\eta_1$ and $\theta_0$
are defined as
\begin{align*}
\hat\theta_2
&=
\Biggl\{
\frac{3\pi^2}{\hat s^{1/\alpha}}
\Biggl(
\frac{1}{(\hat \sigma_{1,2}^2)^{1/\alpha}}
-\frac{1}{(\hat \sigma_{1,1}^2)^{1/\alpha}}
\Biggr)^{-1}
\Biggr\}^{\frac{\alpha}{1-\alpha}},
\\
\hat \sigma^2 &= \hat s \hat \theta_2,
\quad
\hat \theta_1 = \hat \kappa \hat \theta_2,
\quad
\hat \eta_1 = \hat \eta \hat \theta_2,
\\
\hat \theta_0
&=-
\hat \lambda_{1,1}
+\biggl(
\frac{\hat\kappa^2+\hat\eta^2}{4}+2\pi^2
\biggr)\hat \theta_2,
\quad
\hat \lambda_{1,1}
=
\biggl(\frac{\hat s \hat \theta_2}{\hat\sigma_{1,1}^2}\biggr)^{1/\alpha}.
\end{align*}
Let
$\bs \vartheta_{-1}^*=(\theta_1^*,\eta_1^*,\theta_2^*,(\sigma^*)^2)^\TT$,
\begin{align*}
c_1
&=
(\lambda_{1,1}^*)^2
\biggl(
\frac{\lambda_{1,2}^*}{\alpha}-\frac{\theta_0^*}{1-\alpha}
\biggr)^2
+(\lambda_{1,2}^*)^2
\biggl(
\frac{\lambda_{1,1}^*}{\alpha}-\frac{\theta_0^*}{1-\alpha}
\biggr)^2,
\\
c_2
&=
\frac{-1}{1-\alpha}
\biggl\{
(\lambda_{1,1}^*)^2
\biggl(
\frac{\lambda_{1,2}^*}{\alpha}-\frac{\theta_0^*}{1-\alpha}
\biggr)
+(\lambda_{1,2}^*)^2
\biggl(
\frac{\lambda_{1,1}^*}{\alpha}-\frac{\theta_0^*}{1-\alpha}
\biggr)
\biggr\},
\\
c_3
&=
\frac{1}{(1-\alpha)^2}
\bigl\{
(\lambda_{1,1}^*)^2+(\lambda_{1,2}^*)^2
\bigr\}
\\
J
&=
\frac{2}{9\pi^4(\theta_2^*)^2}
\begin{pmatrix}
c_1 & c_2(\bs\vartheta_{-1}^*)^\TT
\\
c_2\bs\vartheta_{-1}^* & c_3\bs\vartheta_{-1}^*(\bs\vartheta_{-1}^*)^\TT
\end{pmatrix},
\end{align*}
where $\TT$ denotes the transpose.

\begin{thm}\label{th2}
Suppose Assumption \ref{asm} holds.
Let $\gamma<\frac{1}{2}\land(1-\alpha)-\frac{\rho}{2}$.
\begin{enumerate}
\item[(1)]
Under $\frac{n^{1-\alpha}}{mN^{2\gamma}}\to0$ and 
$\frac{n^{1-\alpha+\epsilon}}{M_1^{2\epsilon} \land M_2^{2\epsilon}}\to0$
for some $\epsilon<\alpha$, it holds that as $n\to\infty$,
\begin{equation*}
(\hat\theta_0, \hat\theta_1,\hat\eta_1,\hat\theta_2,\hat\sigma^2)
\pto (\theta_0^*,\theta_1^*,\eta_1^*,\theta_2^*,(\sigma^*)^2).
\end{equation*}

\item[(2)]
Under $\frac{n^{2-\alpha}}{mN^{2\gamma}}\to0$ and 
$\frac{n^{2-\alpha+\epsilon}}{M_1^{2\epsilon} \land M_2^{2\epsilon}}\to0$
for some $\epsilon<\alpha$, it holds that as $n\to\infty$,
\begin{equation*}
\sqrt n
\begin{pmatrix}
\hat \theta_0-\theta_0^*
\\
\hat \theta_1-\theta_1^*
\\
\hat \eta_1-\eta_1^*
\\
\hat \theta_2-\theta_2^*
\\
\hat \sigma^2-(\sigma^*)^2
\end{pmatrix}
\dto N(0,J).
\end{equation*}
\end{enumerate}
\end{thm}

\begin{rmk}
Set $M_1=M_2$, $M=O(N^\beta)$, $\beta>0$, 
$n=O(N^{\zeta})$, $0<\zeta\le1$, 
and $m=O(N^\rho)$, $0<\rho<1\land2(1-\alpha)$. 
For the condition of (1) in Theorem \ref{th2}, one has from
$\frac{n^{1-\alpha}}{mN^{2\gamma}}
=O(N^{\zeta(1-\alpha)-\rho-2\gamma})$,
$\frac{n^{1-\alpha+\epsilon}}{M_1^{2\epsilon}\land M_2^{2\epsilon}}
=O(N^{\zeta(1-\alpha+\epsilon)-\epsilon\beta})$
that
\begin{equation}\label{eq-zeta1}
\zeta(1-\alpha)-\rho-2\gamma<0,
\quad
\zeta(1-\alpha+\epsilon)-\epsilon\beta<0
\end{equation}
for some $\epsilon <\alpha$.
For the existence of $\epsilon<\alpha$ such that 
$\zeta(1-\alpha+\epsilon)-\epsilon\beta<0$, 
it suffices that 
$\zeta(1-\alpha)/(\beta-\zeta)<\epsilon<\alpha$,
i.e., $\zeta<\alpha\beta$.
Therefore, \eqref{eq-zeta1} is fulfilled if $\zeta$ is chosen as follows.
\begin{equation}\label{eq-zeta2}
\zeta<\frac{\rho+2\gamma}{1-\alpha} \land \alpha\beta.
\end{equation}
Since we can choose $\gamma$ such that $1-\alpha<\rho+2\gamma$, 
\eqref{eq-zeta2} is represented by $\zeta<\alpha\beta$.
In the same way as (2) in Theorem \ref{th2}, 
we can choose $0<\zeta\le1$ such that
\begin{equation}\label{eq-zeta3}
\zeta<\frac{\rho+2\gamma}{2-\alpha} \land \frac{\alpha\beta}{2},
\end{equation}
and since $\gamma$ can be chosen to satisfy $\rho+2\gamma<1\land2(1-\alpha)$, 
\eqref{eq-zeta3} is represented by 
$\zeta<\frac{1\land2(1-\alpha)}{2-\alpha} \land \frac{\alpha\beta}{2}$.
We also see from \eqref{eq-zeta2} and \eqref{eq-zeta3} that
for some $\zeta<\frac{\rho+2\gamma}{2-\alpha}$, 
we need to set $\beta>\zeta/\alpha$ for consistency, 
and $\beta>2\zeta/\alpha$ for asymptotic normality.
That is, in order to obtain asymptotic normality of the estimators, 
it is necessary to increase the number of observations in space $M$ compared to 
the number of observations in space required for consistency to hold.
\end{rmk}

\subsection{SPDE driven by $Q_2$-Wiener process}
In this subsection, we consider the SPDE \eqref{2d_spde}
driven by the $Q_2$-Wiener process defined as \eqref{QWp_ver2}.
In this case, as in Subsection \ref{sec3.1}, 
we first estimate the parameters $\theta_1/\theta_2$ and $\eta_1/\theta_2$
which appear in the coordinate process, 
and then estimate the coefficient parameters 
using the approximate coordinate process.

In a similar way to Proposition \ref{prop1}, the following proposition holds.

\begin{prop}\label{prop2}
Under Assumption \ref{asm}, 
it holds that uniformly in $(y,z)\in D_\delta$,
\begin{equation}\label{prop2-eq1}
\EE[(\Delta_i X^{Q_2})^2(y,z)]
=
\Delta_N^\alpha\frac{\sigma^2\Gamma(1-\alpha)}{4\pi\alpha\theta_2^{1-\alpha}}
\ee^{-\frac{\theta_1}{\theta_2}y}\ee^{-\frac{\eta_1}{\theta_2}z}
+r_{N,i}+O(\Delta_N),
\end{equation}
where $\sum_{i=1}^N |r_{N,i}|=O(\Delta_N^{\alpha})$,
and thus
\begin{align*}
\EE\biggl[
\frac{1}{N\Delta_N^\alpha}\sum_{i=1}^N(\Delta_i X^{Q_2})^2(y,z)
\biggr]
&=
\frac{\sigma^2\Gamma(1-\alpha)}{4\pi\alpha\theta_2^{1-\alpha}}
\ee^{-\frac{\theta_1}{\theta_2}y}\ee^{-\frac{\eta_1}{\theta_2}z}
+O(\Delta_N^{1-\alpha}).
\end{align*}
\end{prop}

\begin{rmk}
The only difference between \eqref{prop1-eq1} and \eqref{prop2-eq1} is 
the exponent of $\theta_2$ in the denominator. 
This is caused by the fact that the coefficients of $k^2+\ell^2$ 
in $\lambda_{k,\ell}$ and $\mu_{k,\ell}$ are different, 
which are $\theta_2\pi^2$ and $\pi^2$ respectively.
See Lemmas \ref{lem1} and \ref{lem5} for details.
Proposition \ref{prop2} allows us to estimate 
$\sigma^2/\theta_2^{1-\alpha}$, $\theta_1/\theta_2$ and $\eta_1/\theta_2$
when the SPDE \eqref{2d_spde} is driven by the $Q_2$-Wiener process.
\end{rmk}
Let $S=\sigma^2/\theta_2^{1-\alpha}$, $\kappa=\theta_1/\theta_2$, 
and $\eta=\eta_1/\theta_2$, and let 
\begin{equation*}
U_{N,m}^{(2)}(S,\kappa,\eta)
=\sum_{j_1=1}^{m_1}\sum_{j_2=1}^{m_2}
\biggl(Z_N^{Q_2}(\widetilde y_{j_1},\widetilde z_{j_2})
-\frac{\Gamma(1-\alpha)}{4\pi\alpha}S\,
\ee^{-(\kappa \widetilde y_{j_1}+\eta \widetilde z_{j_2})}
\biggr)^2,
\end{equation*}
which is the contrast function of $S$, $\kappa$ and $\eta$.
Let $\check S$, $\check \kappa$ and $\check \eta$ be
minimum contrast estimators defined as
\begin{equation*}
(\check S,\check \kappa, \check \eta)
=\underset{(S,\kappa,\eta)\in\Xi_2}{\mathrm{arginf}}\, U_{N,m}^{(2)}(S,\kappa,\eta),
\end{equation*}
where $\Xi_2$ is a compact convex subset of
$(0,\infty)\times\mathbb R^2$, and  we assume that 
the true value $(S^*,\kappa^*,\eta^*)$ belongs to $\mathrm{Int}\, \Xi_2$.

\begin{thm}\label{th3}
Under Assumption \ref{asm}, 
it holds that for any $\gamma<\frac{1}{2}\land(1-\alpha)-\frac{\rho}{2}$,
as $N\to\infty$ and $m\to\infty$,
\begin{equation*}
m^{1/2}N^\gamma
\begin{pmatrix}
\check S-S^*
\\
\check \kappa-\kappa^*
\\
\check \eta-\eta^*
\end{pmatrix}
\pto0.
\end{equation*}
\end{thm}

We construct the following approximate coordinate process 
by using the estimators $\check \kappa$ and $\check \eta$.
\begin{equation*}
\check x_{k,\ell}^{Q_2}(\widetilde t_i)
=\frac{2}{M}\sum_{j_1=1}^{M_1}\sum_{j_2=1}^{M_2}
X_{\widetilde t_i}^{Q_2}(y_{j_1},z_{j_2})
\sin(\pi k y_{j_1})\sin(\pi \ell z_{j_2})
\ee^{\frac{\hat\kappa}{2}y_{j_1}}\ee^{\frac{\hat\eta}{2}z_{j_2}},
\quad i=1,\ldots,n.
\end{equation*}
Noting that the coordinate process $x_{k,\ell}^{Q_2}(t)$ is 
a diffusion process given by \eqref{dp_ver2},
we can estimate the volatility parameter
$\tau_{k,\ell}:=\sigma\mu_{k,\ell}^{-\alpha/2}$.

If $\mu_0$ is known, then the estimator of $\sigma^2$ is defined as
\begin{equation*}
\check \sigma^2
=\mu_{1,1}^\alpha \sum_{i=1}^{n} \bigl(\check x_{1,1}^{Q_2}(\widetilde t_i)
-\check x_{1,1}^{Q_2}(\widetilde t_{i-1})\bigr)^2,
\end{equation*}
and the estimators of $\theta_2$, $\theta_1$ and $\eta_1$ are defined as
\begin{equation*}
\check \theta_2
=
\biggl(
\frac{\check\sigma^2}{\check S}
\biggr)^{\frac{1}{1-\alpha}},
\quad
\check \theta_1=\check \kappa \check \theta_2,
\quad
\check \eta_1=\check \eta \check \theta_2.
\end{equation*}
On the other hand, since 
\begin{equation*}
\mu_{1,1}=\biggl(\frac{\tau_{1,2}^2}{\tau_{1,1}^2}\biggr)^{1/\alpha}\mu_{1,2},
\quad
\mu_{1,2}=\mu_{1,1}+3\pi^2,
\end{equation*}
$\mu_{1,1}$ and $\sigma^2$ can be expressed as
\begin{equation*}
\mu_{1,1}
=\frac{3\pi^2(\tau_{1,2}^2/\tau_{1,1}^2)^{1/\alpha}}
{1-(\tau_{1,2}^2/\tau_{1,1}^2)^{1/\alpha}}
=\frac{3\pi^2}{\tau_{1,1}^{2/\alpha}}
\biggl(
\frac{1}{\tau_{1,2}^{2/\alpha}}-\frac{1}{\tau_{1,1}^{2/\alpha}}
\biggr)^{-1},
\end{equation*}
\begin{equation*}
\sigma^2=\tau_{1,1}^2\mu_{1,1}^\alpha
=\Biggl\{3\pi^2\biggl(
\frac{1}{\tau_{1,2}^{2/\alpha}}-\frac{1}{\tau_{1,1}^{2/\alpha}}
\biggr)^{-1}\Biggr\}^\alpha.
\end{equation*}
Therefore, if $\mu_0$ is unknown, then
the estimators of $\tau_{k,\ell}^2$ and $\sigma^2$ are defined as
\begin{equation*}
\bar \tau_{k,\ell}^2
=\sum_{i=1}^{n} \bigl(\check x_{k,\ell}^{Q_2}(\widetilde t_i)
-\check x_{k,\ell}^{Q_2}(\widetilde t_{i-1})\bigr)^2,
\end{equation*}
\begin{equation*}
\bar \sigma^2=
\Biggl\{3\pi^2\biggl(
\frac{1}{(\bar\tau_{1,2}^2)^{1/\alpha}}-\frac{1}{(\bar\tau_{1,1}^2)^{1/\alpha}}
\biggr)^{-1}\Biggr\}^\alpha
\end{equation*}
and the estimators of $\theta_2$, $\theta_1$, $\eta_1$ and $\mu_0$ are defined as
\begin{equation*}
\bar \theta_2
=\biggl(\frac{\bar\sigma^2}{\check S}\biggr)^{\frac{1}{1-\alpha}},
\quad
\bar \theta_1=\check \kappa \bar \theta_2,
\quad
\bar \eta_1=\check \eta \bar \theta_2,
\end{equation*}
\begin{equation*}
\bar \mu_0=
\bar \mu_{1,1}-2\pi^2=
\frac{3\pi^2}{(\bar \tau_{1,1}^2)^{1/\alpha}}
\biggl(
\frac{1}{(\bar \tau_{1,2}^2)^{1/\alpha}}-\frac{1}{(\bar \tau_{1,1}^2)^{1/\alpha}}
\biggr)^{-1}-2\pi^2.
\end{equation*}

Let $\bs \nu_{-1}^*
=(\theta_1^*,\eta_1^*,\theta_2^*,(1-\alpha)(\sigma^*)^2)^\TT$,
\begin{equation*}
d_1=\frac{2(\mu_{1,1}^*)^2(\mu_{1,2}^*)^2}{\alpha^2},
\quad
d_2=\frac{\mu_{1,1}^*\mu_{1,2}^*(\mu_{1,1}^*+\mu_{1,2}^*)}{\alpha(1-\alpha)},
\quad
d_3=\frac{(\mu_{1,1}^*)^2+(\mu_{1,2}^*)^2}{(1-\alpha)^2},
\end{equation*}
\begin{equation*}
K=\frac{2}{(1-\alpha)^2}
\bs \nu_{-1}^*(\bs \nu_{-1}^*)^\TT,
\quad
L=
\frac{2}{9\pi^4}
\begin{pmatrix}
d_1 & d_2 (\bs \nu_{-1}^*)^\TT
\\
d_2 \bs \nu_{-1}^* & d_3 \bs \nu_{-1}^*(\bs \nu_{-1}^*)^\TT
\end{pmatrix}.
\end{equation*}

\begin{thm}\label{th4}
Suppose Assumption \ref{asm} holds.
Let $\gamma<\frac{1}{2}\land(1-\alpha)-\frac{\rho}{2}$.

\begin{enumerate}
\item[(a)]
Let $\mu_0$ be known.
\begin{enumerate}
\item[(1)]
Under $\frac{n^{1-\alpha}}{mN^{2\gamma}}\to0$ and
$\frac{n^{1-\alpha+\epsilon}}{M_1^{2\epsilon} \land M_2^{2\epsilon}}\to0$ 
for some $\epsilon<\alpha$,
it holds that as $n\to\infty$,
\begin{equation*}
(\check\theta_1, \check\eta_1, \check\theta_2, \check\sigma^2)
\pto (\theta_1^*, \eta_1^*, \theta_2^*, (\sigma^*)^2).
\end{equation*}

\item[(2)]
Under $\frac{n^{2-\alpha}}{mN^{2\gamma}}\to0$ and 
$\frac{n^{2-\alpha+\epsilon}}{M_1^{2\epsilon} \land M_2^{2\epsilon}}\to0$
for some $\epsilon<\alpha$, it holds that as $n\to\infty$,
\begin{equation*}
\sqrt n
\begin{pmatrix}
\check\theta_1-\theta_1^*
\\
\check\eta_1-\eta_1^*
\\
\check\theta_2-\theta_2^*
\\
\check\sigma^2-(\sigma^*)^2
\end{pmatrix}
\dto N(0,K).
\end{equation*}
\end{enumerate}

\item[(b)]
Let $\mu_0$ be unknown.
\begin{enumerate}
\item[(1)]
Under $\frac{n^{1-\alpha}}{mN^{2\gamma}}\to0$ and
$\frac{n^{1-\alpha+\epsilon}}{M_1^{2\epsilon} \land M_2^{2\epsilon}}\to0$ 
for some $\epsilon<\alpha$,
it holds that as $n\to\infty$,
\begin{equation*}
(\bar\mu_0, \bar\theta_1, \bar\eta_1, \bar\theta_2, \bar\sigma^2)
\pto (\mu_0^*, \theta_1^*, \eta_1^*, \theta_2^*, (\sigma^*)^2).
\end{equation*}

\item[(2)]
Under $\frac{n^{2-\alpha}}{mN^{2\gamma}}\to0$ and 
$\frac{n^{2-\alpha+\epsilon}}{M_1^{2\epsilon} \land M_2^{2\epsilon}}\to0$
for some $\epsilon<\alpha$, it holds that as $n\to\infty$,
\begin{equation*}
\sqrt n
\begin{pmatrix}
\bar\mu_0-\mu_0^*
\\
\bar\theta_1-\theta_1^*
\\
\bar\eta_1-\eta_1^*
\\
\bar\theta_2-\theta_2^*
\\
\bar\sigma^2-(\sigma^*)^2
\end{pmatrix}
\dto N(0,L).
\end{equation*}
\end{enumerate}
\end{enumerate}
\end{thm}

\begin{rmk}
If $\mu_0$ is known, then the damping factor $\mu_{k,\ell}^{-\alpha/2}$ 
of the $Q_2$-Wiener process is known. 
Moreover,  since the cylindrical Brownian motion is given by \eqref{cBm}, 
the value corresponding to $\mu_{k,\ell}^{-\alpha/2}$ in \eqref{QWp_ver2} is $1$, 
that is, it is known.
Therefore, it is natural to consider the $Q_2$-Wiener process when $\mu_0$ is known, 
since the $Q_2$-Wiener process can be regarded as the driving noise 
corresponding to a cylindrical Brownian motion 
in the sense that the value of the damping factor is known.
Indeed, (a)-(2) in Theorem \ref{th4}, especially for $\alpha=0.5$,
corresponds to the result of Kaino and Uchida \cite{Kaino_Uchida2020}.
On the other hand, 
when $\mu_0$ is unknown, the $Q_2$-Wiener process corresponds to 
the $Q_1$-Wiener process in the sense that the damping factor is unknown.
Theorem \ref{th2} and (b) in Theorem \ref{th4} show that 
$\theta_0$ can be estimated 
for the SPDE \eqref{2d_spde} driven by the $Q_1$-Wiener process, 
while $\mu_0$ can be estimated instead of $\theta_0$
for the SPDE \eqref{2d_spde} driven by the $Q_2$-Wiener process.
\end{rmk}

\section{Simulations}\label{sec4}

The numerical solution of the SPDE \eqref{2d_spde} is generated by
\begin{eqnarray}
\tilde{X}^{Q_1}_{t_{i}}(y_{j_1},z_{j_2})
= \sum^{K}_{k = 1}\sum^{L}_{\ell = 1}x^{Q_1}_{k,\ell}(t_{i})
e_{k,\ell}(y_{j_1},z_{j_2}), 
\quad i = 1, ..., N, j_1 = 1, ..., M_1, j_2 = 1, ..., M_2.
\label{appro-SPDE}
\end{eqnarray}
For the characteristics of the parameters of the SPDE \eqref{2d_spde}, 
see the Appendix below.
In this simulation, the true values of parameters
$(\theta_0^*, \theta_1^*,\eta_1^*, \theta_2^*, \sigma^*) = (0,0.2,0.2,0.2,1)$.
We set that $N = 10^3$, $M_1 = M_2 = 200$, $K = L = 10^5$,  
$\xi = 0$,
$\alpha = 0.5$,
$\lambda_{1,1}^* \approx 4.05$.
When $N=10^3, M_1=M_2=200$, the size of data is about 10 GB. 
We used R language to compute the estimators of Theorems \ref{th1} and \ref{th2}.
The computation time of \eqref{appro-SPDE} is directly proportional to $K \times L$.
Therefore, the computation time for the numerical solution of 
the SPDE \eqref{2d_spde} is 
directly proportional to $N \times M_1 \times  M_2 \times K \times L$.
In the setting of this simulation, $N \times M_1 \times  M_2 \times K \times L = 4 \times 10^{17}$.
Three personal computers were used for this simulation,
and it takes about 100h to generate one sample path of the SPDE \eqref{2d_spde}.
The number of iteration is $25$ because the calculation time is enormous.

Figures $\ref{fig1}$-$\ref{fig3}$ are sample paths of $X_t^{Q_1}(y,z)$ 
when $(\theta_0^*, \theta_1^*,\eta_1^*, \theta_2^*, \sigma^*) = (0,0.2,0.2,0.2,1)$.
Figure \ref{fig1} is 
a cross section of 
the sample path at $t=0.1, 0.5 $ and $0.9$.
Figure \ref{fig2} is 
a cross section of 
the sample path at $y=0.1, 0.5 $ and $0.9$.
From Figure \ref{fig2}, it can be seen that there is not much difference 
between the variation of the sample path near $z=0$ and that near $z=1$.
Figure \ref{fig3} is 
a cross section of 
the sample path at $z=0.1, 0.5 $ and $0.9$.
From Figure \ref{fig3}, it can be seen that there is not much difference 
between the variation of the sample path near $y=0$ and that near $y=1$.

\begin{figure}[h] 
\begin{center}
\includegraphics[width=4cm,pagebox=cropbox,clip]{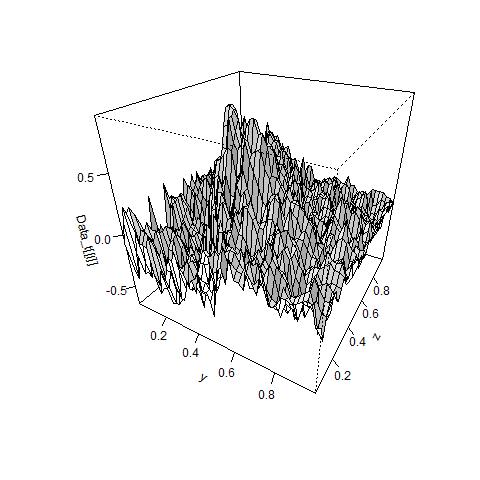}
\includegraphics[width=4cm,pagebox=cropbox,clip]{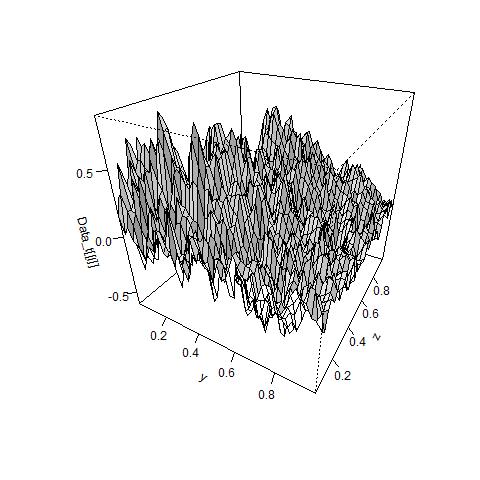}
\includegraphics[width=4cm,pagebox=cropbox,clip]{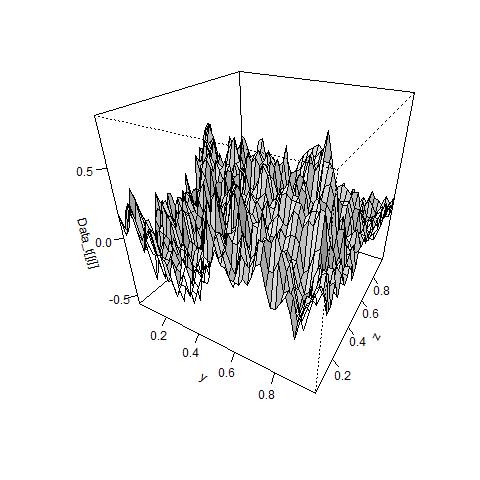}\\
\caption{Cross sections of 
the sample path at $t=0.1$(left), $0.5 $(center) and $0.9$(right)\label{fig1}}
\end{center}
\end{figure}

\begin{figure}[h] 
\begin{center}
\includegraphics[width=4cm,pagebox=cropbox,clip]{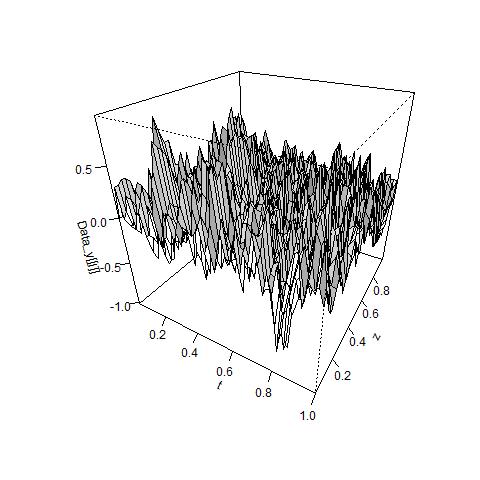}
\includegraphics[width=4cm,pagebox=cropbox,clip]{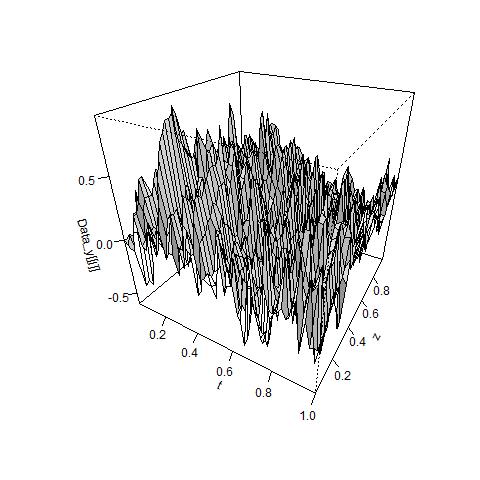}
\includegraphics[width=4cm,pagebox=cropbox,clip]{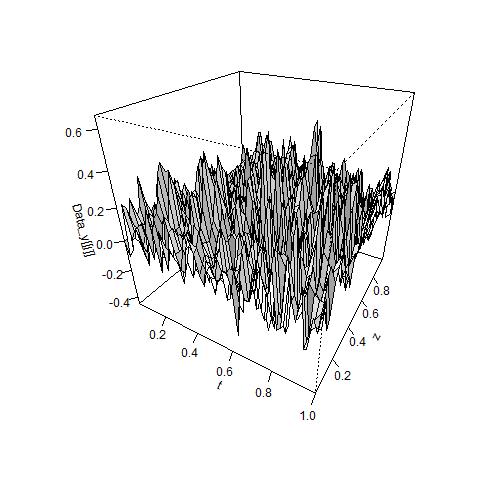}\\
\caption{Cross sections of 
the sample path at $y=0.1$(left), $0.5 $(center) and $0.9$(right) \label{fig2}}
\end{center}
\end{figure}

\begin{figure}[h] 
\begin{center}
\includegraphics[width=4cm,pagebox=cropbox,clip]{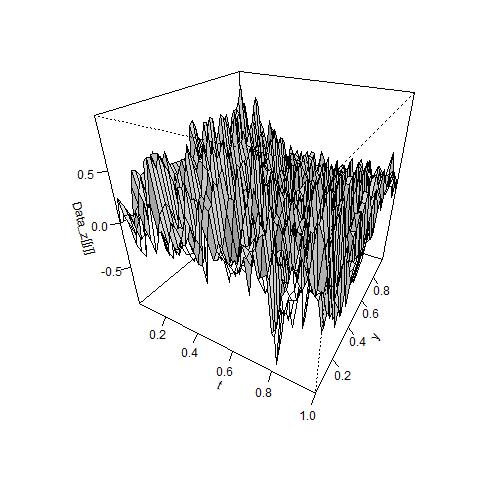}
\includegraphics[width=4cm,pagebox=cropbox,clip]{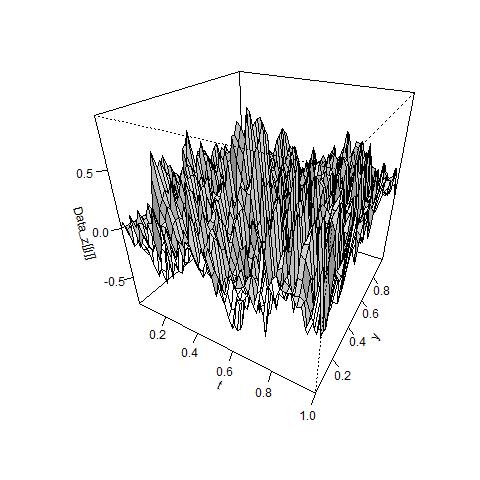}
\includegraphics[width=4cm,pagebox=cropbox,clip]{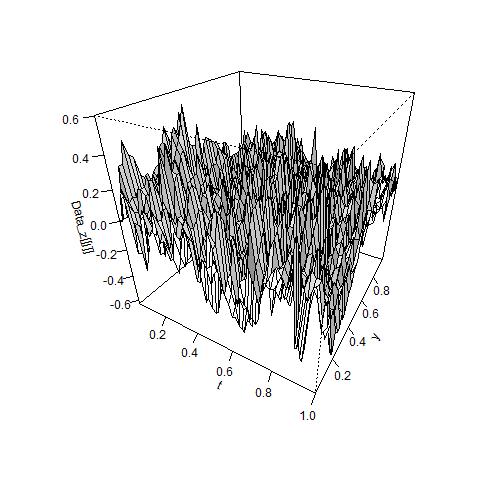}\\
\caption{Cross sections of 
the sample path at $z=0.1$(left), $0.5 $(center) and $0.9$(right) \label{fig3}}
\end{center}
\end{figure}

First, we estimated $s=\sigma^2/\theta_2$, $\kappa=\theta_1/\theta_2$ and $\eta = \eta_1/\theta_2$.
Table \ref{tab1} is the simulation results of 
{\color{black} the means and the standard deviations (s.d.s) of} $\hat{s}$, $\hat{\kappa}$  and $\hat{\eta}$ 
with $(N, m_1,m_2,\alpha,\rho, \gamma ) = (10^3, 5,5,0.5, 0.47,0.26)$.
\begin{table}[h]
\caption{Simulation results of $\hat{s}$, $\hat{\kappa}$  and $\hat{\eta}$ \label{tab1}}
\begin{center}
\begin{tabular}{c|ccc} \hline
		&$\hat{s}$&$\hat{\kappa}$&$\hat{\eta}$
\\ \hline
{\color{black} true value} &5 & 1 & 1
\\ \hline
mean &4.784&  0.987&  0.992
 \\
s.d. & (0.153)& (0.037)& (0.036)
 \\   \hline
\end{tabular}
\end{center}
\end{table}
It seems from Table \ref{tab1} that the biases of $\hat{\kappa}$ and $\hat{\eta}$ are both very small and the result of $\hat{s}$ is not bad.
The estimators of $s$, $\kappa$ and $\eta$ have good performances.

Next, we estimated $(\theta_0,\theta_1,\eta_1,\theta_2,\sigma^2)$.
Table \ref{tab2} is the simulation results of 
{\color{black} the means and the standard deviations (s.d.s) of} $\hat{\theta}_0$, $\hat{\theta}_1$, $\hat{\eta}_1$, $\hat{\theta}_2$ and $\hat{\sigma}^2$
with $(N, m_1,m_2,n,\alpha,\rho, \gamma ,\epsilon)$ $ = (10^3, 5,5,50,0.5,$  $0.47,0.26,0.499)$.
In this case, $\frac{n^{1-\alpha}}{mN^{2\gamma}} \approx 0.008 $ and 
$\frac{n^{1-\alpha+\epsilon}}{M_1^{2\epsilon} \land M_2^{2\epsilon}} \approx 0.25$.
\begin{table}[h]
\caption{Simulation results of $\hat{\theta}_0$, $\hat{\theta}_1$, $\hat{\eta}_1$, $\hat{\theta}_2$ and $\hat{\sigma}^2$ \label{tab2}}
\begin{center}
\begin{tabular}{c|ccccc} \hline
		&$\hat{\theta}_{0}$&$\hat{\theta}_{1}$&$\hat{\eta}_{1}$&$\hat{\theta}_{2}$&$\hat{\sigma}^2$
\\ \hline
{\color{black} true value} &0 & 0.2 & 0.2 &0.2 &1 
\\ \hline
mean &-0.291&  0.150&  0.152 & 0.152& 0.732
 \\
s.d. & (2.710)& (0.081)& (0.085) &(0.083) & (0.403)
 \\   \hline
\end{tabular}
\end{center}
\end{table}
It seems from Table \ref{tab2} that $\hat{\theta}_0$ has good performance.
Although $\hat{\theta}_1$, $\hat{\eta}_1$, $\hat{\theta}_2$ and $\hat{\sigma}^2$ have small biases, their results are not bad.

Table \ref{tab3} is the simulation results of 
{\color{black} the means and the standard deviations (s.d.s) of} $\hat{\theta}_0$, $\hat{\theta}_1$, $\hat{\eta}_1$, $\hat{\theta}_2$ and $\hat{\sigma}^2$
with $(N, m_1,m_2,n,\alpha,\rho, \gamma ,\epsilon) = (10^3, 5,5,100,0.5, 0.47,0.26,0.499)$.
In this case, $\frac{n^{1-\alpha}}{mN^{2\gamma}} \approx 0.011 $ and 
$\frac{n^{1-\alpha+\epsilon}}{M_1^{2\epsilon} \land M_2^{2\epsilon}} \approx 0.50$.
\begin{table}[h]
\caption{Simulation results of $\hat{\theta}_0$, $\hat{\theta}_1$, $\hat{\eta}_1$, $\hat{\theta}_2$ and $\hat{\sigma}^2$ \label{tab3}}
\begin{center}
\begin{tabular}{c|ccccc} \hline
		&$\hat{\theta}_{0}$&$\hat{\theta}_{1}$&$\hat{\eta}_{1}$&$\hat{\theta}_{2}$&$\hat{\sigma}^2$
\\ \hline
{\color{black} true value} &0 & 0.2 & 0.2 &0.2 &1 
\\ \hline
mean &-0.862&  0.190&  0.191 & 0.191& 0.904
 \\
s.d. & (3.393)& (0.080)& (0.081) &(0.081) & (0.416)
 \\   \hline
\end{tabular}
\end{center}
\end{table}
It seems from Table \ref{tab3} that 
$\hat{\theta}_1$, $\hat{\eta}_1$, $\hat{\theta_2}$ and $\hat{\sigma}^2$ 
have good performances and that $\hat{\theta}_0$ has a small bias.

Table \ref{tab4} is the simulation results of 
{\color{black} the means and the standard deviations (s.d.s) of} $\hat{\theta}_0$, $\hat{\theta}_1$, $\hat{\eta}_1$, $\hat{\theta}_2$ and $\hat{\sigma}^2$
with $(N, m_1,m_2,n,\alpha,\rho, \gamma ,\epsilon) = (10^3, 5,5,150,0.5, 0.47,0.26,0.499)$.
In this case, $\frac{n^{1-\alpha}}{mN^{2\gamma}} \approx 0.013 $ and 
$\frac{n^{1-\alpha+\epsilon}}{M_1^{2\epsilon} \land M_2^{2\epsilon}} \approx 0.75$.
\begin{table}[h]
\caption{Simulation results of $\hat{\theta}_0$, $\hat{\theta}_1$, $\hat{\eta}_1$, $\hat{\theta}_2$ and $\hat{\sigma}^2$ \label{tab4}}
\begin{center}
\begin{tabular}{c|ccccc} \hline
		&$\hat{\theta}_{0}$&$\hat{\theta}_{1}$&$\hat{\eta}_{1}$&$\hat{\theta}_{2}$&$\hat{\sigma}^2$
\\ \hline
{\color{black} true value} &0 & 0.2 & 0.2 &0.2 &1 
\\ \hline
mean &-1.235&  0.176&  0.179 & 0.179& 0.85
 \\
s.d. & (3.238)& (0.077)& (0.081) &(0.082) & (0.381)
 \\   \hline
\end{tabular}
\end{center}
\end{table}
It seems from Table \ref{tab4} that 
$\hat{\theta}_1$, $\hat{\eta}_1$, $\hat{\theta_2}$ and $\hat{\sigma}^2$ 
have good performances and that $\hat{\theta}_0$ has a small bias.

At the moment, the number of iteration is only $25$, so the results are not very good. 
However, it is expected that the estimators have good performances if the number of iteration is increased.

\section{Proofs}\label{sec5}

We set the following notation.
\begin{enumerate}
\item[1.]
Let $\mathbb R_{+}=(0,\infty)$ and $\bs k=(k_1,k_2)\in \mathbb N^2$.

\item[2.]
For $A,B\ge0$, we write $A \lesssim B$ if $A\le CB$ for some constant $C>0$.

\item[3.]
For $x=(x_1,\ldots,x_d)\in\mathbb R^d$ and $f:\mathbb R^d\to\mathbb R$, 
we write $\partial_{x_1} f(x)=\frac{\partial}{\partial x_1}f(x)$,
$\partial f(x)=(\partial_{x_1}f(x),\ldots,\partial_{x_d}f(x))$
and $\partial^2 f(x)=(\partial_{x_j}\partial_{x_i} f(x))_{i,j=1}^d$.

\item[4.]
For two functions $f$ and $g$, we write
$f(s) \sim g(s)$ ($s\to s_0$) 
if there exists $c\in(0,\infty)$ such that 
$\lim_{s\to s_0} f(s)/g(s)=c$. 

\item[5.]
Let $\ind_A$ be the indicator function of $A$. 

\item[6.]
Let $\iu$ be the imaginary unit, and
let $\re (z)$ and $\im (z)$ denote the real and imaginary parts of $z\in\mathbb C$, 
respectively.

\end{enumerate}

\subsection{Proofs of lemmas}
In this subsection, 
we prepare some lemmas before proving our assertions in Section \ref{sec3}.
For $\beta<2$, let $\mathcal F_\beta$ be the space of 
all twice continuously differentiable 
functions $f:\mathbb R_{+}\to \mathbb R$ satisfying the following conditions. 
\begin{enumerate}
\item[(a)]
$s f(s^2)$,
$s^2 f'(s^2)\ind_{[1,\infty)}$, 
$s^3 f''(s^2)\ind_{[1,\infty)}\in L^1(\mathbb R_{+})$.

\item[(b)]
$s^\beta f(s^2)\sim 1$ ($s\downarrow0$).

\end{enumerate}
For $f\in\mathcal F_\beta$, we define
\begin{equation*}
R_{1,N}
=
O\Biggl(
\Delta_N
\biggl(
\int_{\Delta_N^{1/2}}^1s|f'(s^2)|\dd s
\lor 
\int_{\Delta_N^{1/2}}^1s^3|f''(s^2)|\dd s
\biggr)
\Biggr)
\end{equation*}
\begin{equation*}
R_{2,N}
=
O\Biggl(
\Delta_N^{3/2}
\int_{\Delta_N^{1/2}}^1|f'(s^2)|\dd s
\Biggr),
\quad
R_{N}^{(\beta)}
=
O\Biggl(
\Delta_N^{1/2}
\int_{\Delta_N^{1/2}}^{1} s^{-\beta}\dd s
\Biggr).
\end{equation*}
Let
\begin{align*}
D_N^{(1)}
&=\biggl\{
(x_1,x_2)\in\mathbb R^2\biggl| 
0< x_1 \le \frac{\Delta_N^{1/2}}{2},\ 
\frac{\Delta_N^{1/2}}{2}< x_2 
\biggl\}, 
\\
D_N^{(2)}
&=\biggl\{
(x_1,x_2)\in\mathbb R^2\biggl| 
\frac{\Delta_N^{1/2}}{2}< x_1,\  
0< x_2 \le \frac{\Delta_N^{1/2}}{2}
\biggl\}.
\end{align*}

\begin{lem}\label{lem1}
If $f\in \mathcal F_\beta$, then the followings hold.
\begin{enumerate}
\item[(1)]
$\displaystyle
\Delta_N\sum_{\bs k\in\mathbb N^2}f(\lambda_{\bs k}\Delta_N)
=\frac{1}{4\pi\theta_2}
\int_0^\infty f(s)\dd s
-\int_{D_N^{(1)}\cup D_N^{(2)}}f(\theta_2\pi^2|x|^2) \dd x
+O(R_{1,N}\lor \Delta_N^{1-\beta/2})$.

\item[(2)]
For any $y,z\in[\delta,1-\delta]$, 
\begin{align*}
\Delta_N\sum_{\bs k\in\mathbb N^2}f(\lambda_{\bs k}\Delta_N)
\cos(2\pi k_1y)
&=
-\int_{D_N^{(1)}}f(\theta_2\pi^2|x|^2) \dd x
+O\biggl(\frac{R_{1,N}\lor R_{2,N}}{\delta^2}\biggr),
\\
\Delta_N\sum_{\bs k\in\mathbb N^2}f(\lambda_{\bs k}\Delta_N)
\cos(2\pi k_2z)
&=
-\int_{D_N^{(2)}}f(\theta_2\pi^2|x|^2) \dd x
+O\biggl(\frac{R_{1,N}\lor R_{2,N}}{\delta^2}\biggr).
\end{align*}

\item[(3)]
For any $y,z\in[\delta,1-\delta]$,
\begin{equation*}
\Delta_N\sum_{\bs k\in\mathbb N^2}f(\lambda_{\bs k}\Delta_N)
\cos(2\pi k_1 y)\cos(2\pi k_2 z)
=
O\biggl(\frac{R_{1,N}\lor R_{2,N}\lor\Delta_N^{1-\beta/2}}{\delta^3}\biggr).
\end{equation*}
\end{enumerate}
Furthermore, it holds that
\begin{equation}\label{eq1-lem1-1}
\int_{D_N^{(1)}\cup D_N^{(2)}}f(\theta_2\pi^2|x|^2)\dd x
=O(R_N^{(\beta)}).
\end{equation}
\end{lem}

\begin{proof}
Let $K_N=\{x\in\mathbb R_+^2| c_1\Delta_N^{1/2}< |x|\}$ for $c_1>0$.
Let $\varphi:\mathbb R_{+}\to\mathbb R$ be a function satisfying (a), 
and let $h(x)=\varphi(c_2|x|^2)$,
$c_2>0$. 
Since
\begin{align*}
\int_{\{x\in\mathbb R_+^2|c_1\Delta_N^{1/2}< |x|\le 1\}}|\varphi'(c_2|x|^2)|\dd x
&
=\int_0^{\pi/2}\dd \theta
\int_{c_1\Delta_N^{1/2}}^{1} s|\varphi'(c_2s^2)|\dd s
=O\biggl(
\int_{\Delta_N^{1/2}}^{1} s|\varphi'(s^2)|\dd s
\biggr),
\\
\int_{\{x\in\mathbb R_+^2| 1 <|x|\}}|\varphi'(c_2|x|^2)|\dd x
&=\int_0^{\pi/2}\dd \theta
\int_{1}^\infty s|\varphi'(c_2s^2)|\dd s
=O(1),
\end{align*}
we have 
\begin{align}
\int_{K_N}|\varphi'(c_2|x|^2)|\dd x
&=
O\biggl(
\int_{\Delta_N^{1/2}}^{1} s|\varphi'(s^2)|\dd s \lor 1
\biggr)
=
O\biggl(
\int_{\Delta_N^{1/2}}^{1} s|\varphi'(s^2)|\dd s
\biggr).
\label{eq1-0100}
\end{align}
Similarly, we obtain
\begin{align}
\int_{K_N}|x|^2|\varphi''(c_2|x|^2)|\dd x
=
O\biggl(
\int_{\Delta_N^{1/2}}^{1} s^3|\varphi''(s^2)|\dd s
\biggr).
\label{eq1-0101}
\end{align}
Therefore, it follows from \eqref{eq1-0100}, \eqref{eq1-0101} and  
\begin{align*}
\partial^2 h(x)
&=4c_2^2\varphi''(c_2|x|^2)
\begin{pmatrix}
x_1^2 & x_1x_2\\
x_1x_2 & x_2^2
\end{pmatrix}
+2c_2\varphi'(c_2|x|^2)
\begin{pmatrix}
1 & 0\\
0 & 1
\end{pmatrix}
\end{align*}
that
\begin{align}\label{eq1-0102}
\Delta_N\int_{K_N}| \partial^2 h(x) |\dd x
&=
O\Biggl(
\Delta_N
\biggl(
\int_{K_N}|x|^2|\varphi''(c_2|x|^2)|\dd x
+\int_{K_N}|\varphi'(c_2|x|^2)|\dd x
\biggr)
\Biggr)
=O(R_{1,N}).
\end{align}

(1) Let $g(x)=f(\theta_2\pi^2|x|^2)$, 
$a_N^{(k)}=(k+1/2)\Delta_N^{1/2}$ for $k\ge0$,
\begin{align*}
E_N^{(\bs k)}
&=
(a_N^{(k_1-1)},a_N^{(k_1)}]\times(a_N^{(k_2-1)},a_N^{(k_2)}]
\subset \mathbb R_+^2, 
\end{align*}
$E_N=\bigcup_{\bs k\in\mathbb N^2} E_N^{(\bs k)}$ 
and $F_N=\{x\in\mathbb R_{+}^2| 
(2\theta_2\pi^2 \land \lambda_{1,1})^{1/2}\Delta_N^{1/2}<|x|\}$.
By the Taylor expansion and \eqref{eq1-0102},  it follows that
\begin{align}
&
\Biggl|
\Delta_N\sum_{\bs k\in\mathbb N^2}f(\lambda_{\bs k}\Delta_N)
-\sum_{\bs k\in\mathbb N^2}
\int_{E_N^{(\bs k)}}f(\theta_2\pi^2|\bs k|^2\Delta_N)\dd x
\Biggr|
\nonumber
\\
&=
\Biggl|
\Delta_N\sum_{\bs k\in\mathbb N^2}
\biggl(
f(\theta_2\pi^2|\bs k|^2\Delta_N)
\nonumber
\\
&\qquad\qquad
+\int_0^1 f'(\theta_2\pi^2|\bs k|^2\Delta_N
+u(\lambda_{\bs k}-\theta_2\pi^2|\bs k|^2)\Delta_N
)\dd u\Bigl(\frac{\theta_1^2+\eta_1^2}{4\theta_2}-\theta_0\Bigr)\Delta_N
\biggr)
\nonumber
\\
&\qquad\qquad
-\sum_{\bs k\in\mathbb N^2}
\int_{E_N^{(\bs k)}}f(\theta_2\pi^2|\bs k|^2\Delta_N)\dd x
\Biggr|
\nonumber
\\
&\le
\Delta_N\sum_{\bs k\in\mathbb N^2}
\int_0^1 
\bigl|f'(\theta_2\pi^2|\bs k|^2\Delta_N
+u(\lambda_{\bs k}-\theta_2\pi^2|\bs k|^2)\Delta_N)\bigr|\dd u
\Bigl|\frac{\theta_1^2+\eta_1^2}{4\theta_2}-\theta_0\Bigr|\Delta_N
\nonumber
\\
&=O\biggl(\Delta_N\int_{F_N}|f'(|x|^2)|\dd x\biggr)
=O(R_{1,N})
\label{eq1-0201}
\end{align}
and that
\begin{align}
&
\Biggl|
\sum_{\bs k\in\mathbb N^2}
\int_{E_N^{(\bs k)}}f(\theta_2\pi^2|\bs k|^2\Delta_N)\dd x-
\int_{E_N} f(\theta_2\pi^2|x|^2)\dd x
\Biggr|
\nonumber
\\
&=
\Biggl|
\sum_{\bs k\in\mathbb N^2}
\int_{E_N^{(\bs k)}}\bigl(g(\bs k\Delta_N^{1/2})-g(x)\bigr)\dd x
\Biggr|
\nonumber
\\
&=
\Biggl|
\sum_{\bs k\in\mathbb N^2}
\int_{E_N^{(\bs k)}}
\biggl(
\partial g(\bs k\Delta_N^{1/2})(x-\bs k\Delta_N^{1/2})
\nonumber
\\
&\qquad\qquad
+\frac{1}{2}
(x-\bs k\Delta_N^{1/2})^\TT
\int_0^1\partial^2 g(\bs k\Delta_N^{1/2}+u(x-\bs k\Delta_N^{1/2}))\dd u
(x-\bs k\Delta_N^{1/2})
\biggr)
\dd x
\Biggr|
\nonumber
\\
&=
\frac{1}{2}
\Biggl|
\sum_{\bs k\in\mathbb N^2}
\int_{E_N^{(\bs k)}}
(x-\bs k\Delta_N^{1/2})^\TT
\int_0^1\partial^2 g(\bs k\Delta_N^{1/2}+u(x-\bs k\Delta_N^{1/2}))\dd u
(x-\bs k\Delta_N^{1/2})
\dd x
\Biggr|
\nonumber
\\
&\le
\frac{1}{2}
\sum_{\bs k\in\mathbb N^2}
\int_{E_N^{(\bs k)}}
\int_0^1\bigl|\partial^2 g(\bs k\Delta_N^{1/2}+u(x-\bs k\Delta_N^{1/2}))\bigr|\dd u
|x-\bs k\Delta_N^{1/2}|^2
\dd x
\nonumber
\\
&\lesssim
\Delta_N
\sum_{\bs k\in\mathbb N^2}
\int_{E_N^{(\bs k)}}|\partial^2 g(x)|\dd x
\nonumber
\\
&=
O\biggl(
\Delta_N
\int_{E_N} |\partial^2 g(x)| \dd x
\biggr)
=
O(R_{1,N}).
\label{eq1-0202}
\end{align}
Therefore, it follows from \eqref{eq1-0201} and \eqref{eq1-0202} that 
\begin{align}\label{eq1-0203}
\Delta_N\sum_{\bs k\in\mathbb N^2}f(\lambda_{\bs k}\Delta_N)
&=\int_{E_N} f(\theta_2\pi^2|x|^2)\dd x
+O(R_{1,N}).
\end{align}
Set $D_N^{(3)}=\{(x_1,x_2)\in\mathbb R^2|0< x_1, x_2\le \Delta_N^{1/2}/2\}$
and $D_N=\bigcup_{j=1,2,3} D_N^{(j)}$.
Since $sf(s^2)\in L^1(\mathbb R_+)$,
$\int_{\mathbb R_{+}^2} |f(\theta_2\pi^2|x|^2)|\dd x<\infty$
and 
\begin{equation*}
\int_{\mathbb R_{+}^2} f(\theta_2\pi^2|x|^2)\dd x
=\frac{1}{\pi^2\theta_2}\int_0^{\pi/2}\dd \theta\int_0^\infty s f(s^2)\dd s
=\frac{1}{4\pi\theta_2}\int_0^\infty f(s)\dd s,
\end{equation*}
\eqref{eq1-0203} can be expressed from $E_N=\mathbb R_+^2\setminus D_N$ as
\begin{equation}\label{eq1-0204}
\Delta_N\sum_{\bs k\in\mathbb N^2}f(\lambda_{\bs k}\Delta_N)
=
\frac{1}{4\pi\theta_2}\int_0^\infty f(s)\dd s
-\int_{D_N} f(\theta_2\pi^2|x|^2)\dd x
+O(R_{1,N}).
\end{equation}
Let $a_N=\Delta_N^{1/2}/\sqrt2$. 
From $s^\beta f(s^2)\sim 1$ $(s\downarrow0)$, 
we see that $s|f(s^2)|\lesssim s^{1-\beta}$ in the vicinity of $s=0$ 
and thus
\begin{align*}
\int_{D_N^{(3)}}|f(\theta_2\pi^2|x|^2)|\dd x
&\lesssim
\int_0^{a_N} s |f(\theta_2\pi^2s^2)|\dd s
\\
&=O\biggl(\int_0^{\Delta_N^{1/2}} s|f(s^2)|\dd s\biggr)
\\
&=O\biggl(\int_0^{\Delta_N^{1/2}} s^{1-\beta}\dd s\biggr)
=O(\Delta_N^{1-\beta/2}).
\end{align*} 
Therefore, from \eqref{eq1-0204} and 
\begin{align*}
&\int_{D_N}f(\theta_2\pi^2|x|^2)\dd x
=\int_{D_N^{(1)}\cup D_N^{(2)}}f(\theta_2\pi^2|x|^2)\dd x
+O(\Delta_N^{1-\beta/2}),
\end{align*}
the proof of (1) is complete.
Furthermore, it follows from $\arcsin x\le \pi x/2$ and 
$|f(s^2)|\lesssim s^{-\beta}$ that 
\begin{align*}
\int_{D_N^{(1)}}|f(\theta_2\pi^2|x|^2)|\dd x
&\le
\int_{a_N}^\infty s|f(\theta_2\pi^2s^2)|
\int_0^{\arcsin(\Delta_N^{1/2}/2s)}\dd \theta \dd s
\\
&=
\int_{a_N}^\infty s|f(\theta_2\pi^2s^2)|
\arcsin\biggl(\frac{\Delta_N^{1/2}}{2s}\biggr)\dd s
\\
&=
O\biggl(
\Delta_N^{1/2}
\int_{\Delta_N^{1/2}}^{\infty} |f(s^2)|\dd s
\biggr)
\\
&=
O\biggl(
\Delta_N^{1/2}
\int_{\Delta_N^{1/2}}^{1} s^{-\beta}\dd s
\biggr)
=
O(R_{N}^{(\beta)}),
\end{align*}
and thus we obtain \eqref{eq1-lem1-1}.


(3) In a similar way to \eqref{eq1-0201}, it holds that
\begin{align}\label{eq1-0301}
&\Delta_N\sum_{\bs k\in\mathbb N^2}f(\lambda_{\bs k}\Delta_N)
\cos(2\pi k_1 y)\cos(2\pi k_2 z)
\nonumber
\\
&=
\Delta_N\sum_{\bs k\in\mathbb N^2}g(\bs k\Delta_N^{1/2})
\cos(2\pi k_1 y)\cos(2\pi k_2 z)
+O(R_{1,N}).
\end{align}
For a real-valued function $f\in L^1(\mathbb R^2)$, 
$\Fou[f]:\mathbb R^2\to\mathbb C$ denotes its Fourier transform  
\begin{align*}
\Fou[f](x)
=\int_{\mathbb R^2} f(u)\ee^{-\iu u^\TT x}\dd u
=\iint_{\mathbb R^2} f(u_1,u_2)\ee^{-\iu(u_1x_1+u_2x_2)}\dd u_1\dd u_2.
\end{align*}
Since
\begin{align*}
&\int_{{E_N^{(\bs k)}}}
\ee^{2\pi\iu (u_1y+u_2z)\Delta_N^{-1/2}}\dd u
\nonumber
\\
&=
\int_{a_N^{({k_1}-1)}}^{a_N^{(k_1)}}
\ee^{2\pi\iu u_1y\Delta_N^{-1/2}}\dd u_1
\int_{a_N^{({k_2}-1)}}^{a_N^{(k_2)}}
\ee^{2\pi\iu u_2z\Delta_N^{-1/2}}\dd u_2
\nonumber
\\
&=
\frac{\Delta_N}{(2\pi\iu)^2 y z}
\bigl(
\ee^{2\pi\iu a_N^{(k_1)} y\Delta_N^{-1/2}}
-\ee^{2\pi\iu a_N^{(k_1-1)}y\Delta_N^{-1/2}}
\bigr)
\bigl(
\ee^{2\pi\iu a_N^{(k_2)} z\Delta_N^{-1/2}}
-\ee^{2\pi\iu a_N^{(k_2-1)}z\Delta_N^{-1/2}}
\bigr)
\nonumber
\\
&=
\frac{\Delta_N}{(2\pi\iu)^2 y z}
(\ee^{\pi\iu y}-\ee^{-\pi\iu y})(\ee^{\pi\iu z}-\ee^{-\pi\iu z})
\ee^{2\pi\iu (k_1y+k_2z)}
\nonumber
\\
&=
\Delta_N\frac{\sin(\pi y)\sin(\pi z)}{\pi^2 y z}
\ee^{2\pi\iu (k_1y+k_2z)},
\end{align*}
it follows that
\begin{align}
&\Delta_N\sum_{\bs k\in\mathbb N^2}g(\bs k\Delta_N^{1/2})
\cos(2\pi k_1 y)\cos(2\pi k_2 z)
\nonumber
\\
&=
\frac{\Delta_N}{2}\sum_{\bs k\in\mathbb N^2}g(\bs k\Delta_N^{1/2})
\bigl(
\cos(2\pi(k_1 y+k_2 z))+\cos(2\pi(k_1 y-k_2 z))
\bigr)
\nonumber
\\
&=
\re\biggl(
\frac{\Delta_N}{2}
\sum_{\bs k\in\mathbb N^2}g(\bs k\Delta_N^{1/2})
\bigl(
\ee^{2\pi\iu(k_1 y+k_2 z)}+\ee^{2\pi\iu(k_1 y-k_2 z)}
\bigr)
\biggr)
\nonumber
\\
&=
\re\biggl(
\frac{\pi^2 y z}{2\sin(\pi y)\sin(\pi z)}
\sum_{\bs k\in\mathbb N^2}g(\bs k\Delta_N^{1/2})
\int_{E_N^{(\bs k)}}
\bigl(
\ee^{2\pi\iu (u_1y+u_2z)\Delta_N^{-1/2}}
+\ee^{2\pi\iu (u_1y-u_2z)\Delta_N^{-1/2}}
\bigr)
\dd u
\biggr)
\nonumber
\\
&=
\re\biggl(
\frac{\pi^2 y z}{2\sin(\pi y)\sin(\pi z)}
\Fou\biggl[
\sum_{\bs k\in\mathbb N^2}g(\bs k\Delta_N^{1/2})
\ind_{E_N^{(\bs k)}}
\bigg]
(-2\pi y\Delta_N^{-1/2},-2\pi z\Delta_N^{-1/2})
\biggr)
\nonumber
\\
&\qquad+
\re\biggl(
\frac{\pi^2 y z}{2\sin(\pi y)\sin(\pi z)}
\Fou\biggl[
\sum_{\bs k\in\mathbb N^2}g(\bs k\Delta_N^{1/2})
\ind_{E_N^{(\bs k)}}
\bigg]
(-2\pi y\Delta_N^{-1/2},2\pi z\Delta_N^{-1/2})
\biggr)
\nonumber
\\
&=
T_1+T_2-T_3,
\label{eq1-0303}
\end{align}
where
\begin{align*}
T_1
&=
\re\biggl(
\frac{\pi^2 y z}{2\sin(\pi y)\sin(\pi z)}
\Fou\biggl[
\sum_{\bs k\in\mathbb N^2}g(\bs k\Delta_N^{1/2})
\ind_{{E_N^{(\bs k)}}}
-g\ind_{E_N}
\bigg]
(-2\pi y\Delta_N^{-1/2}, -2\pi z\Delta_N^{-1/2})
\biggr)
\\
&\qquad+
\re\biggl(
\frac{\pi^2 y z}{2\sin(\pi y)\sin(\pi z)}
\Fou\biggl[
\sum_{\bs k\in\mathbb N^2}g(\bs k\Delta_N^{1/2})
\ind_{{E_N^{(\bs k)}}}
-g\ind_{E_N}
\bigg]
(-2\pi y\Delta_N^{-1/2}, 2\pi z\Delta_N^{-1/2})
\biggr),
\\
T_2
&=
\re\biggl(
\frac{\pi^2 y z}{2\sin(\pi y)\sin(\pi z)}
\Fou[g\ind_{E_N\cup D_N^{(1)}}]
(-2\pi y\Delta_N^{-1/2},-2\pi z\Delta_N^{-1/2})
\biggr)
\\
&\qquad
+\re\biggl(
\frac{\pi^2 y z}{2\sin(\pi y)\sin(\pi z)}
\Fou[g\ind_{E_N\cup D_N^{(1)}}]
(-2\pi y\Delta_N^{-1/2},2\pi z\Delta_N^{-1/2})
\biggr),
\\
T_3
&=
\re\biggl(
\frac{\pi^2 y z}{2\sin(\pi y)\sin(\pi z)}
\Fou[g\ind_{D_N^{(1)}}]
(-2\pi y\Delta_N^{-1/2},-2\pi z\Delta_N^{-1/2})
\biggr)
\\
&\qquad+
\re\biggl(
\frac{\pi^2 y z}{2\sin(\pi y)\sin(\pi z)}
\Fou[g\ind_{D_N^{(1)}}]
(-2\pi y\Delta_N^{-1/2},2\pi z\Delta_N^{-1/2})
\biggr).
\end{align*}
As proved below, $T_1$, $T_2$, and $T_3$ can be evaluated as 
\begin{align}
T_1&=O\bigl(\delta^{-3}(R_{1,N}\lor R_{2,N})\bigr),
\label{eq1-0304}
\\
T_2&=O(\delta^{-3}R_{1,N}),
\label{eq1-0305}
\\
T_3&=O\bigl(\delta^{-2}(\Delta_N^{1-\beta/2} \lor R_{1,N})\bigr),
\label{eq1-0306}
\end{align}
and we obtain the desired result from \eqref{eq1-0301} and \eqref{eq1-0303}.

\textit{Proof of \eqref{eq1-0305}.}
Since
\begin{align*}
x_1^2\Fou[g\ind_{\mathbb R\times(a_N^{(0)},\infty)}](x_1,x_2)
=(-\iu)^2\Fou[\partial_{x_1}^2 g\ind_{\mathbb R\times(a_N^{(0)},\infty)}](x_1,x_2),
\end{align*}
it holds that 
\begin{align*}
|x_1^2\Fou[g\ind_{\mathbb R\times(a_N^{(0)},\infty)}](x_1,x_2)|
&=
|\Fou[\partial_{x_1}^2 g\ind_{\mathbb R\times(a_N^{(0)},\infty)}](x_1,x_2)|
\le
\|\partial_{x_1}^2 g\ind_{\mathbb R\times(a_N^{(0)},\infty)}\|_{L^1(\mathbb R^2)},
\end{align*}
that is, for $x_1\neq0$, 
\begin{align}\label{eq1-0401}
|\Fou[g\ind_{\mathbb R\times(a_N^{(0)},\infty)}](x_1,x_2)|
\le
|x_1|^{-2}
\|\partial_{x_1}^2 g\ind_{\mathbb R\times(a_N^{(0)},\infty)}\|_{L^1(\mathbb R^2)}.
\end{align}
Let
$G_N^{(1)}=E_N\cup D_N^{(1)}=(0,\infty)\times(a_N^{(0)},\infty)$
and
$G_N^{(2)}=(-\infty,0]\times(a_N^{(0)},\infty)$.
From
\begin{align*}
\re\bigl(\Fou[g\ind_{G_N^{(2)}}](x_1,x_2)\bigr)
&=\int_{G_N^{(2)}}g(u)\cos(u_1x_1+u_2x_2)\dd u
\\
&=\int_{G_N^{(1)}}g(u)\cos(u_1x_1-u_2x_2)\dd u
=\re\bigl(\Fou[g\ind_{G_N^{(1)}}](x_1,-x_2)\bigr),
\end{align*}
we see that
\begin{align}\label{eq1-0402}
\re\bigl(
\Fou[g\ind_{E_N\cup D_N^{(1)}}](x_1,x_2)
\bigr)
+\re\bigl(
\Fou[g\ind_{E_N\cup D_N^{(1)}}](x_1,-x_2)
\bigr)
=
\re\bigl(
\Fou[g\ind_{\mathbb R\times(a_N^{(0)},\infty)}](x_1,x_2)
\bigr).
\end{align}
Since $1/\sin(\pi y)\le \delta^{-1}$, $y\in[\delta,1-\delta]$ and
$\Delta_N
\|\partial_{x_1}^2 g\ind_{\mathbb R\times(a_N^{(0)},\infty)}\|_{L^1(\mathbb R^2)}
=O(R_{1,N})$,
it follows from \eqref{eq1-0401} and \eqref{eq1-0402} that  
\begin{align}
|T_2|
&=
\frac{\pi^2 y z}{2\sin(\pi y)\sin(\pi z)}
\bigl|
\re\bigl(
\Fou[g\ind_{\mathbb R\times(a_N^{(0)},\infty)}]
(-2\pi y\Delta_N^{-1/2},-2\pi z\Delta_N^{-1/2})
\bigr)
\bigr|
\nonumber
\\
&\lesssim 
\frac{\pi z \Delta_N}{\pi y\sin(\pi y)\sin(\pi z)}
\|\partial_{x_1}^2 g\ind_{\mathbb R\times(a_N^{(0)},\infty)}\|_{L^1(\mathbb R^2)}
\nonumber
\\
&=O(\delta^{-3}R_{1,N}).
\label{eq1-0403}
\end{align}

\textit{Proof of \eqref{eq1-0306}.}
Applying integration by parts to $T_3$, we obtain that
\begin{align}
T_3
&=
\frac{\pi^2 y z}{\sin(\pi y)\sin(\pi z)}
\int_{D_N^{(2)}}
g(x)\cos(2\pi x_1y\Delta_N^{-1/2})\cos(2\pi x_2z\Delta_N^{-1/2})
\dd x
\nonumber
\\
&=
\frac{\pi^2 y z}{\sin(\pi y)\sin(\pi z)}
\int_{a_N^{(0)}}^\infty
\cos(2\pi x_2z\Delta_N^{-1/2})
\int_0^{a_N^{(0)}}
g(x_1,x_2)\cos(2\pi x_1y\Delta_N^{-1/2})
\dd x_1\dd x_2
\nonumber
\\
&=
\frac{\pi z\Delta_N^{1/2}}{2\sin(\pi z)}
\int_{a_N^{(0)}}^\infty
g(a_N^{(0)},x_2)\cos(2\pi x_2z\Delta_N^{-1/2})\dd x_2
\nonumber
\\
&\qquad-
\frac{\pi^2 y z\Delta_N}{\sin(\pi y)\sin(\pi z)}
\int_{a_N^{(0)}}^\infty
\cos(2\pi x_2z\Delta_N^{-1/2})
\int_0^{1/2}
\partial_{x_1} g(x_1\Delta_N^{1/2},x_2)
\frac{\sin(2\pi x_1y)}{2\pi y}\dd x_1 \dd x_2
\nonumber
\\
&=:I_1-I_2.
\label{eq1-0501}
\end{align}
Applying the integration by parts to $I_1$ again, we see that 
\begin{align*}
I_1
&=\frac{\pi z\Delta_N^{1/2}}{2\sin(\pi z)}
\int_{a_N^{(0)}}^\infty
g(a_N^{(0)},x_2)\cos(2\pi x_2z\Delta_N^{-1/2})\dd x_2
\\
&=
-\frac{\Delta_N}{4}g(a_N^{(0)},a_N^{(0)})
-\frac{\pi z\Delta_N}{2\sin(\pi z)}
\int_{a_N^{(0)}}^\infty
\partial_{x_2}g(a_N^{(0)},x_2)\frac{\sin(2\pi x_2z\Delta_N^{-1/2})}{2\pi z}\dd x_2
\\
&=
-\frac{\Delta_N}{4}g(a_N^{(0)},a_N^{(0)})
-\frac{\Delta_N}{4\sin(\pi z)}
\int_{a_N^{(0)}}^\infty
\partial_{x_2}g(a_N^{(0)},x_2)\sin(2\pi x_2z\Delta_N^{-1/2})\dd x_2.
\end{align*}
Noting that
\begin{equation}\label{eq1-0502}
|\partial_{x_1} g(x_1,x_2)|
\lesssim  |x_1||f'(\theta_2\pi^2|x|^2)|,
\end{equation}
\begin{equation*}
-\frac{\Delta_N}{4}g(a_N^{(0)},a_N^{(0)})
=O(\Delta_N f(\Delta_N))=O(\Delta_N^{1-\beta/2})
\end{equation*}
and
\begin{align*}
&\Biggl|
\frac{\Delta_N}{4\sin(\pi z)}
\int_{a_N^{(0)}}^\infty
\partial_{x_2}g(a_N^{(0)},x_2)\sin(2\pi x_2z\Delta_N^{-1/2})\dd x_2
\Biggr|
\\
&=
O\biggl(
\delta^{-1}\Delta_N
\int_{\Delta_N^{1/2}}^\infty
s|f'(s^2)|\dd s
\biggr)
=O(\delta^{-1}R_{1,N}),
\end{align*}
we have $I_1=O(\Delta_N^{1-\beta/2} \lor \delta^{-1}R_{1,N})$.
Furthermore, since $I_2$ can be evaluated as 
\begin{align*}
|I_2|
&=\Biggl|
\frac{\pi^2 y z\Delta_N}{\sin(\pi y)\sin(\pi z)}
\int_{a_N^{(0)}}^\infty
\cos(2\pi x_2z\Delta_N^{-1/2})
\int_0^{1/2}
\partial_{x_1} g(x_1\Delta_N^{1/2},x_2)
\frac{\sin(2\pi x_1y)}{2\pi y}\dd x_1 \dd x_2
\Biggr|
\\
&=
O\biggl(
\delta^{-2}\Delta_N
\int_{a_N^{(0)}}^\infty\int_0^{1/2}
|\partial_{x_1} g(x_1\Delta_N^{1/2},x_2)|\dd x_1 \dd x_2
\biggr)
\\
&=
O\biggl(
\delta^{-2}\Delta_N
\int_{D_N^{(2)}}
|\partial_{x_1} g(x)|\dd x
\biggr)
=O(\delta^{-2}R_{1,N}),
\end{align*}
we obtain
$T_3=O(\delta^{-2}(\Delta_N^{1-\beta/2} \lor R_{1,N}))$.

\textit{Proof of \eqref{eq1-0304}.}
Since
\begin{align*}
\re(\Fou[g](y,z)+\Fou[g](y,-z))
&=\re
\biggl(
\int_{\mathbb R^2}g(x)
(\ee^{-\iu(x_1y+x_2z)}+\ee^{-\iu(x_1y-x_2z)})\dd x
\biggr)
\\
&=2\int_{\mathbb R^2}g(x)
\cos(x_1y)\cos(x_2z)\dd x,
\end{align*}
$T_1$ can be evaluated as follows.
\begin{align}
|T_1|
&=
\frac{\pi^2 yz}{\sin(\pi y)\sin(\pi z)}
\Biggl|
\sum_{\bs k\in\mathbb N^2}
\int_{E_N^{(\bs k)}}
\bigl(
g(\bs k\Delta_N^{1/2})
-g(x)
\bigr)
\cos(2\pi x_1y\Delta_N^{-1/2})
\cos(2\pi x_2z\Delta_N^{-1/2})
\dd x
\Biggr|
\nonumber
\\
&=
\frac{\pi^2 yz}{\sin(\pi y)\sin(\pi z)}
\Biggl|
\sum_{\bs k\in\mathbb N^2}
\int_{E_N^{(\bs k)}}
\biggl(
\partial g(\bs k\Delta_N^{1/2})(x-\bs k\Delta_N^{1/2})
\nonumber
\\
&\qquad\qquad
+\frac12
(x-\bs k\Delta_N^{1/2})^\TT
\int_0^1(1-u)\partial^2 g(\bs k\Delta_N^{1/2}+u(x-\bs k\Delta_N^{1/2}))
\dd u
(x-\bs k\Delta_N^{1/2})
\biggr)
\nonumber
\\
&\qquad\qquad\times
\cos(2\pi x_1y\Delta_N^{-1/2})
\cos(2\pi x_2z\Delta_N^{-1/2})
\dd x
\Biggr|
\nonumber
\\
&\le
\frac{\pi^2 yz}{\sin(\pi y)\sin(\pi z)}
\Biggl|
\sum_{\bs k\in\mathbb N^2}
\int_{E_N^{(\bs k)}}
\partial g(\bs k\Delta_N^{1/2})(x-\bs k\Delta_N^{1/2})
\cos(2\pi x_1y\Delta_N^{-1/2})
\cos(2\pi x_2z\Delta_N^{-1/2})
\dd x
\Biggr|
\nonumber
\\
&\qquad+
\frac{\pi^2 yz}{2\sin(\pi y)\sin(\pi z)}
\sum_{\bs k\in\mathbb N^2}
\int_{E_N^{(\bs k)}}
\biggl|
\int_0^1(1-u)\partial^2 g(\bs k\Delta_N^{1/2}+u(x-\bs k\Delta_N^{1/2}))\dd u
\biggr|
|x-\bs k\Delta_N^{1/2}|^2\dd x
\nonumber
\\
&=
\frac{\pi^2 yz}{\sin(\pi y)\sin(\pi z)}
\Biggl|
\sum_{\bs k\in\mathbb N^2}
\partial g(\bs k\Delta_N^{1/2})
\int_{E_N^{(\bs k)}}
(x-\bs k\Delta_N^{1/2})
\cos(2\pi x_1y\Delta_N^{-1/2})
\cos(2\pi x_2z\Delta_N^{-1/2})
\dd x
\Biggr|
\nonumber
\\
&\qquad+O\biggl(
\delta^{-2}\Delta_N\int_{E_N}|\partial^2 g(x)|\dd x
\biggr).
\label{eq1-0601}
\end{align}
Since
\begin{align*}
\int_{a_N^{(k_1-1)}}^{a_N^{(k_1)}}
(x_1-k_1\Delta_N^{1/2})\cos(2\pi x_1y\Delta_N^{-1/2})\dd x_1
&=
\frac{\Delta_N}{2\pi^2y^2}
\sin(2\pi k_1y)
\bigl(\pi y\cos(\pi y)-\sin(\pi y)\bigr),
\\
\int_{a_N^{(k_2-1)}}^{a_N^{(k_2)}}
\cos(2\pi x_2z\Delta_N^{-1/2})
\dd x_2
&=
\frac{\Delta_N^{1/2}}{2\pi z}
\Bigl(
\sin\bigl(\pi z(2k_2+1)\bigr)
-\sin\bigl(\pi z(2k_2-1)\bigr)
\Bigr)
\\
&=
\Delta_N^{1/2}\frac{\sin(\pi z)}{\pi z}\cos(2\pi k_2z),
\end{align*}
the integral of the first term of \eqref{eq1-0601} can be expressed as 
\begin{align*}
&\int_{E_N^{(\bs k)}}
(x-\bs k\Delta_N^{1/2})
\cos(2\pi x_1y\Delta_N^{-1/2})\cos(2\pi x_2z\Delta_N^{-1/2})\dd x
\\
&=
\begin{pmatrix}
\displaystyle\int_{a_N^{(k_1-1)}}^{a_N^{(k_1)}}
(x_1-k_1\Delta_N^{1/2})\cos(2\pi x_1y\Delta_N^{-1/2})
\dd x_1
\int_{a_N^{(k_2-1)}}^{a_N^{(k_2)}}
\cos(2\pi x_2z\Delta_N^{-1/2})
\dd x_2
\\
\displaystyle\int_{a_N^{(k_1-1)}}^{a_N^{(k_1)}}
\cos(2\pi x_1y\Delta_N^{-1/2})
\dd x_1
\int_{a_N^{(k_2-1)}}^{a_N^{(k_2)}}
(x_2-k_2\Delta_N^{1/2})\cos(2\pi x_2z\Delta_N^{-1/2})
\dd x_2
\end{pmatrix}
\\
&=
\Delta_N^{3/2}
\begin{pmatrix}
\displaystyle
\sin(2\pi k_1y)\cos(2\pi k_2z)
\frac{\pi y\cos(\pi y)-\sin(\pi y)}{2\pi^2y^2}
\frac{\sin(\pi z)}{\pi z}
\\
\displaystyle
\cos(2\pi k_1y)\sin(2\pi k_2z)
\frac{\sin(\pi y)}{\pi y}
\frac{\pi z\cos(\pi z)-\sin(\pi z)}{2\pi^2z^2}
\end{pmatrix}.
\end{align*}
Since $(\pi y\cos(\pi y)-\sin(\pi y))/2\pi^2y^2$ is bounded in $y\in[0,1]$, 
the first term of \eqref{eq1-0601} can be evaluated as follows.
\begin{align*}
&\frac{\pi^2 yz}{\sin(\pi y)\sin(\pi z)}
\Biggl|
\sum_{\bs k\in\mathbb N^2}
\partial g(\bs k\Delta_N^{1/2})
\int_{E_N^{(\bs k)}}
(x-\bs k\Delta_N^{1/2})
\cos(2\pi x_1y\Delta_N^{-1/2})
\cos(2\pi x_2z\Delta_N^{-1/2})
\dd x
\Biggr|
\\
&\le
\frac{\pi y}{\sin(\pi y)}
\frac{\pi y\cos(\pi y)-\sin(\pi y)}{2\pi^2y^2}
\Biggl|
\Delta_N^{3/2}
\sum_{\bs k\in\mathbb N^2}
\partial g(\bs k\Delta_N^{1/2})
\sin(2\pi k_1y)\cos(2\pi k_2z)
\Biggr|
\\
&\qquad+
\frac{\pi z}{\sin(\pi z)}
\frac{\pi z\cos(\pi z)-\sin(\pi z)}{2\pi^2z^2}
\Biggl|
\Delta_N^{3/2}
\sum_{\bs k\in\mathbb N^2}
\partial g(\bs k\Delta_N^{1/2})
\cos(2\pi k_1y)\sin(2\pi k_2z)
\Biggr|
\\
&=
O\Biggl(\delta^{-1}
\Delta_N^{3/2}
\sum_{\bs k\in\mathbb N^2}
\partial g(\bs k\Delta_N^{1/2})\sin(2\pi k_1y)\cos(2\pi k_2z)
\Biggr).
\end{align*}
According to
\begin{align}\label{eq1-0602}
\Delta_N^{3/2}
\sum_{\bs k\in\mathbb N^2}
\partial g(\bs k\Delta_N^{1/2})\sin(2\pi k_1 y)\cos(2\pi k_2 z)
=O\bigl(\delta^{-2}(R_{1,N}\lor R_{2,N})\bigr),
\end{align}
the first term of \eqref{eq1-0601} is 
$O(\delta^{-3}(R_{1,N}\lor R_{2,N}))$, 
and thus we obtain \eqref{eq1-0304}.
Next, we will show \eqref{eq1-0602}.

\textit{Proof of \eqref{eq1-0602}. }
In a similar way to \eqref{eq1-0303}, it follows that
\begin{align}
&\Delta_N^{3/2}\sum_{\bs k\in\mathbb N^2}
\partial g(\bs k\Delta_N^{1/2})\sin(2\pi k_1 y)\cos(2\pi k_2 z)
\nonumber
\\
&=
\frac{\Delta_N^{3/2}}{2}\sum_{\bs k\in\mathbb N^2}
\partial g(\bs k\Delta_N^{1/2})
\bigl(
\sin(2\pi (k_1 y+k_2 z))+\sin(2\pi (k_1 y-k_2 z))
\bigr)
\nonumber
\\
&=
\im\biggl(
\frac{\pi^2 y z\Delta_N^{1/2}}{2\sin(\pi y)\sin(\pi z)}
\Fou\biggl[
\sum_{\bs k\in\mathbb N^2}\partial g(\bs k\Delta_N^{1/2})
\ind_{{E_N^{(\bs k)}}}
\biggl]
(-2\pi y\Delta_N^{-1/2},-2\pi z\Delta_N^{-1/2})
\biggr)
\nonumber
\\
&\qquad+
\im\biggl(
\frac{\pi^2 y z\Delta_N^{1/2}}{2\sin(\pi y)\sin(\pi z)}
\Fou\biggl[
\sum_{\bs k\in\mathbb N^2}\partial g(\bs k\Delta_N^{1/2})
\ind_{{E_N^{(\bs k)}}}
\biggl]
(-2\pi y\Delta_N^{-1/2},2\pi z\Delta_N^{-1/2})
\biggr)
\nonumber
\\
&=:
S_1+S_2-S_3,
\label{eq1-0603}
\end{align}
where
\begin{align*}
S_1
&=
\im\biggl(
\frac{\pi^2 y z\Delta_N^{1/2}}{2\sin(\pi y)\sin(\pi z)}
\Fou\biggl[
\sum_{\bs k\in\mathbb N^2}\partial g(\bs k\Delta_N^{1/2})
\ind_{{E_N^{(\bs k)}}}
-\partial g\ind_{E_N}
\biggl]
(-2\pi y\Delta_N^{-1/2},-2\pi y\Delta_N^{-1/2})
\biggr)
\\
&\qquad+
\im\biggl(
\frac{\pi^2 y z\Delta_N^{1/2}}{2\sin(\pi y)\sin(\pi z)}
\Fou\biggl[
\sum_{\bs k\in\mathbb N^2}\partial g(\bs k\Delta_N^{1/2})
\ind_{{E_N^{(\bs k)}}}
-\partial g\ind_{E_N}
\biggl]
(-2\pi y\Delta_N^{-1/2},2\pi y\Delta_N^{-1/2})
\biggr),
\\
S_2
&=
\im\biggl(
\frac{\pi^2 y z\Delta_N^{1/2}}{2\sin(\pi y)\sin(\pi z)}
\Fou[\partial g\ind_{E_N\cup D_N^{(1)}}]
(-2\pi y\Delta_N^{-1/2},-2\pi z\Delta_N^{-1/2})
\biggr)
\\
&\qquad+
\im\biggl(
\frac{\pi^2 y z\Delta_N^{1/2}}{2\sin(\pi y)\sin(\pi z)}
\Fou[\partial g\ind_{E_N\cup D_N^{(1)}}]
(-2\pi y\Delta_N^{-1/2},2\pi z\Delta_N^{-1/2})
\biggr),
\\
S_3
&=
\im\biggl(
\frac{\pi^2 y z\Delta_N^{1/2}}{2\sin(\pi y)\sin(\pi z)}
\Fou[\partial g\ind_{D_N^{(1)}}]
(-2\pi y\Delta_N^{-1/2},-2\pi z\Delta_N^{-1/2})
\biggr)
\\
&\qquad+
\im\biggl(
\frac{\pi^2 y z\Delta_N^{1/2}}{2\sin(\pi y)\sin(\pi z)}
\Fou[\partial g\ind_{D_N^{(1)}}]
(-2\pi y\Delta_N^{-1/2},2\pi z\Delta_N^{-1/2})
\biggr).
\end{align*}
Since
\begin{align}\label{eq1-0604}
|S_1|
\le
\frac{\pi^2 y z\Delta_N^{1/2}}{\sin(\pi y)\sin(\pi z)}
\biggl\|
\sum_{\bs k\in\mathbb N^2}
\partial g(\bs k\Delta_N^{1/2})\ind_{{E_N^{(\bs k)}}}-\partial g\ind_{E_N}
\biggr\|_{L^1(\mathbb R^2)},
\end{align}
$\pi y/\sin(\pi y)\le \delta^{-1}$, $y\in[\delta,1-\delta]$ and
\begin{align*}
&\Delta_N^{1/2}\biggl\|
\sum_{\bs k\in\mathbb N^2}
\partial g(\bs k\Delta_N^{1/2})\ind_{{E_N^{(\bs k)}}}
-\partial g\ind_{E_N}
\biggr\|_{L^1(\mathbb R^2)}
\\
&=
\Delta_N^{1/2}
\int_{\mathbb R_+^2}
\biggl|
\sum_{\bs k\in\mathbb N^2}
\partial g(\bs k\Delta_N^{1/2})\ind_{{E_N^{(\bs k)}}}(x)
-\partial g(x)\ind_{E}(x)
\biggr|\dd x
\\
&\le
\Delta_N^{1/2}
\int_{\mathbb R_+^2}
\sum_{\bs k\in\mathbb N^2}
\bigl|
\partial g(\bs k\Delta_N^{1/2})-\partial g(x)
\bigr|\ind_{{E_N^{(\bs k)}}}(x)\dd x
\\
&=
\Delta_N^{1/2}
\sum_{\bs k\in\mathbb N^2}
\int_{{E_N^{(\bs k)}}}
\bigl|
\partial g(\bs k\Delta_N^{1/2})-\partial g(x)
\bigr|\dd x
\\
&=
\Delta_N^{1/2}
\sum_{\bs k\in\mathbb N^2}
\int_{E_N^{(\bs k)}}
\biggl\|
\int_0^1\partial^2 g(\bs k\Delta_N^{1/2}+u(x-\bs k\Delta_N^{1/2}))\dd u
\biggr\|
|x-\bs k\Delta_N^{1/2}|\dd x
\\
&=O\biggl(
\Delta_N
\int_{E_N}\|\partial^2 g(x)\|\dd x
\biggr)
=O(R_{1,N}),
\end{align*}
we have $|S_1|=O(\delta^{-2}R_{1,N})$. 
Noting that for $x_1\neq0$, 
\begin{align*}
|\Fou[\partial g\ind_{\mathbb R\times(a_N^{(0)},\infty)}](x_1,x_2)|
\le
|x_1|^{-1}\|\partial_{x_1} \partial g\ind_{\mathbb R\times(a_N^{(0)},\infty)}
\|_{L^1(\mathbb R^2)},
\end{align*}
we see that
\begin{align}\label{eq1-0605}
|S_2|
\le
\frac{\pi z\Delta_N}{\sin(\pi y)\sin(\pi z)}
\|\partial_{x_1} \partial g\ind_{\mathbb R\times(a_N^{(0)},\infty)}
\|_{L^1(\mathbb R^2)}
=O(\delta^{-2}R_{1,N}).
\end{align}
Moreover, it follows from \eqref{eq1-0502} that 
\begin{align}
S_3
&=
\frac{\pi^2 y z\Delta_N^{1/2}}{\sin(\pi y)\sin(\pi z)}
\int_{D_N^{(1)}}
\partial g(x)\sin(2\pi x_1y\Delta_N^{-1/2})\cos(2\pi x_2z\Delta_N^{-1/2})
\dd x
\nonumber
\\
&=
\frac{\pi^2 y z\Delta_N^{1/2}}{\sin(\pi y)\sin(\pi z)}
\int_{a_N^{(0)}}^\infty \cos(2\pi x_2z\Delta_N^{-1/2})
\int_0^{a_N^{(0)}} \partial g(x_1,x_2)\sin(2\pi x_1y\Delta_N^{-1/2})
\dd x_1\dd x_2 
\nonumber
\\
&=
\frac{\pi z\Delta_N}{2\sin(\pi y)\sin(\pi z)}
\biggl(
\int_{a_N^{(0)}}^\infty 
(\partial g(0,x_2)-\partial g(a_N^{(0)},x_2)\cos(\pi y)) 
\cos(2\pi x_2z\Delta_N^{-1/2})
\dd x_2
\nonumber
\\
&\qquad\qquad-
\int_{a_N^{(0)}}^\infty\int_0^{a_N^{(0)}}
\partial_{x_1} \partial g(x)
\cos(2\pi x_1y\Delta_N^{-1/2})
\cos(2\pi x_2z\Delta_N^{-1/2})
\dd x_1\dd x_2 
\biggr)
\nonumber
\\
&=
O\Biggl(
\delta^{-2}\Delta_N
\biggl(
\int_{\Delta_N^{1/2}}^\infty s|f'(s^2)|\dd s
\lor \int_{D_N^{(1)}}\|\partial^2 g(x)\|\dd x 
\biggr)
\lor
\delta^{-2}\Delta_N^{3/2}
\int_{\Delta_N^{1/2}}^\infty |f'(s^2)|\dd s
\Biggr)
\nonumber
\\
&=
O\bigl(\delta^{-2}(R_{1,N}\lor R_{2,N})\bigr).
\label{eq1-0606}
\end{align}
Therefore, this completes the proof of \eqref{eq1-0602}.


(2) In the same way as \eqref{eq1-0201}, we have
\begin{align}\label{eq1-0701}
\Delta_N\sum_{\bs k\in\mathbb N^2}f(\lambda_{\bs k}\Delta_N)
\cos(2\pi k_1 y)
=
\Delta_N\sum_{\bs k\in\mathbb N^2}g(\bs k\Delta_N^{1/2})
\cos(2\pi k_1 y)
+O(R_{1,N}).
\end{align}
Note that the first term on the right hand side of \eqref{eq1-0701}
is obtained by  substituting $0$ for $z$ for \eqref{eq1-0303}. 
Thus, 
setting $\pi z/\sin(\pi z)=1$ if $z=0$, we have
\begin{align*}
\Delta_N\sum_{\bs k\in\mathbb N^2}g(\bs k\Delta_N^{1/2})\cos(2\pi k_1 y)
=T_1+T_2-T_3,
\end{align*}
where
\begin{align*}
T_1
&=
\re\biggl(
\frac{\pi y}{\sin(\pi y)}
\Fou\biggl[
\sum_{\bs k\in\mathbb N^2}g(\bs k\Delta_N^{1/2})
\ind_{{E_N^{(\bs k)}}}
-g\ind_{E_N}
\biggl]
(-2\pi y\Delta_N^{-1/2},0)
\biggr),
\\
T_2
&=
\re\biggl(
\frac{\pi y}{\sin(\pi y)}
\Fou[g\ind_{E_N\cup D_N^{(1)}}]
(-2\pi y\Delta_N^{-1/2},0)
\biggr),
\\
T_3
&=
\re\biggl(
\frac{\pi y}{\sin(\pi y)}
\Fou[g\ind_{D_N^{(1)}}]
(-2\pi y\Delta_N^{-1/2},0)
\biggr).
\end{align*}
From \eqref{eq1-0603}-\eqref{eq1-0606}, it holds that
\begin{align*}
\Delta_N^{3/2}
\sum_{\bs k\in\mathbb N^2}
\partial g(\bs k\Delta_N^{1/2})
\sin(2\pi k_1y)
=O\bigl(\delta^{-1}(R_{1,N}\lor R_{2,N})\bigr).
\end{align*}
Setting $(\pi z\cos(\pi z)-\sin(\pi z))/2\pi^2z^2=0$ if $z=0$, we see that
\begin{align*}
|T_1|
&\le\frac{\pi y}{\sin(\pi y)}
\frac{\pi y\cos(\pi y)-\sin(\pi y)}{2\pi^2y^2}
\Biggl|
\Delta_N^{3/2}
\sum_{\bs k\in\mathbb N^2}
\partial g(\bs k\Delta_N^{1/2})
\sin(2\pi k_1y)
\Biggr|
\\
&=O\Biggl(\delta^{-1}
\Delta_N^{3/2}
\sum_{\bs k\in\mathbb N^2}
\partial g(\bs k\Delta_N^{1/2})
\sin(2\pi k_1y)
\Biggr)
\\
&=O\bigl(\delta^{-2}(R_{1,N}\lor R_{2,N})\bigr).
\end{align*}
Furthermore, we obtain from \eqref{eq1-0403} and \eqref{eq1-0501} that
$T_2=O(\delta^{-2}R_{1,N})$ and
$T_3=I_1-I_2$, where
\begin{align*}
I_1
&=
\frac{\Delta_N^{1/2}}{2}
\int_{a_N^{(0)}}^\infty 
g(a_N^{(0)},x_2) \dd x_2
=\int_{a_N^{(0)}}^\infty\int_0^{a_N^{(0)}}
g(a_N^{(0)},x_2)\dd x_1\dd x_2,
\\
I_2
&=
\frac{\pi y\Delta_N}{\sin(\pi y)}
\int_{a_N^{(0)}}^\infty\int_0^{1/2}
\partial_{x_1} g(x_1\Delta_N^{1/2},x_2)
\frac{\sin(2\pi x_1y)}{2\pi y}\dd x_1 \dd x_2.
\end{align*}
Noting that \eqref{eq1-0502} and
\begin{align*}
&\Biggl|
\int_{a_N^{(0)}}^\infty\int_0^{a_N^{(0)}}
g(a_N^{(0)},x_2)\dd x_1\dd x_2
-
\int_{a_N^{(0)}}^\infty\int_0^{a_N^{(0)}}
g(x_1,x_2)\dd x_1\dd x_2
\Biggr|
\\
&\le
\int_{a_N^{(0)}}^\infty\int_0^{a_N^{(0)}}
|g(a_N^{(0)},x_2)-g(x_1,x_2)|\dd x_1\dd x_2
\\
&=
\int_{a_N^{(0)}}^\infty\int_0^{a_N^{(0)}}
\Biggl|
\int_0^1\partial_{x_1} g(a_N^{(0)}+u(x_1-a_N^{(0)}),x_2)\dd u
\Biggr|
|x_1-a_N^{(0)}|\dd x_1\dd x_2
\\
&\lesssim 
\Delta_N^{1/2}
\int_{a_N^{(0)}}^\infty |f'(x_2^2)|
\int_0^{a_N^{(0)}}|x_1-a_N^{(0)}|\dd x_1\dd x_2
\\
&\lesssim
\Delta_N^{3/2}
\int_{a_N^{(0)}}^{\infty}|f'(x_2^2)|\dd x_2
\\
&=O(R_{2,N}),
\end{align*}
we have that
$I_1
=\int_{a_N^{(0)}}^\infty\int_0^{a_N^{(0)}} g(x_1,x_2)\dd x_1\dd x_2
+O(R_{2,N})$,
$I_2=O(\delta^{-1}R_{1,N})$
and 
\begin{align*}
T_3
&=\int_{a_N^{(0)}}^\infty\int_0^{a_N^{(0)}}
g(x_1,x_2)\dd x_1\dd x_2
+O(R_{2,N})+O(\delta^{-1}R_{1,N})
\nonumber
\\
&=\int_{D_N^{(1)}}g(x)\dd x
+O\bigl(\delta^{-1}(R_{1,N} \lor R_{2,N})\bigr).
\end{align*}
Therefore, the desired result can be obtained.
\end{proof}

Let
\begin{equation}\label{functions}
f_\alpha(s)=\frac{1-\ee^{-s}}{s^{1+\alpha}},
\quad
f_{\alpha,\tau}(s)=\frac{(1-\ee^{-s})^2}{s^{1+\alpha}}\ee^{-\tau s}
\end{equation}
for $\alpha\in(0,1)$ and $\tau\ge0$.

\begin{lem}\label{lem2}
$f_\alpha\in \mathcal F_{2\alpha}$ 
and $f_{\alpha,\tau} \in \mathcal F_{-2+2\alpha}$.
Moreover, it holds that
\begin{align}
&\Delta_N
\sum_{\bs k\in\mathbb N^2}
f_\alpha(\lambda_{\bs k}\Delta_N)e_{\bs k}^2(y,z)
=
\frac{\Gamma(1-\alpha)}{4\pi\alpha\theta_2}
\ee^{-\frac{\theta_1}{\theta_2}y}\ee^{-\frac{\eta_1}{\theta_2}z}
+O(\delta^{-3}\Delta_{N}^{1-\alpha}),
\label{eq2-lem2-1}
\\
&\Delta_N
\sum_{\bs k\in\mathbb N^2}
f_{\alpha,\tau}(\lambda_{\bs k}\Delta_N)e_{\bs k}^2(y,z)
\nonumber
\\
&=
\frac{\Gamma(1-\alpha)}{4\pi\alpha\theta_2}
\bigl(
-\tau^\alpha+2(1+\tau)^\alpha-(2+\tau)^\alpha
\bigr)
\ee^{-\frac{\theta_1}{\theta_2}y}\ee^{-\frac{\eta_1}{\theta_2}z}
+O(\delta^{-3}\Delta_{N}^{1-\alpha}).
\label{eq2-lem2-2}
\end{align}
\end{lem}

\begin{proof}
Since 
$e_{\bs k}^2(y,z)=(1-\cos(2\pi k_1y))(1-\cos(2\pi k_2z))
\ee^{-\frac{\theta_1}{\theta_2}y}\ee^{-\frac{\eta_1}{\theta_2}z}$,
it holds from Lemma \ref{lem1} that 
for $f\in \mathcal F_\beta$,
\begin{align}
&\Delta_N
\sum_{\bs k\in\mathbb N^2}
f(\lambda_{\bs k}\Delta_N)e_{\bs k}^2(y,z)
\nonumber
\\
&=
\ee^{-\frac{\theta_1}{\theta_2}y}\ee^{-\frac{\eta_1}{\theta_2}z}
\Biggl(
\Delta_N
\sum_{\bs k\in\mathbb N^2}
f(\lambda_{\bs k}\Delta_N)
-
\Delta_N
\sum_{\bs k\in\mathbb N^2}
f(\lambda_{\bs k}\Delta_N)
(\cos(2\pi k_1y)+\cos(2\pi k_2z))
\nonumber
\\
&\qquad\qquad+
\Delta_N
\sum_{\bs k\in\mathbb N^2}
f(\lambda_{\bs k}\Delta_N)
\cos(2\pi k_1y)\cos(2\pi k_2z)
\Biggr)
\nonumber
\\
&=
\ee^{-\frac{\theta_1}{\theta_2}y}\ee^{-\frac{\eta_1}{\theta_2}z}
\Biggl(
\frac{1}{4\pi\theta_2}
\int_0^\infty f(s)\dd s
-\int_{D_N^{(1)}\cup D_N^{(2)}}f(\theta_2\pi^2|x|^2)\dd x
\nonumber
\\
&\qquad\qquad+
\int_{D_N^{(1)}}f(\theta_2\pi^2|x|^2)\dd x
+\int_{D_N^{(2)}}f(\theta_2\pi^2|x|^2)\dd x
+O\bigl(\delta^{-3}(R_{1,N} \lor R_{2,N} \lor \Delta_N^{1-\beta/2})\bigr)
\Biggr)
\nonumber
\\
&=
\frac{1}{4\pi\theta_2}
\ee^{-\frac{\theta_1}{\theta_2}y}\ee^{-\frac{\eta_1}{\theta_2}z}
\int_0^\infty f(s)\dd s
+O\bigl(\delta^{-3}(R_{1,N} \lor R_{2,N} \lor \Delta_N^{1-\beta/2})\bigr).
\label{eq2-0001}
\end{align}

Since
\begin{align*}
f'_\alpha(s)
&=\frac{1}{s^{2+\alpha}}
\bigl(-(1-\ee^{-s})(s+1+\alpha)+s\bigr),
\\
f''_\alpha(s)
&=\frac{1}{s^{3+\alpha}}
\bigl\{
(1-\ee^{-s})\bigl(s^2+2(1+\alpha)s+(1+\alpha)(2+\alpha)\bigr)
-s\bigl(s+2(1+\alpha)\bigr)
\bigr\},
\end{align*}
$1-\ee^{-s} \sim s$ ($s\to0$) 
and $1-\ee^{-s} \sim 1$ ($s\to\infty$),
it follows that
\begin{equation}\label{eq2-0101}
f_\alpha(s^2)\sim \frac{1}{s^{2\alpha}},
\quad 
f'_\alpha(s^2)\sim \frac{1}{s^{2+2\alpha}},
\quad
f''_\alpha(s^2)\sim \frac{1}{s^{4+2\alpha}}
\ (s\to0),
\end{equation}
\begin{equation*}
f_\alpha(s^2)\sim \frac{1}{s^{2+2\alpha}},
\quad 
f'_\alpha(s^2)\sim \frac{1}{s^{4+2\alpha}},
\quad
f''_\alpha(s^2)\sim \frac{1}{s^{6+2\alpha}}
\ (s\to\infty),
\end{equation*}
and therefore $f_\alpha\in\mathcal F_{2\alpha}$. 
Moreover, we have from \eqref{eq2-0101} that 
\begin{align*}
R_{1,N}
&=
O\Biggl(
\Delta_N
\biggl(
\int_{\Delta_N^{1/2}}^1 s|f_\alpha'(s^2)|\dd s
\lor 
\int_{\Delta_N^{1/2}}^1 s^3|f_\alpha''(s^2)|\dd s
\biggr)
\Biggr)
=O\biggl(\Delta_N\int_{\Delta_N^{1/2}}^1\frac{\dd s}{s^{1+2\alpha}}\biggr)
=O(\Delta_N^{1-\alpha}),
\\
R_{2,N}
&=
O\biggl(
\Delta_N^{3/2}
\int_{\Delta_N^{1/2}}^1 |f_\alpha'(s^2)|\dd s
\biggr)
=O\biggl(\Delta_N^{3/2}\int_{\Delta_N^{1/2}}^1\frac{\dd s}{s^{2+2\alpha}}\biggr)
=O(\Delta_N^{1-\alpha})
\end{align*}
and $R_{1,N} \lor R_{2,N} \lor \Delta_{N}^{1-\alpha}=O(\Delta_N^{1-\alpha})$.

In the same way, noting that
\begin{align*}
f'_{\alpha,\tau}(s)
&=
\frac{\ee^{-\tau s}}{s^{2+\alpha}}
\bigl\{
-(1-\ee^{-s})^2\bigl((2-\tau)s+1+\alpha\bigr)
+2s(1-\ee^{-s})
\bigr\},
\\
f''_{\alpha,\tau}(s)
&=
\frac{\ee^{-\tau s}}{s^{3+\alpha}}
\bigl\{
(1-\ee^{-s})^2
\bigl((\tau+2)^2s^2+2(1+\alpha)(\tau+2)s+(1+\alpha)(2+\alpha)\bigr)
\\
&\qquad
-4(1-\ee^{-s})s(2s+1+\alpha+\tau)+2s^2
\bigr\},
\end{align*}
\begin{equation*}
f_{\alpha,\tau}(s^2)\sim s^{2-2\alpha}, 
\quad
f'_{\alpha,\tau}(s^2)\sim \frac{1}{s^{2\alpha}},
\quad
f''_{\alpha,\tau}(s^2)\sim \frac{\tau\ee^{-\tau s^2}}{s^{2+2\alpha}}
\ (s\to0),
\end{equation*}
\begin{equation*}
f_{\alpha,\tau}(s^2)\sim \frac{1}{s^{2+2\alpha}}, 
\quad
f'_{\alpha,\tau}(s^2)
\sim \frac{\tau \ee^{-\tau s^2}}{s^{2+2\alpha}}
\le \frac{1}{s^{4+2\alpha}},
\quad
f''_{\alpha,\tau}(s^2)
\sim \frac{\tau^2 \ee^{-\tau s^2}}{s^{2+2\alpha}}
\le \frac{1}{s^{6+2\alpha}}
\ (s\to\infty),
\end{equation*}
we obtain that
$f_{\alpha,\tau}\in\mathcal F_{-2+2\alpha}$ and 
$R_{1,N} \lor R_{2,N} \lor \Delta_N^{1-(-1+\alpha)}=O(\Delta_N^{1-\alpha})$. 

Since for $\beta>0$,
\begin{align*}
\int_0^\infty 
\frac{1-\ee^{-\beta s}}{s^{1+\alpha}}
\dd s
=
\frac{\beta}{\alpha}
\int_0^\infty 
s^{(1-\alpha)-1}\ee^{-\beta s}
\dd s
=\frac{\beta^\alpha}{\alpha}\Gamma(1-\alpha),
\end{align*}
it follows that
\begin{align*}
\int_0^\infty f_\alpha(s)\dd s
&=\frac{\Gamma(1-\alpha)}{\alpha},
\\
\int_0^\infty f_{\alpha,\tau}(s)\dd s
&=
-\int_0^\infty
\frac{1-\ee^{-\tau s}}{s^{1+\alpha}} \dd s
+2\int_0^\infty
\frac{1-\ee^{-(1+\tau)s}}{s^{1+\alpha}} \dd s
-\int_0^\infty
\frac{1-\ee^{-(2+\tau)s}}{s^{1+\alpha}}
\dd s
\nonumber
\\
&=
\frac{\Gamma(1-\alpha)}{\alpha}
\bigl(
-\tau^\alpha+2(1+\tau)^\alpha-(2+\tau)^\alpha
\bigr).
\end{align*}
Therefore, we obtain from \eqref{eq2-0001} 
that \eqref{eq2-lem2-1} and \eqref{eq2-lem2-2}.
\end{proof}
Since $\Delta_N^{(\beta)}=o(1)$ for $\beta<2$
and (1) in Lemma \ref{lem1} and Lemma \ref{lem2},
it follows that
\begin{align}
\sum_{\bs k\in\mathbb N^2}
\frac{1-\ee^{-\lambda_{\bs k}\Delta_N}}{\lambda_{\bs k}^{1+\alpha}}
&=
\Delta_N^{1+\alpha}
\sum_{\bs k\in\mathbb N^2}
f_\alpha(\lambda_{\bs k}\Delta_N)
=O(\Delta_N^\alpha),
\label{eq2-0002}
\\
\sum_{\bs k\in\mathbb N^2}
\frac{(1-\ee^{-\lambda_{\bs k}\Delta_N})^2}{\lambda_{\bs k}^{1+\alpha}}
&=
\Delta_N^{1+\alpha}
\sum_{\bs k\in\mathbb N^2}
f_{\alpha,0}(\lambda_{\bs k}\Delta_N)
=O(\Delta_N^\alpha).
\label{eq2-0003}
\end{align}

For the coordinate process $x_{\bs k}^{Q}(t)$, let 
$\Delta_i x_{\bs k}^{Q}=x_{\bs k}^{Q}(t_i)-x_{\bs k}^{Q}(t_{i-1})$.
By setting
\begin{align*}
A_{i,\bs k}
&=
-\langle \xi,e_{\bs k}\rangle_\theta
(1-\ee^{-\lambda_{\bs k}\Delta_N})\ee^{-\lambda_{\bs k}(i-1)\Delta_N},
\\
B_{1,i,\bs k}^{Q_1}
&=
-\frac{\sigma(1-\ee^{-\lambda_{\bs k}\Delta_N})}{\lambda_{\bs k}^{\alpha/2}}
\int_0^{(i-1)\Delta_N}\ee^{-\lambda_{\bs k}((i-1)\Delta_N-s)}\dd w_{\bs k}(s),
\\
B_{2,i,\bs k}^{Q_1}
&=
\frac{\sigma}{\lambda_{\bs k}^{\alpha/2}}
\int_{(i-1)\Delta_N}^{i\Delta_N}
\ee^{-\lambda_{\bs k}(i\Delta_N-s)}\dd w_{\bs k}(s),
\end{align*}
and $B_{i,\bs k}^{Q_1}=B_{1,i,\bs k}^{Q_1}+B_{2,i,\bs k}^{Q_1}$,
the increment $\Delta_i x_{\bs k}^{Q_1}$ can be expressed as 
$\Delta_i x_{\bs k}^{Q_1}=A_{i,\bs k}+B_{i,\bs k}^{Q_1}$.

\begin{lem}\label{lem3}
Under Assumption \ref{asm}, 
it holds that uniformly in $(y,z)\in D_\delta$,
\begin{align}
&\sum_{i=1}^N\sum_{\bs k_1,\bs k_2\in\mathbb N^2}
\Bigl|\EE[A_{i,\bs k_1}A_{i,\bs k_2}]e_{\bs k_1}(y,z)e_{\bs k_2}(y,z)\Bigr|
=O(\Delta_N^{\alpha}),
\label{eq3-lem3-1}
\\
&\sup_{j\ge1}\sum_{i=1}^N\sum_{\bs k_1,\bs k_2\in\mathbb N^2}
\Bigl|\EE[A_{i,\bs k_1}A_{j,\bs k_2}]e_{\bs k_1}(y,z)e_{\bs k_2}(y,z)\Bigr|
=O(\Delta_N^{\alpha/2}),
\label{eq3-lem3-2}
\\
&\sum_{i,j=1}^N\sum_{\bs k_1,\ldots,\bs k_4\in\mathbb N^2}
\Bigl|
\EE[A_{i,\bs k_1}A_{i,\bs k_2}A_{j,\bs k_3}A_{j,\bs k_4}]
e_{\bs k_1}(y,z)\cdots e_{\bs k_4}(y,z)
\Bigr|=O(\Delta_N^{\alpha}).
\label{eq3-lem3-3}
\end{align}
\end{lem}

\begin{proof}
First of all, we show \eqref{eq3-lem3-1}-\eqref{eq3-lem3-3} under 
(i), (iii) and (iv) in Assumption \ref{asm}. 
It follows from (i) and (iv) that
\begin{align*}
&\sum_{\bs k_1,\bs k_2\in\mathbb N^2}
\Bigl|\EE[A_{i,\bs k_1}A_{j,\bs k_2}]e_{\bs k_1}(y,z)e_{\bs k_2}(y,z)\Bigr|
\\
&\lesssim
\sum_{\bs k_1,\bs k_2\in\mathbb N^2}
(1-\ee^{-\lambda_{\bs k_1}\Delta_N})(1-\ee^{-\lambda_{\bs k_2}\Delta_N})
\ee^{-\lambda_{\bs k_1}(i-1)\Delta_N }
\ee^{-\lambda_{\bs k_2}(j-1)\Delta_N }
\bigl|\EE[\langle\xi,e_{\bs k_1}\rangle_\theta 
\langle\xi,e_{\bs k_2}\rangle_\theta]\bigr|
\\
&=
\sum_{\bs k\in\mathbb N^2}
(1-\ee^{-\lambda_{\bs k}\Delta_N})^2
\ee^{-\lambda_{g k}(i-1)\Delta_N }
\ee^{-\lambda_{\bs k}(j-1)\Delta_N }
\EE[\langle\xi,e_{\bs k}\rangle_\theta^2]
\\
&\lesssim
\sum_{\bs k\in\mathbb N^2}
\frac{(1-\ee^{-\lambda_{\bs k}\Delta_N})^2}{\lambda_{\bs k}^{1+\alpha}}
\ee^{-\lambda_{\bs k}(i-1)\Delta_N}
\ee^{-\lambda_{\bs k}(j-1)\Delta_N}.
\end{align*}
Therefore, we have from \eqref{eq2-0002} and \eqref{eq2-0003} that 
\begin{align*}
&\sum_{i=1}^N\sum_{\bs k_1,\bs k_2\in\mathbb N^2}
\Bigl|\EE[A_{i,\bs k_1}A_{i,\bs k_2}]e_{\bs k_1}(y,z)e_{\bs k_2}(y,z)\Bigr|
\\
&\lesssim
\sum_{i=1}^N\sum_{\bs k\in\mathbb N^2}
\frac{(1-\ee^{-\lambda_{\bs k}\Delta_N})^2}{\lambda_{\bs k}^{1+\alpha}}
\ee^{-2\lambda_{\bs k}(i-1)\Delta_N }
\\
&\le
\sum_{\bs k\in\mathbb N^2}
\frac{1-\ee^{-\lambda_{\bs k}\Delta_N}}{\lambda_{\bs k}^{1+\alpha}}
=
O(\Delta_N^{\alpha}),
\end{align*}
and
\begin{align*}
&\sup_{j\ge1}\sum_{i=1}^N\sum_{\bs k_1,\bs k_2\in\mathbb N^2}
\Bigl|\EE[A_{i,\bs k_1}A_{j,\bs k_2}]e_{\bs k_1}(y,z)e_{\bs k_2}(y,z)\Bigr|
\\
&\lesssim
\sup_{j\ge1}\sum_{i=1}^N
\sum_{\bs k\in\mathbb N^2}
\frac{(1-\ee^{-\lambda_{\bs k}\Delta_N})^2}{\lambda_{\bs k}^{1+\alpha}}
\ee^{-\lambda_{\bs k}(i-1)\Delta_N}
\ee^{-\lambda_{\bs k}(j-1)\Delta_N}
\\
&\le
\sum_{\bs k_1\in\mathbb N^2}
\frac{1-\ee^{-\lambda_{\bs k_1}\Delta_N}}{\lambda_{\bs k_1}^{1+\alpha}}
=O(\Delta_N^{\alpha}).
\end{align*}
Moreover, since (i), (iii), (iv) and 
\begin{align*}
&\sum_{\bs k_1,\ldots,\bs k_4\in\mathbb N^2}
\Bigl|\EE[A_{i,\bs k_1}A_{i,\bs k_2}A_{j,\bs k_3}A_{j,\bs k_4}]
e_{\bs k_1}(y,z)\cdots e_{\bs k_4}(y,z)\Bigr|
\\
&\lesssim
\sum_{\bs k_1,\ldots,\bs k_4\in\mathbb N^2}
(1-\ee^{-\lambda_{\bs k_1}\Delta_N})\cdots(1-\ee^{-\lambda_{\bs k_4}\Delta_N})
\ee^{-(\lambda_{\bs k_1}+\lambda_{\bs k_2})(i-1)\Delta_N}
\ee^{-(\lambda_{\bs k_3}+\lambda_{\bs k_4})(j-1)\Delta_N}
\\
&\qquad\qquad\times
\bigl|
\EE[\langle \xi,e_{\bs k_1}\rangle_\theta\cdots\langle \xi,e_{\bs k_4}\rangle_\theta]
\bigr|
\\
&=
\sum_{\bs k\in\mathbb N^2}
(1-\ee^{-\lambda_{\bs k}\Delta_N})^4
\ee^{-2\lambda_{\bs k}(i-1)\Delta_N}
\ee^{-2\lambda_{\bs k}(j-1)\Delta_N}
\EE[\langle \xi,e_{\bs k}\rangle_\theta^4]
\\
&\qquad+
\sum_{\bs k_1\neq\bs  k_2}
(1-\ee^{-\lambda_{\bs k_1}\Delta_N})^2(1-\ee^{-\lambda_{\bs k_2}\Delta_N})^2
\\
&\qquad\qquad\times
\bigl(
\ee^{-2\lambda_{\bs k_1}(i-1)\Delta_N}
\ee^{-2\lambda_{\bs k_2}(j-1)\Delta_N}
+2\ee^{-(\lambda_{\bs k_1}+\lambda_{\bs k_2})(i-1)\Delta_N}
\ee^{-(\lambda_{\bs k_1}+\lambda_{\bs k_2})(j-1)\Delta_N}
\bigr)
\\
&\qquad\qquad\times
\EE[\langle \xi,e_{\bs k_1}\rangle_\theta^2]
\EE[\langle \xi,e_{\bs k_2}\rangle_\theta^2]
\\
&\lesssim
\sum_{\bs k\in\mathbb N^2}
\frac{(1-\ee^{-\lambda_{\bs k}\Delta_N})^4}{\lambda_{\bs k}^{1+\alpha}}
\ee^{-2\lambda_{\bs k}(i-1)\Delta_N}
\ee^{-2\lambda_{\bs k}(j-1)\Delta_N}
\\
&\qquad+
\Biggl(
\sum_{\bs k\in\mathbb N^2}
\frac{(1-\ee^{-\lambda_{\bs k}\Delta_N})^2}{\lambda_{\bs k}^{1+\alpha}}
\ee^{-2\lambda_{\bs k}(i-1)\Delta_N}
\Biggr)
\Biggl(
\sum_{\bs k\in\mathbb N^2}
\frac{(1-\ee^{-\lambda_{\bs k}\Delta_N})^2}{\lambda_{\bs k}^{1+\alpha}}
\ee^{-2\lambda_{\bs k}(j-1)\Delta_N}
\Biggr)
\\
&\qquad+
\Biggl(
\sum_{\bs k\in\mathbb N^2}
\frac{(1-\ee^{-\lambda_{\bs k}\Delta_N})^2}{\lambda_{\bs k}^{1+\alpha}}
\ee^{-\lambda_{\bs k}(i+j-2)\Delta_N}
\Biggr)^2,
\end{align*}
it holds from \eqref{eq2-0002} and \eqref{eq2-0003} that
\begin{align*}
&\sum_{i,j=1}^N\sum_{\bs k_1,\ldots,\bs k_4\in\mathbb N^2}
\Bigl|\EE[A_{i,\bs k_1}A_{i,\bs k_2}A_{j,\bs k_3}A_{j,\bs k_4}]
e_{\bs k_1}(y,z)\cdots e_{\bs k_4}(y,z)\Bigr|
\\
&\lesssim
\sum_{\bs k\in\mathbb N^2}
\frac{(1-\ee^{-\lambda_{\bs k}\Delta_N})^4}{\lambda_{\bs k}^{1+\alpha}}
\sum_{i=1}^N\ee^{-2\lambda_{\bs k}(i-1)\Delta_N}
\sum_{j=1}^N\ee^{-2\lambda_{\bs k}(j-1)\Delta_N}
\\
&\qquad+
\Biggl(
\sum_{\bs k\in\mathbb N^2}
\frac{(1-\ee^{-\lambda_{\bs k}\Delta_N})^2}{\lambda_{\bs k}^{1+\alpha}}
\sum_{i=1}^N\ee^{-2\lambda_{\bs k}(i-1)\Delta_N}
\Biggr)
\Biggl(
\sum_{\bs k\in\mathbb N^2}
\frac{(1-\ee^{-\lambda_{\bs k}\Delta_N})^2}{\lambda_{\bs k}^{1+\alpha}}
\sum_{j=1}^N\ee^{-2\lambda_{\bs k}(j-1)\Delta_N}
\Biggr)
\\
&\qquad+
\Biggl(
\sum_{\bs k\in\mathbb N^2}
\frac{(1-\ee^{-\lambda_{\bs k}\Delta_N})^2}{\lambda_{\bs k}^{1+\alpha}}
\sum_{i=1}^N\ee^{-\lambda_{\bs k}(i-1)\Delta_N}
\Biggr)
\Biggl(
\sum_{\bs k\in\mathbb N^2}
\frac{(1-\ee^{-\lambda_{\bs k}\Delta_N})^2}{\lambda_{\bs k}^{1+\alpha}}
\sum_{j=1}^N\ee^{-\lambda_{\bs k}(j-1)\Delta_N}
\Biggr)
\\
&\lesssim
\sum_{\bs k\in\mathbb N^2}
\frac{(1-\ee^{-\lambda_{\bs k}\Delta_N})^2}{\lambda_{\bs k}^{1+\alpha}}
+
\Biggl(
\sum_{\bs k\in\mathbb N^2}
\frac{1-\ee^{-\lambda_{\bs k}\Delta_N}}{\lambda_{\bs k}^{1+\alpha}}
\Biggr)^2
\\
&=O(\Delta_N^{\alpha} \lor \Delta_N^{2\alpha})
\\
&=O(\Delta_N^\alpha).
\end{align*}

We next prove \eqref{eq3-lem3-1}-\eqref{eq3-lem3-3} 
under (ii)-(iv) in Assumption \ref{asm}. 
We obtain from (ii), (iii) and the Schwarz inequality that 
\begin{align*}
&\sum_{i=1}^N\sum_{\bs k_1,\bs k_2\in\mathbb N^2}
\Bigl|\EE[A_{i,\bs k_1}A_{i,\bs k_2}]e_{\bs k_1}(y,z)e_{\bs k_2}(y,z)\Bigr|
\\
&\lesssim 
\sum_{i=1}^N\sum_{\bs k_1,\bs k_2\in\mathbb N^2}
(1-\ee^{-\lambda_{\bs k_1}\Delta_N})(1-\ee^{-\lambda_{\bs k_2}\Delta_N})
\\
&\qquad\qquad\times
\ee^{-\lambda_{\bs k_1}(i-1)\Delta_N }
\ee^{-\lambda_{\bs k_2}(i-1)\Delta_N }
\EE\bigl[
|\langle\xi,e_{\bs k_1}\rangle_\theta \langle\xi,e_{\bs k_2}\rangle_\theta|
\bigr]
\\
&=
\sum_{i=1}^N
\EE\Biggl[
\biggl(
\sum_{\bs k\in\mathbb N^2}
(1-\ee^{-\lambda_{\bs k}\Delta_N})
\ee^{-\lambda_{\bs k}(i-1)\Delta_N }
|\langle\xi,e_{\bs k}\rangle_\theta|
\biggr)^2
\Biggr]
\\
&\le
\Biggl(
\sum_{\bs k\in\mathbb N^2}
\lambda_{\bs k}^{1+\alpha}
\EE[\langle\xi,e_{\bs k}\rangle_\theta^2]
\Biggr)
\Biggl(
\sum_{\bs k\in\mathbb N^2}
\frac{(1-\ee^{-\lambda_{\bs k}\Delta_N})^2}{\lambda_{\bs k}^{1+\alpha}}
\sum_{i=1}^N
\ee^{-2\lambda_{\bs k}(i-1)\Delta_N }
\Biggr)
\\
&\lesssim
\sum_{\bs k\in\mathbb N^2}
\frac{1-\ee^{-\lambda_{\bs k}\Delta_N}}{\lambda_{\bs k}^{1+\alpha}}
\\
&=
O(\Delta_N^{\alpha})
\end{align*}
and
\begin{align*}
&\sup_{j\ge1}\sum_{i=1}^N\sum_{\bs k_1,\bs k_2\in\mathbb N^2}
\Bigl|\EE[A_{i,\bs k_1}A_{j,\bs k_2}]e_{\bs k_1}(y,z)e_{\bs k_2}(y,z)\Bigr|
\\
&\lesssim 
\sup_{j\ge1}\sum_{i=1}^N\sum_{\bs k_1,\bs k_2\in\mathbb N^2}
(1-\ee^{-\lambda_{\bs k_1}\Delta_N})(1-\ee^{-\lambda_{\bs k_2}\Delta_N})
\\
&\qquad\qquad\times
\ee^{-\lambda_{\bs k_1}(i-1)\Delta_N }
\ee^{-\lambda_{\bs k_2}(j-1)\Delta_N }
\EE\bigl[
|\langle\xi,e_{\bs k_1}\rangle_\theta \langle\xi,e_{\bs k_2}\rangle_\theta|
\bigr]
\\
&\le 
\sum_{\bs k_1,\bs k_2\in\mathbb N^2}
(1-\ee^{-\lambda_{\bs k_2}\Delta_N})
\EE\bigl[
|\langle\xi,e_{\bs k_1}\rangle_\theta \langle\xi,e_{\bs k_2}\rangle_\theta|
\bigr]
\\
&=
\EE\Biggl[
\biggl(
\sum_{\bs k\in\mathbb N^2}
|\langle\xi,e_{\bs k}\rangle_\theta|
\biggr)
\biggl(
\sum_{\bs k\in\mathbb N^2}
(1-\ee^{-\lambda_{\bs k}\Delta_N})
|\langle\xi,e_{\bs k}\rangle_\theta|
\biggr)
\Biggr]
\\
&\le
\Biggl(
\sum_{\bs k\in\mathbb N^2}
\frac{1}{\lambda_{\bs k}^{1+\alpha}}
\Biggr)^{1/2}
\Biggl(
\sum_{\bs k\in\mathbb N^2}
\frac{(1-\ee^{-\lambda_{\bs k}\Delta_N})^2}{\lambda_{\bs k}^{1+\alpha}}
\Biggr)^{1/2}
\Biggl(
\sum_{\bs k\in\mathbb N^2}
\lambda_{\bs k}^{1+\alpha}
\EE[\langle\xi,e_{\bs k}\rangle_\theta^2]
\Biggr)
\\
&=
O(\Delta_N^{\alpha/2}).
\end{align*}
Noting that for $\ell=3,4$,  
\begin{align*}
&\sum_{\bs k\in\mathbb N^2}
(1-\ee^{-\lambda_{\bs k}\Delta_N})^2
\EE\bigl[|\langle \xi,e_{\bs k}\rangle_\theta|^\ell\bigr]
\\
&\le
\sum_{\bs k\in\mathbb N^2}
(1-\ee^{-\lambda_{\bs k}\Delta_N})^2
\EE\bigl[\langle \xi,e_{\bs k}\rangle_\theta^4\bigr]^{1/2}
\EE\bigl[\langle \xi,e_{\bs k}\rangle_\theta^{2(\ell-2)}\bigr]^{1/2}
\\
&\lesssim
\sum_{\bs k\in\mathbb N^2}
\frac{(1-\ee^{-\lambda_{\bs k}\Delta_N})^2}{\lambda_{\bs k}^{1+\alpha}}
=
O(\Delta_N^{\alpha}),
\end{align*}
\begin{equation*}
\sum_{\bs k\in\mathbb N^2}
(1-\ee^{-\lambda_{\bs k}\Delta_N})
\EE\bigl[\langle \xi,e_{\bs k}\rangle_\theta^2\bigr]
\lesssim
\sum_{\bs k\in\mathbb N^2}
\frac{1-\ee^{-\lambda_{\bs k}\Delta_N}}{\lambda_{\bs k}^{1+\alpha}}
=O(\Delta_N^{\alpha}),
\end{equation*}
and that for $j=0,1$, 
\begin{align*}
&\sum_{\bs k\in\mathbb N^2}
(1-\ee^{-\lambda_{\bs k}\Delta_N})^j
\EE\bigl[|\langle \xi,e_{\bs k}\rangle_\theta|\bigr]
\\
&\le
\Biggl(
\sum_{\bs k\in\mathbb N^2}
\frac{(1-\ee^{-\lambda_{\bs k}\Delta_N})^{2j}}{\lambda_{\bs k}^{1+\alpha}}
\Biggr)^{1/2}
\Biggl(
\sum_{\bs k\in\mathbb N^2}
\lambda_{\bs k}^{1+\alpha}
\EE\bigl[\langle \xi,e_{\bs k}\rangle_\theta^2\bigr]
\Biggr)^{1/2}
\\
&=
\begin{cases}
O(1), & j=0,
\\
O(\Delta_N^{\alpha/2}), & j=1,
\end{cases}
\end{align*}
we obtain from (ii)-(iv) that
\begin{align*}
&\sum_{i,j=1}^N\sum_{\bs k_1,\ldots,\bs k_4\in\mathbb N^2}
\Bigl|\EE[A_{i,\bs k_1}A_{i,\bs k_2}A_{j,\bs k_3}A_{j,\bs k_4}]
e_{\bs k_1}(y,z)\cdots e_{\bs k_4}(y,z)\Bigr|
\\
&\lesssim
\sum_{\bs k_1,\ldots,\bs k_4\in\mathbb N^2}
(1-\ee^{-\lambda_{\bs k_1}\Delta_N})\cdots(1-\ee^{-\lambda_{\bs k_4}\Delta_N})
\EE[|\langle \xi,e_{\bs k_1}\rangle_\theta\cdots\langle \xi,e_{\bs k_4}\rangle_\theta|]
\\
&\qquad\qquad\times
\sum_{i=1}^N
\ee^{-(\lambda_{\bs k_1}+\lambda_{\bs k_2})(i-1)\Delta_N}
\sum_{j=1}^N
\ee^{-(\lambda_{\bs k_3}+\lambda_{\bs k_4})(j-1)\Delta_N}
\\
&\le
\sum_{\bs k_1,\ldots,\bs k_4\in\mathbb N^2}
\frac{(1-\ee^{-\lambda_{\bs k_1}\Delta_N})\cdots(1-\ee^{-\lambda_{\bs k_4}\Delta_N})}
{(1-\ee^{-(\lambda_{\bs k_1}+\lambda_{\bs k_2})\Delta_N})
(1-\ee^{-(\lambda_{\bs k_3}+\lambda_{\bs k_4})\Delta_N})}
\EE[|\langle \xi,e_{\bs k_1}\rangle_\theta\cdots\langle \xi,e_{\bs k_4}\rangle_\theta|]
\\
&\lesssim
\sum_{\bs k\in\mathbb N^2}
(1-\ee^{-\lambda_{\bs k}\Delta_N})^2
\EE[\langle \xi,e_{\bs k}\rangle_\theta^4]
\\
&\qquad+
\sum_{\bs k_1,\bs k_2\in\mathbb N^2}
(1-\ee^{-\lambda_{\bs k_1}\Delta_N})^2
\EE\bigl[|\langle \xi,e_{\bs k_1}\rangle_\theta|^3\bigr]
\EE[|\langle \xi,e_{\bs k_2}\rangle_\theta|]
\\
&\qquad+
\sum_{\bs k_1,\bs k_2\in\mathbb N^2}
(1-\ee^{-\lambda_{\bs k_1}\Delta_N})(1-\ee^{-\lambda_{\bs k_2}\Delta_N})
\EE\bigl[\langle \xi,e_{\bs k_1}\rangle_\theta^2\bigr]
\EE\bigl[\langle \xi,e_{\bs k_2}\rangle_\theta^2\bigr]
\\
&\qquad+
\sum_{\bs k_1,\ldots,\bs k_3\in\mathbb N^2}
(1-\ee^{-\lambda_{\bs k_1}\Delta_N})(1-\ee^{-\lambda_{\bs k_2}\Delta_N})
\EE\bigl[\langle \xi,e_{\bs k_1}\rangle_\theta^2\bigr]
\EE[|\langle \xi,e_{\bs k_2}\rangle_\theta|]
\EE[|\langle \xi,e_{\bs k_3}\rangle_\theta|]
\\
&\qquad+
\sum_{\bs k_1,\ldots,\bs k_4\in\mathbb N^2}
(1-\ee^{-\lambda_{\bs k_1}\Delta_N})(1-\ee^{-\lambda_{\bs k_2}\Delta_N})
\EE[|\langle \xi,e_{\bs k_1}\rangle_\theta|]
\cdots
\EE[|\langle \xi,e_{\bs k_4}\rangle_\theta|]
\\
&=
\sum_{\bs k\in\mathbb N^2}
(1-\ee^{-\lambda_{\bs k}\Delta_N})^2
\EE[\langle \xi,e_{\bs k}\rangle_\theta^4]
\\
&\qquad+
\Biggl(
\sum_{\bs k\in\mathbb N^2}
(1-\ee^{-\lambda_{\bs k}\Delta_N})^2
\EE\bigl[|\langle \xi,e_{\bs k}\rangle_\theta|^3\bigr]
\Biggr)
\Biggl(
\sum_{\bs k\in\mathbb N^2}
\EE[|\langle \xi,e_{\bs k}\rangle_\theta|]
\Biggr)
\\
&\qquad+
\Biggl(
\sum_{\bs k\in\mathbb N^2}
(1-\ee^{-\lambda_{\bs k}\Delta_N})
\EE\bigl[\langle \xi,e_{\bs k}\rangle_\theta^2\bigr]
\Biggr)^2
\\
&\qquad+
\Biggl(
\sum_{\bs k\in\mathbb N^2}
(1-\ee^{-\lambda_{\bs k}\Delta_N})
\EE\bigl[\langle \xi,e_{\bs k}\rangle_\theta^2\bigr]
\Biggr)
\Biggl(
\sum_{\bs k\in\mathbb N^2}
(1-\ee^{-\lambda_{\bs k}\Delta_N})
\EE[|\langle \xi,e_{\bs k}\rangle_\theta|]
\Biggr)
\Biggl(
\sum_{\bs k\in\mathbb N^2}
\EE[|\langle \xi,e_{\bs k}\rangle_\theta|]
\Biggr)
\\
&\qquad+
\Biggl(
\sum_{\bs k\in\mathbb N^2}
(1-\ee^{-\lambda_{\bs k}\Delta_N})
\EE[|\langle \xi,e_{\bs k}\rangle_\theta|]
\Biggr)^2
\Biggl(
\sum_{\bs k\in\mathbb N^2}
\EE[|\langle \xi,e_{\bs k}\rangle_\theta|]
\Biggr)^2
\\
&=
O(\Delta_N^{\alpha} \lor \Delta_N^{\alpha} \lor \Delta_N^{2\alpha}
\lor \Delta_N^{3\alpha/2} \lor \Delta_N^{\alpha})
\\
&=O(\Delta_N^{\alpha}).
\end{align*}
\end{proof}

\begin{lem}\label{lem4}
It holds that uniformly in $(y,z)\in D_\delta$,
\begin{align}\label{eq4-lem4-1}
\sum_{\bs k,\bs \ell\in\mathbb N^2}
\EE[B_{i,\bs k}^{Q_1}B_{i,\bs \ell}^{Q_1}]e_{\bs k}(y,z)e_{\bs \ell}(y,z)
=
\Delta_N^{\alpha}\frac{\sigma^2\Gamma(1-\alpha)}{4\pi\alpha\theta_2}
\ee^{-\frac{\theta_1}{\theta_2}y}\ee^{-\frac{\eta_1}{\theta_2}z}
+r_{N,i}+O(\Delta_N),
\end{align}
where $\sum_{i=1}^N |r_{N,i}|=O(\Delta_N^{\alpha})$.
Moreover, it holds that uniformly in $(y,z)\in D_\delta$,
\begin{align}\label{eq4-lem4-2}
\sup_{i\neq j}\sum_{j=1}^N
\sum_{\bs k,\bs \ell\in\mathbb N^2}
\Bigl|
\EE[B_{i,\bs k}^{Q_1}B_{j,\bs \ell}^{Q_1}]e_{\bs k}(y,z)e_{\bs \ell}(y,z)
\Bigr|
=O(1).
\end{align}
\end{lem}

\begin{proof}
First of all, we prove \eqref{eq4-lem4-1}. 
Noting that
$\EE[B_{1,i,\bs k}^{Q_1}B_{2,i,\bs \ell}^{Q_1}]=0$, 
$\EE[B_{1,i,\bs k}^{Q_1}B_{1,i,\bs \ell}^{Q_1}]
=\EE[B_{2,i,\bs k}^{Q_1}B_{2,i,\bs \ell}^{Q_1}]=0$ 
($\bs k\neq\bs \ell$),
\begin{align*}
\EE[(B_{1,i,\bs k}^{Q_1})^2]
&=
\frac{\sigma^2(1-\ee^{-\lambda_{\bs k}\Delta_N})^2}{\lambda_{\bs k}^{\alpha}}
\int_0^{(i-1)\Delta_N}
\ee^{-2\lambda_{\bs k}((i-1)\Delta_N-s)}\dd s
\\
&=
\frac{\sigma^2(1-\ee^{-\lambda_{\bs k}\Delta_N})^2}{2\lambda_{\bs k}^{1+\alpha}}
(1-\ee^{-2\lambda_{\bs k}(i-1)\Delta_N}),
\\
\EE[(B_{2,i,\bs k}^{Q_1})^2]
&=
\frac{\sigma^2}{\lambda_{\bs k}^{\alpha}}
\int_{(i-1)\Delta_N}^{i\Delta_N}
\ee^{-2\lambda_{\bs k}(i\Delta_N-s)}\dd s
\\
&=
\frac{\sigma^2(1-\ee^{-\lambda_{\bs k}\Delta_N})(1+\ee^{-\lambda_{\bs k}\Delta_N})}
{2\lambda_{\bs k}^{1+\alpha}},
\end{align*}
we have
\begin{align}
&\sum_{\bs k,\bs \ell\in\mathbb N^2}
\EE[B_{i,\bs k}^{Q_1}B_{i,\bs \ell}^{Q_1}]e_{\bs k}(y,z)e_{\bs \ell}(y,z)
\nonumber
\\
&=
\sum_{\bs k\in\mathbb N^2}
\bigl(\EE[(B_{1,i,\bs k}^{Q_1})^2]+\EE[(B_{2,i,\bs k}^{Q_1})^2]\bigr)e_{\bs k}^2(y,z)
\nonumber
\\
&=
\sigma^2
\sum_{\bs k\in\mathbb N^2}
\frac{1-\ee^{-\lambda_{\bs k}\Delta_N}}{2\lambda_{\bs k}^{1+\alpha}}
\bigl(
(1-\ee^{-\lambda_{\bs k}\Delta_N})
(1-\ee^{-2\lambda_{\bs k}(i-1)\Delta_N})
+(1+\ee^{-\lambda_{\bs k}\Delta_N})
\bigr)
e_{\bs k}^2(y,z)
\nonumber
\\
&=
\sigma^2
\sum_{\bs k\in\mathbb N^2}
\frac{1-\ee^{-\lambda_{\bs k}\Delta_N}}{\lambda_{\bs k}^{1+\alpha}}
\biggl(
1-\frac{1-\ee^{-\lambda_{\bs k}\Delta_N}}{2}
\ee^{-2\lambda_{\bs k}(i-1)\Delta_N }
\biggr)
e_{\bs k}^2(y,z)
\nonumber
\\
&=
\sigma^2\sum_{\bs k\in\mathbb N^2}
\frac{1-\ee^{-\lambda_{\bs k}\Delta_N}}{\lambda_{\bs k}^{1+\alpha}}e_{\bs k}^2(y,z)
+s_{N,i},
\label{eq4-0101}
\end{align}
where
\begin{align*}
s_{N,i}
=-\frac{\sigma^2}{2}\sum_{\bs k\in\mathbb N^2}
\frac{(1-\ee^{-\lambda_{\bs k}\Delta_N})^2}{\lambda_{\bs k}^{1+\alpha}}
\ee^{-2\lambda_{\bs k}(i-1)\Delta_N }e_{\bs k}^2(y,z).
\end{align*}
It follows from Lemma \ref{lem2} on fixed $\delta>0$ 
that uniformly in $(y,z)\in D_\delta$,
\begin{align}
\sum_{\bs k\in\mathbb N^2}
\frac{1-\ee^{-\lambda_{\bs k}\Delta_N}}{\lambda_{\bs k}^{1+\alpha}}
e_{\bs k}^2(y,z)
&=
\Delta_N^{1+\alpha}
\sum_{\bs k\in\mathbb N^2}
f_\alpha(\lambda_{\bs k}\Delta_N)e_{\bs k}^2(y,z)
\nonumber
\\
&=
\frac{\Delta_N^{\alpha}\Gamma(1-\alpha)}{4\pi\alpha\theta_2}
\ee^{-\frac{\theta_1}{\theta_2}y}\ee^{-\frac{\eta_1}{\theta_2}z}
+O(\Delta_N),
\label{eq4-0102}
\\
\sum_{\bs k\in\mathbb N^2}
\frac{(1-\ee^{-\lambda_{\bs k}\Delta_N})^2}{\lambda_{\bs k}^{1+\alpha}}
\ee^{-2\lambda_{\bs k}(i-1)\Delta_N }e_{\bs k}^2(y,z)
&=
\Delta_N^{1+\alpha}
\sum_{\bs k\in\mathbb N^2}
f_{\alpha,2(i-1)}(\lambda_{\bs k}\Delta_N)e_{\bs k}^2(y,z)
\nonumber
\\
&=
\frac{\Delta_N^{\alpha}\Gamma(1-\alpha)}{4\pi\alpha\theta_2}J(i)
\ee^{-\frac{\theta_1}{\theta_2}y}\ee^{-\frac{\eta_1}{\theta_2}z}
+O(\Delta_N)
\nonumber
\\
&=:
r_{N,i}+O(\Delta_N),
\label{eq4-0103}
\end{align}
where
$J(i)=-(2(i-1))^\alpha+2(1+2(i-1))^\alpha-(2+2(i-1))^\alpha$.
Noting that for $|x|<1$, 
\begin{align*}
(1+x)^\alpha
&=\sum_{m\ge0}{\alpha\choose m}x^m, 
\ 
\text{where }
{\alpha\choose m}
=\frac{\alpha(\alpha-1)\cdots(\alpha-m+1)}{m!},
\end{align*}
and that
\begin{align*}
\biggl|{\alpha\choose m}\biggr|
&=
\frac{\alpha(1-\alpha)\cdots(m-1-\alpha)}{m!}
\le
\frac{\alpha}{m}
\le1,
\end{align*}
we obtain that for $i\ge3$, 
\begin{align*}
J(i)
&=
\bigl(2(i-1)\bigr)^\alpha
\Biggl\{
2\biggl(1+\frac{1}{2(i-1)}\biggr)^\alpha-1-\biggl(1+\frac{1}{i-1}\biggr)^\alpha
\Biggr\}
\\
&=
\bigl(2(i-1)\bigr)^\alpha
\Biggl\{
2\Biggl(
1+\frac{\alpha}{2(i-1)}
+\sum_{m\ge2}{\alpha\choose m}\biggl(\frac{1}{2(i-1)}\biggr)^m
\Biggr)
\\
&\qquad\qquad
-1-
\Biggl(
1+\frac{\alpha}{i-1}
+\sum_{m\ge2}{\alpha\choose m}\biggl(\frac{1}{i-1}\biggr)^m
\Biggr)
\Biggr\}
\\
&=
\bigl(2(i-1)\bigr)^\alpha
\sum_{m\ge2}{\alpha\choose m}\biggl(\frac{1}{2^{m-1}}-1\biggr)
\frac{1}{(i-1)^m}
\\
&\lesssim
\frac{1}{(i-1)^{2-\alpha}},
\end{align*}
that is,
$\sum_{i\ge1} J(i)< \infty$.
Therefore, we obtain that
\begin{equation*}
\sum_{i=1}^N |r_{N,i}|
\lesssim
\Delta_N^{\alpha}\sum_{i=1}^N J(i)
\lesssim
\Delta_N^{\alpha},
\end{equation*}
and from \eqref{eq4-0101}-\eqref{eq4-0103}, we have \eqref{eq4-lem4-1}. 

Next, we prove \eqref{eq4-lem4-2}.
Without loss of generality, consider 
\begin{align*}
\sup_{i< j}\sum_{j=1}^N
\sum_{\bs k,\bs \ell\in\mathbb N^2}
\Bigl|
\EE[B_{i,\bs k}^{Q_1}B_{j,\bs \ell}^{Q_1}]e_{\bs k}(y,z)e_{\bs \ell}(y,z)
\Bigr|.
\end{align*}
Since for $i<j$, 
\begin{align*}
\EE[B_{1,i,\bs k}^{Q_1}B_{1,j,\bs k}^{Q_1}]
&=
\frac{\sigma^2(1-\ee^{-\lambda_{\bs k}\Delta_N})^2}{\lambda_{\bs k}^{\alpha}}
\int_0^{(i-1)\Delta_N}\ee^{-\lambda_{\bs k}((i+j-2)\Delta_N-2s)}\dd s
\\
&=
\frac{\sigma^2(1-\ee^{-\lambda_{\bs k}\Delta_N})^2}{2\lambda_{\bs k}^{1+\alpha}}
(\ee^{-\lambda_{\bs k}(j-i)\Delta_N}-\ee^{-\lambda_{\bs k}(i+j-2)\Delta_N}),
\\
\EE[B_{1,i,\bs k}^{Q_1}B_{2,j,\bs k}^{Q_1}]&
=\EE[B_{2,i,\bs k}^{Q_1}B_{2,j,k}^{Q_1}]=0,
\\
\EE[B_{2,i,\bs k}^{Q_1}B_{1,j,\bs k}^{Q_1}]
&=
-\frac{\sigma^2(1-\ee^{-\lambda_{\bs k}\Delta_N})}{\lambda_{\bs k}^{\alpha}}
\int_{(i-1)\Delta_N}^{i\Delta_N}\ee^{-\lambda_{\bs k}((i+j-1)\Delta_N-2s)}\dd s
\\
&=
-\frac{\sigma^2(1-\ee^{-\lambda_{\bs k}\Delta_N})^2}{2\lambda_{\bs k}^{1+\alpha}}
\ee^{-\lambda_{\bs k}(j-i-1)\Delta_N}(1+\ee^{-\lambda_{\bs k}\Delta_N}),
\end{align*}
it holds that
\begin{align*}
&\sum_{\bs k,\bs \ell\in\mathbb N^2}
\Bigl|
\EE[B_{i,\bs k}^{Q_1}B_{j,\bs \ell}^{Q_1}]
e_{\bs k}(y,z)e_{\bs \ell}(y,z)
\Bigr|
\\
&=\sum_{\bs k\in\mathbb N^2}
\Bigl|
\EE[B_{i,\bs k}^{Q_1}B_{j,\bs k}^{Q_1}]
e_{\bs k}^2(y,z)
\Bigr|
\\
&=\sum_{\bs k\in\mathbb N^2}
\Bigl|
\EE[B_{1,i,\bs k}^{Q_1}B_{1,j,\bs k}^{Q_1}]
+\EE[B_{1,i,\bs k}^{Q_1}B_{2,j,\bs k}^{Q_1}]
+\EE[B_{2,i,\bs k}^{Q_1}B_{1,j,\bs k}^{Q_1}]
+\EE[B_{2,i,\bs k}^{Q_1}B_{2,j,\bs k}^{Q_1}]
\Bigr|
e_{\bs k}^2(y,z)
\\
&=
\sum_{\bs k\in\mathbb N^2}
\frac{\sigma^2(1-\ee^{-\lambda_{\bs k}\Delta_N})^2}{2\lambda_{\bs k}^{1+\alpha}}
\bigl|
\ee^{-\lambda_{\bs k}(j-i)\Delta_N}-\ee^{-\lambda_{\bs k}(i+j-2)\Delta_N}
-\ee^{-\lambda_{\bs k}(j-i-1)\Delta_N}(1+\ee^{-\lambda_{\bs k}\Delta_N})
\bigr|
e_{\bs k}^2(y,z)
\\
&\le
\sum_{\bs k\in\mathbb N^2}
\frac{\sigma^2(1-\ee^{-\lambda_{\bs k}\Delta_N})^2}{2\lambda_{\bs k}^{1+\alpha}}
\ee^{-\lambda_{\bs k}(j-i-1)\Delta_N}
e_{\bs k}^2(y,z)
+
\sum_{\bs k\in\mathbb N^2}
\frac{\sigma^2(1-\ee^{-\lambda_{\bs k}\Delta_N})^2}{2\lambda_{\bs k}^{1+\alpha}}
\ee^{-\lambda_{\bs k}(i+j-2)\Delta_N}
e_{\bs k}^2(y,z)
\\
&=:
\sum_{\bs k\in\mathbb N^2}
\frac{\sigma^2(1-\ee^{-\lambda_{\bs k}\Delta_N})^2}{2\lambda_{\bs k}^{1+\alpha}}
\ee^{-\lambda_{\bs k}(j-i-1)\Delta_N}
e_{\bs k}^2(y,z)
+s_{i,j}.
\end{align*}
Noting that
\begin{align*}
\sum_{\bs k\in\mathbb N^2}
\frac{(1-\ee^{-\lambda_{\bs k}\Delta_N})^2}{\lambda_{\bs k}^{1+\alpha}}
\ee^{-\lambda_{\bs k}(j-i-1)\Delta_N}
e_{\bs k}^2(y,z)
=
\frac{\Delta_N^{\alpha}}{4\pi\theta_2}
\ee^{-\frac{\theta_1}{\theta_2}y}\ee^{-\frac{\eta_1}{\theta_2}z}
J(j-i)
+O(\Delta_N),
\end{align*}
\begin{equation*}
\sup_{i\ge1}\sum_{j-i\ge1}J(j-i)
\le
\sum_{k\ge1}J(k)<\infty
\end{equation*}
and
\begin{align*}
(0\le)\ 
\sup_{i\ge1}\sum_{j=1}^N s_{i,j}
&=
\sup_{i\ge1}\sum_{j=1}^N
\sum_{\bs k\in\mathbb N^2}
\frac{(1-\ee^{-\lambda_{\bs k}\Delta_N})^2}{\lambda_{\bs k}^{1+\alpha}}
\ee^{-\lambda_{\bs k}(i+j-2)\Delta_N}
e_{\bs k}^2(y,z)
\\
&\lesssim
\sum_{\bs k\in\mathbb N^2}
\frac{1-\ee^{-\lambda_{\bs k}\Delta_N}}{\lambda_{\bs k}^{1+\alpha}}
=O(\Delta_N^{\alpha}),
\end{align*}
we see that
\begin{align*}
&\sup_{i< j}\sum_{j=1}^N
\sum_{\bs k,\bs \ell\in\mathbb N^2}
\Bigl|
\EE[B_{i,\bs k}^{Q_1}B_{j,\bs \ell}^{Q_1}]e_{\bs k}(y_1,z_1)e_{\bs \ell}(y_2,z_2)
\Bigr|
\\
&\le
\sup_{i\ge1}\sum_{j-i\ge1}^N
\sum_{\bs k\in\mathbb N^2}
\frac{\sigma^2(1-\ee^{-\lambda_{\bs k}\Delta_N})^2}{2\lambda_{\bs k}^{1+\alpha}}
\ee^{-\lambda_{\bs k}(j-i-1)\Delta_N}
e_{\bs k}^2(y,z)
+\sup_{i\ge1}\sum_{j=1}^N
s_{i,j}
\\
&=O(1).
\end{align*}
\end{proof}

\subsection{Proofs of Proposition \ref{prop1}, Theorems \ref{th1} and \ref{th2}}
\begin{proof}[\bf Proof of Proposition \ref{prop1}]
Since $\delta$ is fixed, it holds from 
$\EE[A_{i,\bs k}B_{i,\bs \ell}^{Q_1}]=0$, Lemmas \ref{lem3} and \ref{lem4} that
uniformly in $(y,z)\in D_\delta$,
\begin{align*}
\EE[(\Delta_i X^{Q_1})^2(y,z)]
&=\sum_{\bs k,\bs \ell\in\mathbb N^2}
\EE[\Delta_ix_{\bs k}^{Q_1}\Delta_ix_{\bs \ell}^{Q_1}]
e_{\bs k}(y,z)e_{\bs \ell}(y,z)
\\
&=
\sum_{\bs k,\bs \ell\in\mathbb N^2}
\EE[A_{i,\bs k}A_{i,\bs \ell}]e_{\bs k}(y,z)e_{\bs \ell}(y,z)
+
\sum_{\bs k,\bs \ell\in\mathbb N^2}
\EE[B_{i,\bs k}^{Q_1}B_{i,\bs \ell}^{Q_1}]e_{\bs k}(y,z)e_{\bs \ell}(y,z)
\\
&=
\Delta_N^\alpha\frac{\sigma^2\Gamma(1-\alpha)}{4\pi\alpha\theta_2}
\ee^{-\frac{\theta_1}{\theta_2}y}\ee^{-\frac{\eta_1}{\theta_2}z}
+r_{N,i}+O(\Delta_N),
\end{align*}
where
$\sum_{i=1}^N |r_{N,i}|=O(\Delta_N^{\alpha})$, and
this completes the proof.
\end{proof}

\begin{proof}[\bf Proof of Theorem \ref{th1}]
Let
$\nu=(s,\kappa,\eta)^\TT$, $\nu^*=(s^*,\kappa^*,\eta^*)^\TT$, 
and $\hat\nu=(\hat s,\hat \kappa,\hat \eta)^\TT$.
Set 
$g_\nu(y,z)=\frac{\Gamma(1-\alpha)}{4\pi\alpha}s\, \ee^{-(\kappa y+\eta z)}$, 
$\xi_N(y,z,\nu)=Z_{N}^{Q_1}(y,z)-g_\nu(y,z)$ and $\xi_N^*(y,z)=\xi_N(y,z,\nu^*)$.
Note that
\begin{align}
\partial{U}_{N,m}^{(1)}(\nu)
&=-
\frac{\Gamma(1-\alpha)}{2\pi\alpha}
\sum_{j_1=1}^{m_1}\sum_{j_2=1}^{m_2}
\xi_{N}(\widetilde y_{j_1},\widetilde z_{j_2},\nu)
\ee^{-(\kappa \widetilde y_{j_1}+\eta \widetilde z_{j_2})}
(1, -s\widetilde y_{j_1}, -s\widetilde z_{j_2}),
\nonumber
\\
\partial^2{U}_{N,m}^{(1)}(\nu)
&=
\sum_{j_1=1}^{m_1}\sum_{j_2=1}^{m_2}
\xi_{N}(\widetilde y_{j_1},\widetilde z_{j_2},\nu)
K_\nu(\widetilde y_{j_1},\widetilde z_{j_2})
+\sum_{j_1=1}^{m_1}\sum_{j_2=1}^{m_2}
L_\nu(\widetilde y_{j_1},\widetilde z_{j_2}),
\label{eq5-0102}
\end{align}
where
\begin{align*}
K_\nu(y,z)
&=
-\frac{\Gamma(1-\alpha)}{2\pi\alpha}
\ee^{-(\kappa y+\eta z)}
\begin{pmatrix}
0 & -y & -z\\
-y & sy^2 & syz \\
-z & syz & sz^2
\end{pmatrix},
\\
L_\nu(y,z)
&=
\frac{\Gamma(1-\alpha)^2}{8(\pi\alpha)^2}
\ee^{-2(\kappa y+\eta z)}
\begin{pmatrix}
1 & -sy & -sz \\
-sy & s^2y^2 & s^2yz \\
-sz & s^2yz & s^2z^2
\end{pmatrix}.
\end{align*}
By the Taylor expansion,
\begin{align*}
-\partial{U}_{N,m}^{(1)}(\nu^*)^\TT
=\partial{U}_{N,m}^{(1)}(\hat\nu)^\TT-\partial{U}_{N,m}^{(1)}(\nu^*)^\TT
=\int_0^1\partial^2{U}_{N,m}^{(1)}(\nu^*+u(\hat\nu-\nu^*))\dd u(\hat\nu-\nu^*),
\end{align*}
that is, for $\gamma<\frac{1}{2}\land(1-\alpha)-\frac{\rho}{2}$,
\begin{align*}
-\frac{N^{\gamma}}{m^{1/2}}\partial{U}_{N,m}^{(1)}(\nu^*)^\TT
=\frac{1}{m}\int_0^1\partial^2{U}_{N,m}^{(1)}(\nu^*+u(\hat\nu-\nu^*))\dd u\, 
m^{1/2}N^{\gamma}(\hat\nu-\nu^*).
\end{align*}
Let
\begin{align*}
L(\nu)=\frac{1}{|D_\delta|}\iint_{D_\delta} L_{\nu}(y,z)\dd y\dd z.
\end{align*}
Note that by computing the determinant of $L(\nu)$ and the Schwartz inequality,
\begin{align*}
\det L(\nu)
&=
\biggl(\frac{\Gamma(1-\alpha)^2}{8(\pi\alpha)^2|D_\delta|}\biggr)^3
\int_{\delta}^{1-\delta}\ee^{-2\kappa y}\dd y
\int_{\delta}^{1-\delta}\ee^{-2\eta z}\dd z
\\
&\qquad\times
\Biggl(
\int_{\delta}^{1-\delta}\ee^{-2\kappa y}\dd y
\int_{\delta}^{1-\delta}\ee^{-2\kappa y}s^2y^2\dd y
-\biggl(
\int_{\delta}^{1-\delta}\ee^{-2\kappa y}sy\dd y
\biggr)^2
\Biggr)
\\
&\qquad\times
\Biggl(
\int_{\delta}^{1-\delta}\ee^{-2\eta z}\dd z
\int_{\delta}^{1-\delta}\ee^{-2\eta z}s^2z^2\dd z
-\biggl(
\int_{\delta}^{1-\delta}\ee^{-2\eta z}sz\dd z
\biggr)^2
\Biggr)
\\
&>0,
\end{align*}
and $L(\nu)$ is non-singular. 
To complete the proof, it suffice to show that
\begin{equation}\label{eq5-0103}
\hat\nu\pto\nu^*,
\end{equation}
\begin{equation}\label{eq5-0104}
\frac{\partial^2{U}_{N,m}^{(1)}(\nu^*)}{m}
\pto L(\nu^*),
\end{equation}
\begin{equation}\label{eq5-0105}
\frac{1}{m}\sup_{|\nu-\nu^*|\le\epsilon_N}
|\partial^2{U}_{N,m}^{(1)}(\nu)-\partial^2{U}_{N,m}^{(1)}(\nu^*)|\pto0
\quad \text{for } \epsilon_N\downarrow0,
\end{equation}
\begin{equation}\label{eq5-0106}
-\frac{N^{\gamma}}{m^{1/2}}\partial{U}_{N,m}^{(1)}(\nu^*)^\TT \pto 0. 
\end{equation}

First, note that 
\begin{align}
\sup_{(y,z)\in D_\delta}\EE[Z_{N}^{Q_1}(y,z)]
&\lesssim 1,
\label{eq5-0201}
\\
\sup_{(y,z)\in D_\delta}\EE[\xi_N^*(y,z)^2]
&=O(\Delta_N^{1\land2(1-\alpha)}).
\label{eq5-0202}
\end{align}
We can see from Proposition \ref{prop1} that \eqref{eq5-0201} holds.

\textit{Proof of \eqref{eq5-0202}. }
It follows that
\begin{align*}
\sup_{(y,z)\in D_\delta}\EE[\xi_N^*(y,z)^2]
\le
\sup_{(y,z)\in D_\delta}\VV[Z_N^{Q_1}(y,z)]
+\sup_{(y,z)\in D_\delta}\EE[\xi_N^*(y,z)]^2
\end{align*}
and from Proposition \ref{prop1}, one has
$\sup_{(y,z)\in D_\delta}\EE[\xi_N^*(y,z)]^2=
O(\Delta_N^{2(1-\alpha)})$.
It also follows from the Isserlis' theorem that
\begin{align*}
&\VV[Z_N^{Q_1}(y,z)]
=\EE[Z_N^{Q_1}(y,z)^2]-\EE[Z_N^{Q_1}(y,z)]^2
\\
&=\biggl(\frac{1}{N\Delta_N^{\alpha}}\biggr)^2
\sum_{i,j=1}^N
\Bigl(
\EE[(\Delta_i X^{Q_1})^2(y,z)(\Delta_j X^{Q_1})^2(y,z)]
-\EE[(\Delta_i X^{Q_1})^2(y,z)]\EE[(\Delta_j X^{Q_1})^2(y,z)]
\Bigr)
\\
&=\biggl(\frac{1}{N\Delta_N^{\alpha}}\biggr)^2
\sum_{i,j=1}^N
\sum_{\bs k_1,\ldots,\bs k_4\in\mathbb N^2}
\Bigl(
\EE[A_{i,\bs k_1}A_{i,\bs k_2}A_{j,\bs k_3}A_{j,\bs k_4}]
+2\EE[A_{i,\bs k_1}A_{i,\bs k_2}]\EE[B_{j,\bs k_3}^{Q_1}B_{j,\bs k_4}^{Q_1}]
\nonumber
\\
&\qquad\qquad\qquad
+4\EE[A_{i,\bs k_1}A_{j,\bs k_2}]\EE[B_{i,\bs k_3}^{Q_1}B_{j,\bs k_4}^{Q_1}]
+\EE[B_{i,\bs k_1}^{Q_1}B_{i,\bs k_2}^{Q_1}B_{j,\bs k_3}^{Q_1}B_{j,\bs k_4}^{Q_1}]
\nonumber
\\
&\qquad\qquad\qquad
-\EE[A_{i,\bs k_1}A_{i,\bs k_2}]\EE[A_{j,\bs k_3}A_{j,\bs k_4}]
-2\EE[A_{i,\bs k_1}A_{j,\bs k_2}]\EE[B_{i,\bs k_3}^{Q_1}B_{j,\bs k_4}^{Q_1}]
\nonumber
\\
&\qquad\qquad\qquad-
\EE[B_{i,\bs k_1}^{Q_1}B_{i,\bs k_2}^{Q_1}]\EE[B_{j,\bs k_3}^{Q_1}B_{j,\bs k_4}^{Q_1}]
\Bigr)
e_{\bs k_1}(y,z)\cdots e_{\bs k_4}(y,z)
\\
&\lesssim
\biggl(\frac{1}{N\Delta_N^{\alpha}}\biggr)^2
\sum_{i,j=1}^N
\sum_{\bs k_1,\ldots,\bs k_4\in\mathbb N^2}
\Bigl(
\EE[A_{i,\bs k_1}A_{i,\bs k_2}A_{j,\bs k_3}A_{j,\bs k_4}]
+\EE[A_{i,\bs k_1}A_{i,\bs k_2}]\EE[A_{j,\bs k_3}A_{j,\bs k_4}]
\nonumber
\\
&\qquad\qquad\qquad
+\EE[A_{i,\bs k_1}A_{i,\bs k_2}]\EE[B_{j,\bs k_3}^{Q_1}B_{j,\bs k_4}^{Q_1}]
\Bigr)
e_{\bs k_1}(y,z)\cdots e_{\bs k_4}(y,z)
\nonumber
\\
&\qquad+
\biggl(\frac{1}{N\Delta_N^{\alpha}}\biggr)^2
\sum_{i=1}^N
\sum_{\bs k_1,\ldots,\bs k_4\in\mathbb N^2}
\EE[B_{i,\bs k_1}^{Q_1}B_{i,\bs k_2}^{Q_1}]\EE[B_{i,\bs k_3}^{Q_1}B_{i,\bs k_4}^{Q_1}]
e_{\bs k_1}(y,z)\cdots e_{\bs k_4}(y,z)
\nonumber
\\
&\qquad+
\biggl(\frac{1}{N\Delta_N^{\alpha}}\biggr)^2
\sum_{i\neq j}
\sum_{\bs k_1,\ldots,\bs k_4\in\mathbb N^2}
\Bigl(
\EE[A_{i,\bs k_1}A_{j,\bs k_2}]\EE[B_{i,\bs k_3}^{Q_1}B_{j,\bs k_4}^{Q_1}]
\nonumber
\\
&\qquad\qquad\qquad+
\EE[B_{i,\bs k_1}^{Q_1}B_{j,\bs k_2}^{Q_1}]\EE[B_{i,\bs k_3}^{Q_1}B_{j,\bs k_4}^{Q_1}]
\Bigr)
e_{\bs k_1}(y,z)\cdots e_{\bs k_4}(y,z)
\\
&=:S_1+S_2+S_3.
\end{align*}
From Lemmas \ref{lem3} and \ref{lem4}, we have 
\begin{align*}
S_1
&\le
\biggl(\frac{1}{N\Delta_N^{\alpha}}\biggr)^2
\Biggl\{
\sum_{i,j=1}^N
\sum_{\bs k_1,\ldots,\bs k_4\in\mathbb N^2}
\Bigl|
\EE[A_{i,\bs k_1}A_{i,\bs k_2}A_{j,\bs k_3}A_{j,\bs k_4}]
e_{\bs k_1}(y,z)\cdots e_{\bs k_4}(y,z)
\Bigr|
\\
&\qquad+
\Biggl(
\sum_{i=1}^N
\sum_{\bs k_1,\bs k_2\in\mathbb N^2}
\Bigl|
\EE[A_{i,\bs k_1}A_{i,\bs k_2}]
e_{\bs k_1}(y,z)e_{\bs k_2}(y,z)
\Bigr|
\Biggr)^2
\\
&\qquad+
\Biggl(
\sum_{i=1}^N
\sum_{\bs k_1,\bs k_2\in\mathbb N^2}
\Bigl|
\EE[A_{i,\bs k_1}A_{i,\bs k_2}]
e_{\bs k_1}(y,z)e_{\bs k_2}(y,z)
\Bigr|
\Biggr)
\\
&\qquad\qquad\times
\Biggl(
\sum_{i=1}^N
\sum_{\bs k_1,\bs k_2\in\mathbb N^2}
\Bigl|
\EE[B_{i,\bs k_1}^{Q_1}B_{i,\bs k_2}^{Q_1}]
e_{\bs k_1}(y,z)e_{\bs k_2}(y,z)
\Bigr|
\Biggr)
\Biggr\}
\\
&=O\Biggl(
\biggl(\frac{1}{N\Delta_N^\alpha}\biggr)^2
(\Delta_N^\alpha \lor \Delta_N^{2\alpha} \lor \Delta_N^{2\alpha-1})
\Biggr)
=O(\Delta_N),
\\
S_2
&=
\biggl(\frac{1}{N\Delta_N^{\alpha}}\biggr)^2
\sum_{i=1}^N
\Biggl(
\sum_{\bs k_1,\bs k_2\in\mathbb N^2}
\EE[B_{i,\bs k_1}^{Q_1}B_{i,\bs k_2}^{Q_1}]
e_{\bs k_1}(y,z)e_{\bs k_2}(y,z)
\Biggr)^2
\\
&=O\Biggl(
\biggl(\frac{1}{N\Delta_N^\alpha}\biggr)^2
\Delta_N^{2\alpha-1}
\Biggr)
=O(\Delta_N),
\\
S_3
&\le 
\biggl(\frac{1}{N\Delta_N^{\alpha}}\biggr)^2
\Biggl\{
\Biggl(
\sup_{j\ge1}\sum_{i=1}^N
\sum_{\bs k_1,\bs k_2\in\mathbb N^2}
\Bigl|
\EE[A_{i,\bs k_1}A_{j,\bs k_2}]
e_{\bs k_1}(y,z)e_{\bs k_2}(y,z)
\Bigr|
\Biggr)
\\
&\qquad\qquad\qquad\times
\Biggl(
\sup_{i\neq j}
\sum_{j=1}^N
\sum_{\bs k_1,\bs k_2\in\mathbb N^2}
\Bigl|
\EE[B_{i,\bs k_1}^{Q_1}B_{j,\bs k_2}^{Q_1}]
e_{\bs k_1}(y,z)e_{\bs k_2}(y,z)
\Bigr|
\Biggr)
\\
&\qquad\qquad+ 
\Biggl(
\sup_{j\neq i}\sum_{i=1}^N
\sum_{\bs k_1,\bs k_2\in\mathbb N^2}
\Bigl|
\EE[B_{i,\bs k_1}^{Q_1}B_{j,\bs k_2}^{Q_1}]e_{\bs k_1}(y,z)e_{\bs k_2}(y,z)
\Bigr|
\Biggr)^2
\Biggr\}
\\
&=O\Biggl(
\biggl(\frac{1}{N\Delta_N^\alpha}\biggr)^2
(\Delta_N^{\alpha/2} \lor 1)
\Biggr)
=O(\Delta_N^{2(1-\alpha)}),
\end{align*}
and $\sup_{(y,z)\in D_\delta}\VV[Z_N^{Q_1}(y,z)]
=O(\Delta_N^{1\land2(1-\alpha)})$. 
Therefore, we obtain \eqref{eq5-0202}.

Note also that 
for any continuous function $h(y,z)$ with respect to $(y,z)\in D_\delta$,
\begin{align}
\frac{1}{m}\sum_{j_1=1}^{m_1}\sum_{j_2=1}^{m_2} h(\widetilde y_{j_1},\widetilde z_{j_2})
\to
\frac{1}{|D_\delta|}\iint_{D_\delta}h(y,z)\dd y\dd z,
\label{eq5-0203}
\end{align}
and that for any continuous function $h_\nu(y,z)$ with respect to 
$(y,z,\nu)\in D_\delta\times\Xi_1$,
\begin{align}
\sup_{\nu\in\Xi_1}
\Biggl|
\frac{1}{m}\sum_{j_1=1}^{m_1}\sum_{j_2=1}^{m_2} 
h_\nu(\widetilde y_{j_1},\widetilde z_{j_2})
-\frac{1}{|D_\delta|}\iint_{D_\delta}h_\nu(y,z)\dd y\dd z
\Biggr|
\to0.
\label{eq5-0204}
\end{align}

We will show \eqref{eq5-0103}-\eqref{eq5-0106}. 

\textit{Proof of \eqref{eq5-0103}. }
Let
\begin{align*}
U(\nu,\nu^*)
=\frac{1}{|D_\delta|}
\iint_{D_\delta}
(g_\nu(y,z)-g_{\nu^*}(y,z))^2 \dd y\dd z.
\end{align*}
For the proof of the consistency, it is enough to show that
\begin{itemize}
\item[(i)]
If $U(\nu,\nu^*)=0$, then $\nu=\nu^*$,

\item[(ii)]
$\displaystyle\sup_{\nu\in\Xi_1}
\biggl|\frac{U_{N,m}^{(1)}(\nu)}{m}-U(\nu,\nu^*)\biggr|\pto0$.
\end{itemize}

If $U(\nu,\nu^*)=0$, then $g_\nu(y,z)=g_{\nu^*}(y,z)$ for any $(y,z)\in D_\delta$. 
It implies that
$\ee^{-(\kappa y+\eta z)}s=\ee^{-(\kappa^* y+\eta^* z)}s^*$
for any $(y,z)\in D_\delta$, that is,
$\nu=\nu^*$.

Let $h_\nu(y,z)=g_{\nu^*}(y,z)-g_\nu(y,z)$. 
It follows from $\xi_N(y,z,\nu)=\xi_N^*(y,z)+h_\nu(y,z)$ that
\begin{align}
&\sup_{\nu\in\Xi_1}
\biggl|\frac{U_{N,m}^{(1)}(\nu)}{m}-U(\nu,\nu^*)\biggr|
\nonumber
\\
&=
\sup_{\nu\in\Xi_1}
\Biggl|
\frac{1}{m}
\sum_{j_1=1}^{m_1}\sum_{j_2=1}^{m_2}
\xi_{N}(\widetilde y_{j_1},\widetilde z_{j_2},\nu)^2
-\frac{1}{|D_\delta|}
\iint_{D_\delta}h_\nu(y,z)^2\dd y\dd z
\Biggr|
\nonumber
\\
&\le
\frac{1}{m}\sum_{j_1=1}^{m_1}\sum_{j_2=1}^{m_2}
\xi_N^*(\widetilde y_{j_1},\widetilde z_{j_2})^2
+2\sup_{(y,z,\nu)\in D_\delta\times\Xi_1}
|h_\nu(y,z)|
\frac{1}{m}\sum_{j_1=1}^{m_1}\sum_{j_2=1}^{m_2}
|\xi_N^*(\widetilde y_{j_1},\widetilde z_{j_2})|
\nonumber
\\
&\qquad+
\sup_{\nu\in\Xi_1}
\Biggl|
\frac{1}{m}\sum_{j_1=1}^{m_1}\sum_{j_2=1}^{m_2}
h_\nu(\widetilde y_{j_1},\widetilde z_{j_2})^2
-\frac{1}{|D_\delta|}
\iint_{D_\delta} h_\nu(y,z)^2\dd y\dd z
\Biggr|.
\label{eq5-0301}
\end{align}
According to \eqref{eq5-0202}, \eqref{eq5-0204} and the Schwartz inequality,
it holds that  
\begin{align}\label{eq5-0302}
\frac{1}{m}\sum_{j_1=1}^{m_1}\sum_{j_2=1}^{m_2}
\EE[\xi_N^*(\widetilde y_{j_1},\widetilde z_{j_2})^2]
=o(1),
\quad
\frac{1}{m}\sum_{j_1=1}^{m_1}\sum_{j_2=1}^{m_2}
\EE[|\xi_N^*(\widetilde y_{j_1},\widetilde z_{j_2})|]
=o(1)
\end{align}
and
\begin{align}
\sup_{\nu\in\Xi_1}
\Biggl|
\frac{1}{m}\sum_{j_1=1}^{m_1}\sum_{j_2=1}^{m_2}
h_{\nu}(\widetilde y_{j_1},\widetilde z_{j_2})^2
-\frac{1}{|D_\delta|}
\iint_{D_\delta} h_{\nu}(y,z)^2\dd y\dd z
\Biggr|\to0.
\label{eq5-0303}
\end{align}
Therefore, 
we obtain from \eqref{eq5-0301}-\eqref{eq5-0303} and 
the boundedness of $h_\nu(y,z)$ that
\begin{align*}
&\sup_{\nu\in\Xi_1}
\biggl|\frac{U_{N,m}^{(1)}(\nu)}{m}-U(\nu,\nu^*)\biggr|
\pto0.
\end{align*}

\textit{Proof of \eqref{eq5-0104}. }
It follows 
from \eqref{eq5-0202}, \eqref{eq5-0203} and
the boundedness of $K_\nu(y,z)$ on $D_\delta\times\Xi_1$ that
\begin{equation*}
\frac{1}{m}\sum_{j_1=1}^{m_1}\sum_{j_2=1}^{m_2}
\EE\bigl[
\bigl|
\xi_N^*(\widetilde y_{j_1},\widetilde z_{j_2})
K_\nu(\widetilde y_{j_1},\widetilde z_{j_2})
\bigr|
\bigr]
=o(1)
\end{equation*}
and that
\begin{equation*}
\frac{1}{m}
\sum_{j_1=1}^{m_1}\sum_{j_2=1}^{m_2}
L_{\nu^*}(\widetilde y_{j_1},\widetilde z_{j_2})
\to 
\frac{1}{|D_\delta|}
\iint_{D_\delta} L_{\nu^*}(y,z)\dd y\dd z
=L(\nu^*).
\end{equation*}
Therefore, we see from \eqref{eq5-0102} that \eqref{eq5-0104} holds.

\textit{Proof of \eqref{eq5-0105}. }
From \eqref{eq5-0102}, 
there exist uniformly continuous functions 
$V_1(y,z,\nu)$ and $V_2(y,z,\nu)$
with respect to $(y,z,\nu) \in D_\delta\times \Xi_1$ such that
\begin{align*}
\partial^2{U}_{N,m}^{(1)}(\nu)
=
\sum_{j_1=1}^{m_1}\sum_{j_2=1}^{m_2}
Z_{N}^{Q_1}(\widetilde y_{j_1},\widetilde z_{j_2})
V_1(\widetilde y_{j_1},\widetilde z_{j_2},\nu)
+\sum_{j_1=1}^{m_1}\sum_{j_2=1}^{m_2}
V_2(\widetilde y_{j_1},\widetilde z_{j_2},\nu).
\end{align*}
Let $\epsilon>0$ and $\epsilon_N\downarrow0$. 
It follows from the uniform continuity of $V_k$ on $D_\delta\times \Xi_1$ that
\begin{align*}
\sup_{(y,z)\in D_\delta}\sup_{|\nu-\nu^*|\le\epsilon_N}
|V_k(y,z,\nu)-V_k(y,z,\nu^*)|
\le \epsilon
\end{align*}
for large $N$, and that 
\begin{align*}
&\frac{1}{m}\sup_{|\nu-\nu^*|\le\epsilon_N}
|\partial^2{U}_{N,m}^{(1)}(\nu)-\partial^2{U}_{N,m}^{(1)}(\nu^*)|
\\
&=
\frac{1}{m}
\sup_{|\nu-\nu^*|\le\epsilon_N}
\biggl|
\sum_{j_1=1}^{m_1}\sum_{j_2=1}^{m_2}
\Bigl(
Z_{N}^{Q_1}(\widetilde y_{j_1},\widetilde z_{j_2})
\{V_1(y_{j_2},\widetilde z_{j_2},\nu)
-V_1(\widetilde y_{j_1},\widetilde z_{j_2},\nu^*)\}
\\
&\qquad\qquad
-\{V_2(y_{j_2},\widetilde z_{j_2},\nu)
-V_2(\widetilde y_{j_1},\widetilde z_{j_2},\nu^*)\}
\Bigr)
\biggr|
\\
&\le
\frac{1}{m}\sum_{j_1=1}^{m_1}\sum_{j_2=1}^{m_2}
\Bigl(
Z_{N}^{Q_1}(\widetilde y_{j_1},\widetilde z_{j_2})
\sup_{(y,z)\in D_\delta}\sup_{|\nu-\nu^*|\le\epsilon_N}
|V_1(y,z,\nu)-V_1(y,z,\nu^*)|
\\
&\qquad\qquad
+\sup_{(y,z)\in D_\delta}\sup_{|\nu-\nu^*|\le\epsilon_N}
|V_2(y,z,\nu)-V_2(y,z,\nu^*)|
\Bigr)
\\
&\le
\frac{\epsilon}{m}\sum_{j_1=1}^{m_1}\sum_{j_2=1}^{m_2}
(Z_{N}^{Q_1}(\widetilde y_{j_1},\widetilde z_{j_2})+1).
\end{align*}
Therefore, we obtain from \eqref{eq5-0201} that for large $N$,
\begin{align*}
&\EE\biggl[
\frac{1}{m}\sup_{|\nu-\nu^*|\le\epsilon_N}
|\partial^2{U}_{N,m}^{(1)}(\nu)-\partial^2{U}_{N,m}^{(1)}(\nu^*)|
\biggr]
\\
&\le
\frac{\epsilon}{m}\sum_{j_1=1}^{m_1}\sum_{j_2=1}^{m_2}
\EE[Z_{N}^{Q_1}(\widetilde y_{j_1},\widetilde z_{j_2})+1]
\lesssim \epsilon,
\end{align*}
and the proof of \eqref{eq5-0105} is complete.

\textit{Proof of \eqref{eq5-0106}. }
Let $\varphi_\nu(y,z)=\ee^{-(\kappa y+\eta z)}(1,-s y,-s z)^\TT$.
Noting that 
\begin{align*}
\EE[\xi_N^*(y,z)]=O(\Delta_N^{1-\alpha}),
\quad
\EE\bigl[(\xi_N^*(y,z)-\EE[\xi_N^*(y,z)])^2\bigr]
=O(\Delta_N^{1\land2(1-\alpha)})
\end{align*}
uniformly in $(y,z)\in D_\delta$,
we obtain from $m=O(N^\rho)$ and 
$\frac{\rho}{2}-\frac{1}{2}\land(1-\alpha)+\gamma<0$ that
\begin{align*}
&-\frac{N^{\gamma}}{m^{1/2}}\partial{U}_{N,m}^{(1)}(\nu^*)^\TT
\\
&=
\frac{N^\gamma}{m^{1/2}}
\frac{\Gamma(1-\alpha)}{2\pi\alpha}
\sum_{j_1=1}^{m_1}\sum_{j_2=1}^{m_2}
\bigl(
\xi_N^*(\widetilde y_{j_1},\widetilde z_{j_2})
-\EE[\xi_N^*(\widetilde y_{j_1},\widetilde z_{j_2})]
\bigr)
\varphi_{\nu^*}(\widetilde y_{j_1}, \widetilde z_{j_2})
\\
&\qquad+
\frac{N^\gamma}{m^{1/2}}
\frac{\Gamma(1-\alpha)}{2\pi\alpha}
\sum_{j_1=1}^{m_1}\sum_{j_2=1}^{m_2}
\EE[\xi_N^*(\widetilde y_{j_1},\widetilde z_{j_2})]
\varphi_{\nu^*}(\widetilde y_{j_1}, \widetilde z_{j_2})
\\
&\le
m^{1/2}N^\gamma 
\frac{\Gamma(1-\alpha)}{2\pi\alpha}
\Biggl(
\frac{1}{m}
\sum_{j_1=1}^{m_1}\sum_{j_2=1}^{m_2}
\bigl(
\xi_N^*(\widetilde y_{j_1},\widetilde z_{j_2})
-\EE[\xi_N^*(\widetilde y_{j_1},\widetilde z_{j_2})]
\bigr)^2
\Biggr)^{1/2}
\\
&\qquad\qquad\times
\Biggl(
\frac{1}{m}
\sum_{j_1=1}^{m_1}\sum_{j_2=1}^{m_2}
|\varphi_{\nu^*}(\widetilde y_{j_1}, \widetilde z_{j_2})|^2
\Biggr)^{1/2}
\\
&\qquad+
\frac{N^\gamma}{m^{1/2}}
\frac{\Gamma(1-\alpha)}{2\pi\alpha}
\sum_{j_1=1}^{m_1}\sum_{j_2=1}^{m_2}
\EE[\xi_N^*(\widetilde y_{j_1},\widetilde z_{j_2})]
\varphi_{\nu^*}(\widetilde y_{j_1}, \widetilde z_{j_2})
\\
&=O_p\bigl(m^{1/2}N^{-\frac{1}{2}\land(1-\alpha)+\gamma}\bigr)
+O\bigl(m^{1/2}N^{-(1-\alpha)+\gamma}\bigr)
\\
&=o_p(1).
\end{align*}
\end{proof}

\begin{proof}[\bf Proof of Theorem \ref{th2}]
(2) For any $\bs r=(r_1,r_2)$, $r_1.r_2\ge1$, let
\begin{align*}
\mathcal A_n
=\sum_{i=1}^{n}
\bigl\{
(\hat x_{\bs r}^{Q_1}(\widetilde t_i)-x_{\bs r}^{Q_1}(\widetilde t_i))
-(\hat x_{\bs r}^{Q_1}(\widetilde t_{i-1})-x_{\bs r}^{Q_1}(\widetilde t_{i-1}))
\bigr\}^2.
\end{align*}
Note that
\begin{align}
&\Biggl|\hat\sigma_{\bs r}^2
-\sum_{i=1}^{n}(x_{\bs r}^{Q_1}(\widetilde t_i)-x_{\bs r}^{Q_1}(\widetilde t_{i-1}))^2
\Biggr|
\nonumber
\\
&=
\Biggl|
\sum_{i=1}^{n}(\hat x_{\bs r}^{Q_1}(\widetilde t_i)
-\hat x_{\bs r}^{Q_1}(\widetilde t_{i-1}))^2
-\sum_{i=1}^{n}(x_{\bs r}^{Q_1}(\widetilde t_i)-x_{\bs r}^{Q_1}(\widetilde t_{i-1}))^2
\Biggr|
\nonumber
\\
&=
\Biggl|
\sum_{i=1}^{n}\bigl\{
(\hat x_{\bs r}^{Q_1}(\widetilde t_i)-x_{\bs r}^{Q_1}(\widetilde t_i))
-(\hat x_{\bs r}^{Q_1}(\widetilde t_{i-1})-x_{\bs r}^{Q_1}(\widetilde t_{i-1}))
\bigr\}^2
\nonumber
\\
&\qquad
+2\sum_{i=1}^{n}\bigl\{
(\hat x_{\bs r}^{Q_1}(\widetilde t_i)-x_{\bs r}^{Q_1}(\widetilde t_i))
-(\hat x_{\bs r}^{Q_1}(\widetilde t_{i-1})-x_{\bs r}^{Q_1}(\widetilde t_{i-1}))
\bigr\}(x_{\bs r}^{Q_1}(\widetilde t_i)-x_{\bs r}^{Q_1}(\widetilde t_{i-1}))
\Biggr|
\nonumber
\\
&\lesssim
\mathcal A_n
+\mathcal A_n^{1/2}
\Biggl(
\sum_{i=1}^{n}
(x_{\bs r}^{Q_1}(\widetilde t_i)-x_{\bs r}^{Q_1}(\widetilde t_{i-1}))^2
\Biggr)^{1/2}
\label{eq6-0101}
\end{align}
and 
\begin{align*}
\sum_{i=1}^{n}
(x_{\bs r}^{Q_1}(\widetilde t_i)-x_{\bs r}^{Q_1}(\widetilde t_{i-1}))^2
=O_p(1).
\end{align*}

Let $\Delta_n=\lfloor\frac{N}{n}\rfloor\frac{1}{N}$,
$D_{j_1,j_2}=(y_{j_1-1},y_{j_1}]\times(z_{j_2-1},z_{j_2}]$,
$\Delta_i X^{Q_1}(y,z)=X_{\widetilde t_i}^{Q_1}(y,z)-X_{\widetilde t_{i-1}}^{Q_1}(y,z)$
and 
$\Delta_i x_{\bs r}^{Q_1}
=x_{\bs r}^{Q_1}(\widetilde t_i)-x_{\bs r}^{Q_1}(\widetilde t_{i-1})$.
Since
\begin{align*}
\hat x_{\bs r}^{Q_1}(\widetilde t_i)-\hat x_{\bs r}^{Q_1}(\widetilde t_{i-1})
&=
\frac{2}{M}\sum_{j_1=1}^{M_1}\sum_{j_2=1}^{M_2}
\Delta_i X^{Q_1}(y_{j_1},z_{j_2})\sin(\pi r_1 y_{j_1})\sin(\pi r_2 z_{j_2})
\ee^{\frac{\hat\kappa}{2}y_{j_1}}\ee^{\frac{\hat\eta}{2}z_{j_2}},
\\
x_{\bs r}^{Q_1}(\widetilde t_i)-x_{\bs r}^{Q_1}(\widetilde t_{i-1})
&=
2\iint_{D}\Delta_i X^{Q_1}(y,z)
\sin(\pi r_1 y)\sin(\pi r_2 z)\ee^{\frac{\kappa^*}{2}y}\ee^{\frac{\eta^*}{2}z}
\dd y\dd z,
\end{align*}
it follows that 
\begin{align*}
&(\hat x_{\bs r}^{Q_1}(\widetilde t_i)-x_{\bs r}^{Q_1}(\widetilde t_i))
-(\hat x_{\bs r}^{Q_1}(\widetilde t_{i-1})-x_{\bs r}^{Q_1}(\widetilde t_{i-1}))
\\
&=
\frac{2}{M}\sum_{j_1=1}^{M_1}\sum_{j_2=1}^{M_2}
\Delta_i X^{Q_1}(y_{j_1},z_{j_2})\sin(\pi r_1 y_{j_1})\sin(\pi r_2 z_{j_2})
\bigl(\ee^{\frac{\hat\kappa}{2}y_{j_1}}\ee^{\frac{\hat\eta}{2}z_{j_2}}
-\ee^{\frac{\kappa^*}{2}y_{j_1}}\ee^{\frac{\eta^*}{2}z_{j_2}}
\bigr)
\\
&\qquad+
2\sum_{j_1=1}^{M_1}\sum_{j_2=1}^{M_2}\iint_{D_{j_1,j_2}}
\Delta_i X^{Q_1}(y_{j_1},z_{j_2})
\\
&\qquad\qquad\times
\bigl(\sin(\pi r_1 y_{j_1})\sin(\pi r_2 z_{j_2})
\ee^{\frac{\kappa^*}{2}y_{j_1}}\ee^{\frac{\eta^*}{2}z_{j_2}}
-\sin(\pi r_1 y)\sin(\pi r_2 z)
\ee^{\frac{\kappa^*}{2}y}\ee^{\frac{\eta^*}{2}z}
\bigr)
\dd y\dd z
\\
&\qquad+
2\sum_{j_1=1}^{M_1}\sum_{j_2=1}^{M_2}\iint_{D_{j_1,j_2}}
\bigl(\Delta_i X^{Q_1}(y_{j_1},z_{j_2})-\Delta_i X^{Q_1}(y,z)\bigr)
\sin(\pi r_1 y)\sin(\pi r_2 z)
\ee^{\frac{\kappa^*}{2}y}\ee^{\frac{\eta^*}{2}z}
\dd y\dd z
\\
&=:2(\mathcal B_{1,i}+\mathcal B_{2,i}+\mathcal B_{3,i})
\end{align*}
and that
\begin{align*}
\mathcal A_n
\lesssim
\sum_{i=1}^{n}
(\mathcal B_{1,i}^2+\mathcal B_{2,i}^2+\mathcal B_{3,i}^2).
\end{align*}

For the evaluation of $\mathcal B_{1,i}$,
we obtain from the Taylor expansion that 
\begin{align*}
\mathcal B_{1,i}^2
&\le
\frac{1}{M}\sum_{j_1=1}^{M_1}\sum_{j_2=1}^{M_2}
(\Delta_i X^{Q_1})^2(y_{j_1},z_{j_2})\sin^2(\pi r_1 y_{j_1})\sin^2(\pi r_2 z_{j_2})
\bigl(\ee^{\frac{\hat\kappa}{2}y_{j_1}}\ee^{\frac{\hat\eta}{2}z_{j_2}}
-\ee^{\frac{\kappa}{2}y_{j_1}}\ee^{\frac{\eta}{2}z_{j_2}}
\bigr)^2
\\
&=\frac{1}{M}\sum_{j_1=1}^{M_1}\sum_{j_2=1}^{M_2}
(\Delta_i X^{Q_1})^2(y_{j_1},z_{j_2})\sin^2(\pi r_1 y_{j_1})\sin^2(\pi r_2 z_{j_2})
\\
&\qquad\times
\Biggl\{
\frac{1}{2}\int_0^1(y_{j_1},z_{j_2})
\exp\biggl(
\frac{y_{j_1}}{2}(\kappa^*+u(\hat\kappa-\kappa^*))
+\frac{z_{j_2}}{2}(\eta^*+u(\hat\eta-\eta^*))
\biggr)\dd u
\begin{pmatrix}
\hat\kappa-\kappa^*
\\
\hat\eta-\eta^*
\end{pmatrix}
\Biggr\}^2
\\
&\le \frac{1}{mN^{2\gamma}}\frac{1}{M}\sum_{j_1=1}^{M_1}\sum_{j_2=1}^{M_2}
(\Delta_i X^{Q_1})^2(y_{j_1},z_{j_2})
\\
&\qquad\times
\Biggl|
\frac{1}{2}\int_0^1(y_{j_1},z_{j_2})
\exp\biggl(
\frac{y_{j_1}}{2}(\kappa^*+u(\hat\kappa-\kappa^*))
+\frac{z_{j_2}}{2}(\eta^*+u(\hat\eta-\eta^*))
\biggr)\dd u
\Biggr|^2
\\
&\qquad\times
mN^{2\gamma}(|\hat\kappa-\kappa^*|^2+|\hat\eta-\eta^*|^2)
\\
&=:\mathcal C_{1,i}
\times mN^{2\gamma}(|\hat\kappa-\kappa^*|^2+|\hat\eta-\eta^*|^2).
\end{align*}
Note that uniformly in $(y,z)\in D$,
\begin{align*}
\EE[(\Delta_i X^{Q_1})^2(y,z)]
\lesssim
\frac{1}{n^{\alpha}}.
\end{align*}
Let $\epsilon_1$, $\epsilon_2$ be arbitrary positive numbers.
On $\Omega_1=\{|\hat\kappa-\kappa^*|+|\hat\eta-\eta^*|<\epsilon_1\}$,
\begin{align*}
\mathcal C_n
:=n\sum_{i=1}^{n}\mathcal C_{1,i}
&\lesssim
\frac{n}{mN^{2\gamma}}\sum_{i=1}^{n}
\frac{1}{M}
\sum_{j_1=1}^{M_1}\sum_{j_2=1}^{M_2}
(\Delta_i X^{Q_1})^2(y_{j_1},z_{j_2}),
\end{align*}
and
\begin{align*}
P(|\mathcal C_n|>\epsilon_2)
&=P(\{|\mathcal C_n|>\epsilon_2\}\cap \Omega_1)
+P(\{|\mathcal C_n|>\epsilon_2\}\cap \Omega_1^{\mathrm c})
\nonumber
\\
&\le
P\Biggl(
\frac{n}{mN^{2\gamma}}\sum_{i=1}^{n}
\frac{1}{M}\sum_{j_1=1}^{M_1}\sum_{j_2=1}^{M_2}
(\Delta_i X^{Q_1})^2(y_{j_1},z_{j_2})
\gtrsim\epsilon_2
\Biggr)
+P(\Omega_1^{\mathrm c})
\nonumber
\\
&\lesssim
\frac{n}{\epsilon_2 mN^{2\gamma}}\sum_{i=1}^{n}
\frac{1}{M}\sum_{j_1=1}^{M_1}\sum_{j_2=1}^{M_2}
\EE[(\Delta_i X^{Q_1})^2(y_{j_1},z_{j_2})]
+P(\Omega_1^{\mathrm c})
\nonumber
\\
&\lesssim
\frac{n^{2-\alpha}}{\epsilon_2 mN^{2\gamma}}+P(\Omega_1^{\mathrm c}).
\end{align*}
From $\frac{n^{2-\alpha}}{mN^{2\gamma}}\to0$ and Theorem \ref{th1}, 
one has $\mathcal C_n=o_p(1)$ and
\begin{align*}
n\sum_{i=1}^{n}
\mathcal B_{1,i}^2
=o_p(1).
\end{align*}

Noting that for the evaluation of $\mathcal B_{2,i}$,
\begin{align*}
\EE[\mathcal B_{2,i}^2]
&\le
\sum_{j_1=1}^{M_1}\sum_{j_2=1}^{M_2}\iint_{D_{j_1,j_2}}
\EE[(\Delta_i X^{Q_1})^2(y_{j_1},z_{j_2})]
\\
&\qquad\qquad\times
\bigl(\sin(\pi r_1 y_{j_1})\sin(\pi r_2 z_{j_2})
\ee^{\frac{\kappa}{2}y_{j_1}}\ee^{\frac{\eta}{2}z_{j_2}}
-\sin(\pi r_1 y)\sin(\pi r_2 z)
\ee^{\frac{\kappa}{2}y}\ee^{\frac{\eta}{2}z}
\bigr)^2
\dd y\dd z
\\
&\lesssim
\sum_{j_1=1}^{M_1}\sum_{j_2=1}^{M_2}
\EE[(\Delta_i X^{Q_1})^2(y_{j_1},z_{j_2})]
\iint_{D_{j_1,j_2}}
\biggl(\frac{1}{M_1^2}+\frac{1}{M_2^2}\biggr)
\dd y\dd z
\\
&\lesssim
\frac{1}{n^\alpha}
\biggl(\frac{1}{M_1^2}+\frac{1}{M_2^2}\biggr)
\lesssim
\frac{1}{n^\alpha(M_1^2 \land M_2^2)},
\end{align*}
one has that under
$\frac{n^{2-\alpha+\epsilon}}{M_1^{2\epsilon} \land M_2^{2\epsilon}}\to0$, 
\begin{equation*}
n\sum_{i=1}^{n}
\mathcal B_{2,i}^2
=O_p\biggl(
\frac{n^{2-\alpha}}{M_1^2 \land M_2^2}
\biggr)
=o_p(1).
\end{equation*}

For the evaluation of $\mathcal B_{3,i}$, we have that
\begin{align*}
\mathcal B_{3,i}^2
&\le
\sum_{j_1=1}^{M_1}\sum_{j_2=1}^{M_2}\iint_{D_{j_1,j_2}}
\bigl(\Delta_i X^{Q_1}(y_{j_1},z_{j_2})-\Delta_i X^{Q_1}(y,z)\bigr)^2
\sin^2(\pi r_1 y)\sin^2(\pi r_2 z)
\ee^{\kappa^*y}\ee^{\eta^*z}
\dd y\dd z
\\
&\lesssim
\sum_{j_1=1}^{M_1}\sum_{j_2=1}^{M_2}\iint_{D_{j_1,j_2}}
\bigl(\Delta_i X^{Q_1}(y_{j_1},z_{j_2})-\Delta_i X^{Q_1}(y,z)\bigr)^2
\dd y\dd z.
\end{align*}
Since for $(y,z)\in D_{j_1,j_2}$,
\begin{equation*}
\partial_y e_{\bs k}(y,z)
=
2\Bigl(
\pi k_1 \cos(\pi k_1 y)-\frac{\kappa}{2}\sin(\pi k_1 y)
\Bigr)\sin(\pi k_2 z)
\ee^{-\frac{\kappa}{2}y}\ee^{-\frac{\eta}{2}z},
\end{equation*}
\begin{equation*}
|\partial e_{\bs k}(y,z)|
\lesssim
\Bigl| \pi k_1 \cos(\pi k_1 y)-\frac{\kappa}{2}\sin(\pi k_1 y) \Bigr|
+\Bigl| \pi k_2 \cos(\pi k_2 z)-\frac{\eta}{2}\sin(\pi k_2 z) \Bigr|
\lesssim
|\bs k|,
\end{equation*}
\begin{align*}
\Delta_i X^{Q_1}(y_{j_1},z_{j_2})-\Delta_i X^{Q_1}(y,z)
&=\sum_{\bs k\in\mathbb N^2}
\Delta_i x_{\bs k}^{Q_1}(e_{\bs k}(y_{j_1},z_{j_2})-e_{\bs k}(y,z)),
\\
|e_{\bs k}(y_{j_1},z_{j_2})-e_{\bs k}(y,z)|
&\le
\int_0^1
\bigl|\partial e_{\bs k}(y_{j_1}+u(y-y_{j_1}),z_{j_2}+u(z-z_{j_2}))\bigr|\dd u
\\
&\qquad\times
\bigl(|y-y_{j_1}|^2+|z-z_{j_2}|^2\bigr)^{1/2}
\\
&\lesssim
|\bs k|
\biggl(\frac{1}{M_1^2}+\frac{1}{M_2^2}\biggr)^{1/2}
\lesssim
\frac{\lambda_{\bs k}^{1/2}}{M_1\land M_2}
\end{align*}
and $|e_{\bs k}(y_{j_1},z_{j_2})-e_{\bs k}(y,z)|\lesssim 1$, it holds that
for $\epsilon<\alpha$,
\begin{align*}
&\EE\bigl[(\Delta_i X^{Q_1}(y_{j_1},z_{j_2})-\Delta_i X^{Q_1}(y,z))^2\bigr]
\\
&=\sum_{\bs k,\bs \ell\in\mathbb N^2}
\EE[\Delta_i x_{\bs k}^{Q_1}\Delta_i x_{\bs \ell}^{Q_1}]
(e_{\bs k}(y_{j_1},z_{j_2})-e_{\bs k}(y,z))
(e_{\bs \ell}(y_{j_1},z_{j_2})-e_{\bs \ell}(y,z))
\\
&=
\sum_{\bs k,\bs \ell\in\mathbb N^2}
\EE[A_{i,\bs k}A_{i,\bs \ell}]
(e_{\bs k}(y_{j_1},z_{j_2})-e_{\bs k}(y,z))
(e_{\bs \ell}(y_{j_1},z_{j_2})-e_{\bs \ell}(y,z))
\\
&\qquad+
\sum_{\bs k,\bs \ell\in\mathbb N^2}
\EE[B_{i,\bs k}^{Q_1}B_{i,\bs \ell}^{Q_1}]
(e_{\bs k}(y_{j_1},z_{j_2})-e_{\bs k}(y,z))
(e_{\bs \ell}(y_{j_1},z_{j_2})-e_{\bs \ell}(y,z))
\\
&\lesssim
\sum_{\bs k\in\mathbb N^2}
\frac{1-\ee^{-\lambda_{\bs k}
\Delta_{n}
}}{\lambda_{\bs k}^{1+\alpha}}
(e_{\bs k}(y_{j_1},z_{j_2})-e_{\bs k}(y,z))^2
\\
&=
\sum_{\bs k\in\mathbb N^2}
\frac{1-\ee^{-\lambda_{\bs k}\Delta_{n}}}{\lambda_{\bs k}^{1+\alpha-\epsilon}}
\frac{(e_{\bs k}(y_{j_1},z_{j_2})-e_{\bs k}(y,z))^2}{\lambda_{\bs k}^\epsilon}
\\
&\lesssim
\sum_{\bs k\in\mathbb N^2}
\frac{1-\ee^{-\lambda_{\bs k}\Delta_{n}}}{\lambda_{\bs k}^{1+\alpha-\epsilon}}
\frac{1}{\lambda_{\bs k}^\epsilon}
\biggl(
\frac{\lambda_{\bs k}^{1/2}}{M_1 \land M_2}\land 1
\biggr)^2
\\
&\le
\sum_{\bs k\in\mathbb N^2}
\frac{1-\ee^{-\lambda_{\bs k}\Delta_{n}}}{\lambda_{\bs k}^{1+\alpha-\epsilon}}
\frac{1}{\lambda_{\bs k}^\epsilon}
\biggl(
\frac{\lambda_{\bs k}^{1/2}}{M_1 \land M_2}
\biggr)^{2\epsilon}
1^{2-2\epsilon}
\\
&=
\frac{1}{M_1^{2\epsilon} \land M_2^{2\epsilon}}
\sum_{\bs k\in\mathbb N^2}
\frac{1-\ee^{-\lambda_{\bs k}\Delta_{n}}}{\lambda_{\bs k}^{1+\alpha-\epsilon}}
=O\biggl(
\frac{1}{n^{\alpha-\epsilon}(M_1^{2\epsilon} \land M_2^{2\epsilon})}
\biggr)
\end{align*}
and that under
$\frac{n^{2-\alpha+\epsilon}}{M_1^{2\epsilon} \land M_2^{2\epsilon}}\to0$, 
\begin{equation*}
n\sum_{i=1}^{n}
\mathcal B_{3,i}^2
=O_p\biggl(
\frac{n^{2-\alpha+\epsilon}}{M_1^{2\epsilon} \land M_2^{2\epsilon}}
\biggr)
=o_p(1).
\end{equation*}

Therefore, 
it holds that 
under $\frac{n^{2-\alpha}}{mN^{2\gamma}}\to0$ and
$\frac{n^{2-\alpha+\epsilon}}{M_1^{2\epsilon} \land M_2^{2\epsilon}}\to0$,
$n\mathcal A_n=o_p(1)$ and 
\begin{align*}
\sqrt{n}\Biggl(
\hat\sigma_{\bs r}^2
-\sum_{i=1}^{n}(x_{\bs r}^{Q_1}(\widetilde t_i)-x_{\bs r}^{Q_1}(\widetilde t_{i-1}))^2
\Biggr)
=o_p(1).
\end{align*}
Since
\begin{equation}\label{eq6-0102}
\sqrt{n}
\begin{pmatrix}
\displaystyle\sum_{i=1}^{n}
(x_{1,1}^{Q_1}(\widetilde t_i)-x_{1,1}^{Q_1}(\widetilde t_{i-1}))^2
-(\sigma_{1,1}^*)^2
\\
\displaystyle
\sum_{i=1}^{n}(x_{1,2}^{Q_1}(\widetilde t_i)-x_{1,2}^{Q_1}(\widetilde t_{i-1}))^2
-(\sigma_{1,2}^*)^2
\end{pmatrix}
\dto
N
\Biggl(0,
\begin{pmatrix}
2(\sigma_{1,1}^*)^4 & 0
\\
0 & 2(\sigma_{1,2}^*)^4 
\end{pmatrix}
\Biggr),
\end{equation}
it follows that
\begin{equation}\label{eq6-0103}
\sqrt{n}
\begin{pmatrix}
\hat\sigma_{1,1}^2-(\sigma_{1,1}^*)^2
\\
\hat\sigma_{1,2}^2-(\sigma_{1,2}^*)^2
\end{pmatrix}
\dto 
N
\Biggl(0,
\begin{pmatrix}
2(\sigma_{1,1}^*)^4 & 0
\\
0 & 2(\sigma_{1,2}^*)^4 
\end{pmatrix}
\Biggr).
\end{equation}

Let
\begin{align*}
G(x,y)
=\biggl(
\frac{1}{y^{1/\alpha}}
-\frac{1}{x^{1/\alpha}}
\biggr)^{-\frac{\alpha}{1-\alpha}},
\quad
H(x,y)=(x^{-1}G(x,y))^{1/\alpha}.
\end{align*}
Note that under 
$\frac{n^{2-\alpha}}{mN^{2\gamma}}\to0$ and
$\frac{n^{2-\alpha+\epsilon}}{M_1^{2\epsilon} \land M_2^{2\epsilon}}\to0$,
\begin{align*}
&\sqrt{n}
\Biggl\{
\biggl(\frac{1}{\hat s}\biggr)^{\frac{1}{1-\alpha}}
-\biggl(\frac{1}{s^*}\biggr)^{\frac{1}{1-\alpha}}
\Biggr\}
\\
&=\frac{\sqrt{n}}{m^{1/2} N^\gamma}
\int_0^1\frac{-1}{1-\alpha}
\biggl(\frac{1}{s^*+u(\hat s-s^*)}\biggr)^{\frac{1}{1-\alpha}+1}\dd u\,
m^{1/2} N^\gamma(\hat s-s^*)
\\
&=o_p(1),
\end{align*}
\begin{align*}
&\sqrt{n}(\hat\theta_2-\theta_2^*)
\\
&=\sqrt{n}
\Biggl[
\biggl\{
\frac{3\pi^2}{\hat s^{1/\alpha}}
\biggl(
\frac{1}{(\hat \sigma_{1,2}^2)^{1/\alpha}}
-\frac{1}{(\hat \sigma_{1,1}^2)^{1/\alpha}}
\biggr)^{-1}
\biggr\}^{\frac{\alpha}{1-\alpha}}
-
\biggl\{
\frac{3\pi^2}{(s^*)^{1/\alpha}}
\biggl(
\frac{1}{(\sigma_{1,2}^*)^{2/\alpha}}
-\frac{1}{(\sigma_{1,1}^*)^{2/\alpha}}
\biggr)^{-1}
\biggr\}^{\frac{\alpha}{1-\alpha}}
\Biggl]
\\
&=\sqrt{n}
\Biggl[
\Biggl\{
\biggl(
\frac{3\pi^2}{\hat s^{1/\alpha}}
\biggr)^{\frac{\alpha}{1-\alpha}}
-\biggl(
\frac{3\pi^2}{(s^*)^{1/\alpha}}
\biggr)^{\frac{\alpha}{1-\alpha}}
\Biggr\}
\biggl(
\frac{1}{(\hat \sigma_{1,2}^2)^{1/\alpha}}
-\frac{1}{(\hat \sigma_{1,1}^2)^{1/\alpha}}
\biggr)^{-\frac{\alpha}{1-\alpha}}
\\
&\qquad+
\biggl(\frac{3\pi^2}{(s^*)^{1/\alpha}}\biggr)^{\frac{\alpha}{1-\alpha}}
\Biggl\{
\biggl(
\frac{1}{(\hat \sigma_{1,2}^2)^{1/\alpha}}
-\frac{1}{(\hat \sigma_{1,1}^2)^{1/\alpha}}
\biggr)^{-\frac{\alpha}{1-\alpha}}
-
\biggl(
\frac{1}{(\sigma_{1,2}^*)^{2/\alpha}}
-\frac{1}{(\sigma_{1,1}^*)^{2/\alpha}}
\biggr)^{-\frac{\alpha}{1-\alpha}}
\Biggr\}
\Biggl]
\\
&=
\biggl(\frac{3\pi^2}{(s^*)^{1/\alpha}}\biggr)^{\frac{\alpha}{1-\alpha}}
\sqrt{n}
\Biggl\{
\biggl(
\frac{1}{(\hat \sigma_{1,2}^2)^{1/\alpha}}
-\frac{1}{(\hat \sigma_{1,1}^2)^{1/\alpha}}
\biggr)^{-\frac{\alpha}{1-\alpha}}
-
\biggl(
\frac{1}{(\sigma_{1,2}^*)^{2/\alpha}}
-\frac{1}{(\sigma_{1,1}^*)^{2/\alpha}}
\biggr)^{-\frac{\alpha}{1-\alpha}}
\Biggr\}
+o_p(1)
\\
&=
\biggl(
\frac{3\pi^2\theta_2^*}{(\sigma^*)^{2/\alpha}}
\biggr)^{\frac{\alpha}{1-\alpha}}
\theta_2^*
\sqrt{n}
\Bigl(G(\hat \sigma_{1,1}^2,\hat \sigma_{1,2}^2)
-G\bigl((\sigma_{1,1}^*)^2,(\sigma_{1,2}^*)^2\bigr)\Bigl)
+o_p(1)
\\
&=
\biggl(
\frac{3\pi^2\theta_2^*}{(\sigma^*)^{2/\alpha}}
\biggr)^{\frac{\alpha}{1-\alpha}}
\theta_2^*
\partial G\bigl((\sigma_{1,1}^*)^2,(\sigma_{1,2}^*)^2\bigr)
\sqrt{n}
\begin{pmatrix}
\hat\sigma_{1,1}^2-(\sigma_{1,1}^*)^2
\\
\hat\sigma_{1,2}^2-(\sigma_{1,2}^*)^2
\end{pmatrix}
+o_p(1)
\end{align*}
and 
\begin{align*}
&\sqrt{n}(\hat\lambda_{1,1}-\lambda_{1,1}^*)
\\
&=\sqrt{n}
\Biggl\{
\biggl(\frac{3\pi^2}{\hat s}\biggr)^{\frac{1}{1-\alpha}}
\biggl(
\frac{1}{(\hat\sigma_{1,2}^2)^{1/\alpha}}
-\frac{1}{(\hat\sigma_{1,1}^2)^{1/\alpha}}
\biggr)^{-\frac{1}{1-\alpha}}
\frac{1}{(\hat\sigma_{1,1}^2)^{1/\alpha}}
\\
&\qquad\qquad-
\biggl(\frac{3\pi^2}{s^*}\biggr)^{\frac{1}{1-\alpha}}
\biggl(
\frac{1}{(\sigma_{1,2}^*)^{2/\alpha}}
-\frac{1}{(\sigma_{1,1}^*)^{2/\alpha}}
\biggr)^{-\frac{1}{1-\alpha}}
\frac{1}{(\sigma_{1,1}^*)^{2/\alpha}}
\Biggr\}
\\
&=\sqrt{n}
\Biggl[
\Biggl\{
\biggl(\frac{3\pi^2}{\hat s}\biggr)^{\frac{1}{1-\alpha}}
-\biggl(\frac{3\pi^2}{s^*}\biggr)^{\frac{1}{1-\alpha}}
\Biggr\}
\biggl(
\frac{1}{(\hat\sigma_{1,2}^2)^{1/\alpha}}
-\frac{1}{(\hat\sigma_{1,1}^2)^{1/\alpha}}
\biggr)^{-\frac{1}{1-\alpha}}
\frac{1}{(\hat\sigma_{1,1}^2)^{1/\alpha}}
\\
&\qquad\qquad+
\biggl(\frac{3\pi^2}{s^*}\biggr)^{\frac{1}{1-\alpha}}
\Biggl\{
\biggl(
\frac{1}{(\hat\sigma_{1,2}^2)^{1/\alpha}}
-\frac{1}{(\hat\sigma_{1,1}^2)^{1/\alpha}}
\biggr)^{-\frac{1}{1-\alpha}}
\frac{1}{(\hat\sigma_{1,1}^2)^{1/\alpha}}
\\
&\qquad\qquad\qquad-
\biggl(
\frac{1}{(\sigma_{1,2}^*)^{2/\alpha}}
-\frac{1}{(\sigma_{1,1}^*)^{2/\alpha}}
\biggr)^{-\frac{1}{1-\alpha}}
\frac{1}{(\sigma_{1,1}^*)^{2/\alpha}}
\Biggr\}
\Biggr]
\\
&=
\biggl(\frac{3\pi^2}{s^*}\biggr)^{\frac{1}{1-\alpha}}
\sqrt{n}
\Biggl\{
\biggl(
\frac{1}{(\hat\sigma_{1,2}^2)^{1/\alpha}}
-\frac{1}{(\hat\sigma_{1,1}^2)^{1/\alpha}}
\biggr)^{-\frac{1}{1-\alpha}}
\frac{1}{(\hat\sigma_{1,1}^2)^{1/\alpha}}
\\
&\qquad\qquad-
\biggl(
\frac{1}{(\sigma_{1,2}^*)^{2/\alpha}}
-\frac{1}{(\sigma_{1,1}^*)^{2/\alpha}}
\biggr)^{-\frac{1}{1-\alpha}}
\frac{1}{(\sigma_{1,1}^*)^{2/\alpha}}
\Biggr\}
+o_p(1)
\\
&=
\biggl(\frac{3\pi^2\theta_2^*}{(\sigma^*)^2}\biggr)^{\frac{1}{1-\alpha}}
\sqrt{n}
\Bigl(H(\hat \sigma_{1,1}^2,\hat \sigma_{1,2}^2)
-H\bigl((\sigma_{1,1}^*)^2,(\sigma_{1,2}^*)^2\bigr)\Bigr)
+o_p(1)
\\
&=
\biggl(\frac{3\pi^2\theta_2^*}{(\sigma^*)^2}\biggr)^{\frac{1}{1-\alpha}}
\partial H\bigl((\sigma_{1,1}^*)^2,(\sigma_{1,2}^*)^2\bigr)
\sqrt{n}
\begin{pmatrix}
\hat\sigma_{1,1}^2-(\sigma_{1,1}^*)^2
\\
\hat\sigma_{1,2}^2-(\sigma_{1,2}^*)^2
\end{pmatrix}
+o_p(1).
\end{align*}
It then follows from 
$\sigma^2=s\theta_2$, $\theta_1=\kappa\theta_2$ and $\eta_1=\eta\theta_2$ that
\begin{align*}
\sqrt{n}(\hat\sigma^2-(\sigma^*)^2)
&=\sqrt{n}(\hat s \hat\theta_2-s^*\theta_2^*)
\\
&=\sqrt{n}\bigl((\hat s-s^*)\hat\theta_2+s^*(\hat\theta_2-\theta_2^*)\bigr)
\\
&=\frac{\sqrt{n}}{m^{1/2}N^\gamma}\cdot m^{1/2}N^\gamma(\hat s-s^*)\hat\theta_2
+s^*\sqrt{n}(\hat\theta_2-\theta_2^*)
\\
&=s^* \sqrt{n}(\hat\theta_2-\theta_2^*)+o_p(1)
\\
&=
\biggl(
\frac{3\pi^2\theta_2^*}{(\sigma^*)^{2/\alpha}}
\biggr)^{\frac{\alpha}{1-\alpha}}
(\sigma^*)^2
\partial G\bigl((\sigma_{1,1}^*)^2,(\sigma_{1,2}^*)^2\bigr)
\sqrt{n}
\begin{pmatrix}
\hat\sigma_{1,1}^2-(\sigma_{1,1}^*)^2
\\
\hat\sigma_{1,2}^2-(\sigma_{1,2}^*)^2
\end{pmatrix}
+o_p(1),
\\
\sqrt{n}(\hat\theta_1-\theta_1^*)
&=\kappa^* \sqrt{n}(\hat\theta_2-\theta_2^*)+o_p(1)
\\
&=
\biggl(
\frac{3\pi^2\theta_2^*}{(\sigma^*)^{2/\alpha}}
\biggr)^{\frac{\alpha}{1-\alpha}}
\theta_1^*
\partial G\bigl((\sigma_{1,1}^*)^2,(\sigma_{1,2}^*)^2\bigr)
\sqrt{n}
\begin{pmatrix}
\hat\sigma_{1,1}^2-(\sigma_{1,1}^*)^2
\\
\hat\sigma_{1,2}^2-(\sigma_{1,2}^*)^2
\end{pmatrix}
+o_p(1),
\\
\sqrt{n}(\hat\eta_1-\eta_1^*)
&=\eta^* \sqrt{n}(\hat\theta_2-\theta_2^*)+o_p(1)
\\
&=
\biggl(
\frac{3\pi^2\theta_2^*}{(\sigma^*)^{2/\alpha}}
\biggr)^{\frac{\alpha}{1-\alpha}}
\eta_1^*
\partial G\bigl((\sigma_{1,1}^*)^2,(\sigma_{1,2}^*)^2\bigr)
\sqrt{n}
\begin{pmatrix}
\hat\sigma_{1,1}^2-(\sigma_{1,1}^*)^2
\\
\hat\sigma_{1,2}^2-(\sigma_{1,2}^*)^2
\end{pmatrix}
+o_p(1).
\end{align*}
Furthermore, we obtain from 
$\lambda_{1,1}+\theta_0=((\kappa^2+\eta^2)/4+2\pi^2)\theta_2$ that
\begin{align*}
&\sqrt{n}(\hat\theta_0-\theta_0^*)
\\
&=\sqrt{n}
\Biggl[
-\hat\lambda_{1,1}
+\biggl(
\frac{\hat\kappa^2+\hat\eta^2}{4}+2\pi^2
\biggr)\hat\theta_2
-\biggl\{
-\lambda_{1,1}^*
+\biggl(
\frac{(\kappa^*)^2+(\eta^*)^2}{4}+2\pi^2
\biggr)\theta_2^*
\biggr\}
\Biggr]
\\
&=-\sqrt{n}
(\hat\lambda_{1,1}-\lambda_{1,1}^*)
\\
&\qquad+
\sqrt{n}
\Biggl\{
\biggl(
\frac{\hat\kappa^2+\hat\eta^2}{4}
-\frac{(\kappa^*)^2+(\eta^*)^2}{4}
\biggr)\hat\theta_2
+\biggl(
\frac{(\kappa^*)^2+(\eta^*)^2}{4}+2\pi^2
\biggr)(\hat\theta_2-\theta_2^*)
\Biggr\}
\\
&=-\sqrt{n}
(\hat\lambda_{1,1}-\lambda_{1,1}^*)
+\biggl(
\frac{(\kappa^*)^2+(\eta^*)^2}{4}+2\pi^2
\biggr)
\sqrt{n}
(\hat\theta_2-\theta_2^*)
\\
&\qquad
+\frac{\sqrt{n}}{m^{1/2}N^\gamma}
m^{1/2}N^\gamma
\biggl(
\frac{\hat\kappa^2+\hat\eta^2}{4}
-\frac{(\kappa^*)^2+(\eta^*)^2}{4}
\biggr)\hat\theta_2
\\
&=
-\biggl(\frac{3\pi^2\theta_2^*}{(\sigma^*)^2}\biggr)^{\frac{1}{1-\alpha}}
\partial H\bigl((\sigma_{1,1}^*)^2,(\sigma_{1,2}^*)^2\bigr)
\sqrt{n}
\begin{pmatrix}
\hat\sigma_{1,1}^2-(\sigma_{1,1}^*)^2
\\
\hat\sigma_{1,2}^2-(\sigma_{1,2}^*)^2
\end{pmatrix}
\\
&\qquad+
\biggl(
\frac{3\pi^2\theta_2^*}{(\sigma^*)^{2/\alpha}}
\biggr)^{\frac{\alpha}{1-\alpha}}
(\lambda_{1,1}^*+\theta_0^*)
\partial G\bigl((\sigma_{1,1}^*)^2,(\sigma_{1,2}^*)^2\bigr)
\sqrt{n}
\begin{pmatrix}
\hat\sigma_{1,1}^2-(\sigma_{1,1}^*)^2
\\
\hat\sigma_{1,2}^2-(\sigma_{1,2}^*)^2
\end{pmatrix}
+o_p(1).
\end{align*}
Therefore, setting 
$\bs\vartheta^*=(\theta_0^*,\bs \vartheta_{-1}^*)^\TT$ and 
$\bs e_1=(1,0,0,0,0)^\TT$, we obtain from \eqref{eq6-0103}
that under
$\frac{n^{2-\alpha}}{mN^{2\gamma}}\to0$ and
$\frac{n^{2-\alpha+\epsilon}}{M_1^{2\epsilon} \land M_2^{2\epsilon}}\to0$,
\begin{align*}
\sqrt n
\begin{pmatrix}
\hat \theta_0-\theta_0^*
\\
\hat \theta_1-\theta_1^*
\\
\hat \eta_1-\eta_1^*
\\
\hat \theta_2-\theta_2^*
\\
\hat \sigma^2-(\sigma^*)^2
\end{pmatrix}
&=
\Biggl\{
\biggl(
\frac{3\pi^2\theta_2^*}{(\sigma^*)^{2/\alpha}}
\biggr)^{\frac{\alpha}{1-\alpha}}
\bs\vartheta^*
\partial G((\sigma_{1,1}^*)^2,(\sigma_{1,2}^*)^2)
\\
&\qquad
+\biggl(
\frac{3\pi^2\theta_2^*}{(\sigma^*)^{2/\alpha}}
\biggr)^{\frac{\alpha}{1-\alpha}}
\lambda_{1,1}^*
\bs e_1
\partial G((\sigma_{1,1}^*)^2,(\sigma_{1,2}^*)^2)
\\
&\qquad
-\biggl(\frac{3\pi^2\theta_2^*}{(\sigma^*)^2}\biggr)^{\frac{1}{1-\alpha}}
\bs e_1
\partial H((\sigma_{1,1}^*)^2,(\sigma_{1,2}^*)^2)
\Biggr\}
\sqrt{n}
\begin{pmatrix}
\hat\sigma_{1,1}^2-(\sigma_{1,1}^*)^2
\\
\hat\sigma_{1,2}^2-(\sigma_{1,2}^*)^2
\end{pmatrix}
+o_p(1)
\\
&=:
A\sqrt{n}
\begin{pmatrix}
\hat\sigma_{1,1}^2-(\sigma_{1,1}^*)^2
\\
\hat\sigma_{1,2}^2-(\sigma_{1,2}^*)^2
\end{pmatrix}
+o_p(1)
\\
&\dto
N(0,J),
\end{align*}
where
\begin{align*}
J=
A
\begin{pmatrix}
2(\sigma_{1,1}^*)^4 & 0\\
0 & 2(\sigma_{1,2}^*)^4
\end{pmatrix}
A^\TT.
\end{align*}
Let
\begin{align*}
\bs a_1
&=
\frac{-\lambda_{1,1}^*}{3\pi^2(1-\alpha)\theta_2^*}\bs \vartheta^*
+\frac{\lambda_{1,1}^*\lambda_{1,2}^*}{3\pi^2\alpha\theta_2^*}
\bs e_1
=: r_1\bs \vartheta^*+r_2\bs e_1,
\\
\bs a_2
&=
\frac{\lambda_{1,2}^*}{3\pi^2(1-\alpha)\theta_2^*}\bs \vartheta^*
-\frac{\lambda_{1,1}^*\lambda_{1,2}^*}{3\pi^2\alpha\theta_2^*}
\bs e_1
=: r_3\bs \vartheta^*-r_2\bs e_1.
\end{align*}
Note that
\begin{align*}
\partial G((\sigma_{1,1}^*)^2,(\sigma_{1,2}^*)^2)
&=
\frac{(\sigma^*)^{\frac{2\alpha}{1-\alpha}}}
{(1-\alpha)(3\pi^2\theta_2^*)^{\frac{1}{1-\alpha}}}
\bigl(-(\lambda_{1,1}^*)^{1+\alpha}, (\lambda_{1,2}^*)^{1+\alpha}\bigr),
\\
\partial H((\sigma_{1,1}^*)^2,(\sigma_{1,2}^*)^2)
&=
\frac{(\sigma^*)^{\frac{2\alpha}{1-\alpha}}}
{\alpha(1-\alpha)(3\pi^2\theta_2^*)^{\frac{1}{1-\alpha}+1}}
\bigl(
-(\lambda_{1,1}^*)^{1+\alpha}
\bigl(\alpha\lambda_{1,1}^*+(1-\alpha)\lambda_{1,2}^*\bigr), 
\lambda_{1,1}^*(\lambda_{1,2}^*)^{1+\alpha}
\bigr),
\\
A
&=
\frac{1}{3\pi^2(1-\alpha)\theta_2^*(\sigma^*)^2}
\bs \vartheta^*
\bigl(-(\lambda_{1,1}^*)^{1+\alpha}, (\lambda_{1,2}^*)^{1+\alpha}\bigr)
\\
&\qquad
+\frac{1}{3\pi^2\alpha\theta_2^*(\sigma^*)^2}
\bs e_1
\bigl(
(\lambda_{1,1}^*)^{1+\alpha}\lambda_{1,2}^*, 
-\lambda_{1,1}^*(\lambda_{1,2}^*)^{1+\alpha}
\bigr)
\\
&=((\sigma_{1,1}^*)^{-2}\bs a_1,(\sigma_{1,2}^*)^{-2}\bs a_2)
\end{align*}
and $J=2(\bs a_1\bs a_1^\TT+\bs a_2\bs a_2^\TT)$. 
Since
\begin{equation*}
r_1\theta_0^*+r_2
=\frac{\lambda_{1,1}^*}{3\pi^2\theta_2^*}
\biggl(
\frac{\lambda_{1,2}^*}{\alpha}-\frac{\theta_0^*}{1-\alpha}
\biggr),
\quad
r_3\theta_0^*-r_2
=\frac{-\lambda_{1,2}^*}{3\pi^2\theta_2^*}
\biggl(
\frac{\lambda_{1,1}^*}{\alpha}-\frac{\theta_0^*}{1-\alpha}
\biggr),
\end{equation*}
\begin{align*}
&(r_1\theta_0^*+r_2)^2+(r_3\theta_0^*-r_2)^2
\\
&=\frac{1}{9\pi^4(\theta_2^*)^2}
\Biggl\{
(\lambda_{1,1}^*)^2
\biggl(
\frac{\lambda_{1,2}^*}{\alpha}-\frac{\theta_0^*}{1-\alpha}
\biggr)^2
+
(\lambda_{1,2}^*)^2
\biggl(
\frac{\lambda_{1,1}^*}{\alpha}-\frac{\theta_0^*}{1-\alpha}
\biggr)^2
\Biggr\}
=\frac{c_1}{9\pi^4(\theta_2^*)^2},
\\
&r_1(r_1\theta_0^*+r_2)+r_3(r_3\theta_0^*-r_2)
\\
&=\frac{-1}{9\pi^4(1-\alpha)(\theta_2^*)^2}
\Biggl\{
(\lambda_{1,1}^*)^2
\biggl(
\frac{\lambda_{1,2}^*}{\alpha}-\frac{\theta_0^*}{1-\alpha}
\biggr)
+
(\lambda_{1,2}^*)^2
\biggl(
\frac{\lambda_{1,1}^*}{\alpha}-\frac{\theta_0^*}{1-\alpha}
\biggr)
\Biggr\}
=\frac{c_2}{9\pi^4(\theta_2^*)^2},
\\
&r_1^2+r_3^2
=\frac{1}{9\pi^4(1-\alpha)^2(\theta_2^*)^2}
\bigl\{(\lambda_{1,1}^*)^2+(\lambda_{1,2}^*)^2\bigr\}
=\frac{c_3}{9\pi^4(\theta_2^*)^2},
\end{align*}
we see that
\begin{align*}
\bs a_1\bs a_1^\TT+\bs a_2\bs a_2^\TT
=
\frac{1}{9\pi^4(\theta_2^*)^2}
\begin{pmatrix}
c_1 & c_2(\bs \vartheta_{-1}^*)^\TT
\\
c_2\bs \vartheta_{-1}^* & c_3\bs \vartheta_{-1}^*(\bs \vartheta_{-1}^*)^\TT
\end{pmatrix}
\end{align*}
and thus we obtain the desired result. 

(1) For the evaluation of $\mathcal A_n$, 
the following holds by the same argument as above.
For a sequence $\{r_n\}$ with 
$\frac{r_n n^{1-\alpha}}{mN^{2\gamma}}\to0$
and $\frac{r_n n^{1-\alpha+\epsilon}}{M_1^{2\epsilon} \land M_2^{2\epsilon}}\to0$
for some $\epsilon<\alpha$ as $n\to\infty$ and $M\to\infty$, 
\begin{equation}\label{eq6-0104}
r_n \mathcal A_n=o_p(1).
\end{equation}
Therefore, it follows from \eqref{eq6-0104} that 
under $\frac{n^{1-\alpha}}{mN^{2\gamma}}\to0$
and $\frac{n^{1-\alpha+\epsilon}}{M_1^{2\epsilon} \land M_2^{2\epsilon}}\to0$,
$\mathcal A_n=o_p(1)$, 
and from \eqref{eq6-0101} and \eqref{eq6-0102}, we obtain that
$\hat\sigma_{1,2}^2 \pto (\sigma_{1,1}^*)^2$,
$\hat\sigma_{1,2}^2 \pto (\sigma_{1,2}^*)^2$ 
and then 
\begin{equation*}
(\hat\theta_0, \hat\theta_1,\hat\eta_1,\hat\theta_2,\hat\sigma^2)
\pto (\theta_0^*,\theta_1^*,\eta_1^*,\theta_2^*,(\sigma^*)^2).
\end{equation*}
\end{proof}

\subsection{Proofs of Proposition \ref{prop2}, Theorems \ref{th3} and \ref{th4}}
For $f\in\mathcal F_\beta$, choose two functions 
$g_1, g_2:\mathbb R_{+}\to\mathbb R$ satisfying the following conditions.
\begin{enumerate}
\item[(c)]
$f(s)=g_1(s)g_2(s)$ and $g_2(as)=g_2(a)g_2(s)$, $a,s>0$.

\item[(d)]
$s g_1'g_2(s^2)\ind_{[1,\infty)}$, 
$s g_1g_2'(s^2)\ind_{[1,\infty)}\in L^1(\mathbb R_{+})$.

\end{enumerate}
For $f=g_1g_2\in\mathcal F_\beta$ with (c) and (d), we define
\begin{equation*}
\widetilde{R}_{1,N}
=
O\Biggl(
\Delta_N
\biggl(
\int_{\Delta_N^{1/2}}^1 s |g_1'g_2(s^2)|\dd s
\lor
\int_{\Delta_N^{1/2}}^1 s |g_1g_2'(s^2)| \dd s
\lor
\int_{\Delta_N^{1/2}}^1 s^3|f''(s^2)|\dd s
\biggr)
\Biggr).
\end{equation*}
As an extension of Lemma \ref{lem1}, the following lemma holds.
\begin{lem}\label{lem5}
If $f=g_1g_2\in \mathcal F_\beta$ satisfies (c) and (d), then the following hold.
\begin{enumerate}
\item[(1)]
$\displaystyle\Delta_N\sum_{\bs k\in\mathbb N^2}
g_1(\lambda_{\bs k}\Delta_N)g_2(\mu_{\bs k}\Delta_N)
\\
=\frac{1}{4\pi\theta_2g_2(\theta_2)}
\int_0^\infty f(s)\dd s
-\frac{1}{g_2(\theta_2)}\int_{D_N^{(1)}\cup D_N^{(2)}}f(\theta_2\pi^2|x|^2) \dd x
+O(\widetilde{R}_{1,N}\lor \Delta_N^{1-\beta/2})$.

\item[(2)]
For any $y,z\in[\delta,1-\delta]$, 
\begin{align*}
\Delta_N\sum_{\bs k\in\mathbb N^2}
g_1(\lambda_{\bs k}\Delta_N)g_2(\mu_{\bs k}\Delta_N)\cos(2\pi k_1y)
&=
-\frac{1}{g_2(\theta_2)}\int_{D_N^{(1)}}f(\theta_2\pi^2|x|^2) \dd x
+O\biggl(\frac{\widetilde{R}_{1,N}\lor R_{2,N}}{\delta^2}\biggr),
\\
\Delta_N\sum_{\bs k\in\mathbb N^2}
g_1(\lambda_{\bs k}\Delta_N)g_2(\mu_{\bs k}\Delta_N)\cos(2\pi k_2z)
&=
-\frac{1}{g_2(\theta_2)}\int_{D_N^{(2)}}f(\theta_2\pi^2|x|^2) \dd x
+O\biggl(\frac{\widetilde{R}_{1,N}\lor R_{2,N}}{\delta^2}\biggr).
\end{align*}

\item[(3)]
For any $y,z\in[\delta,1-\delta]$,
\begin{align*}
\Delta_N\sum_{\bs k\in\mathbb N^2}
g_1(\lambda_{\bs k}\Delta_N)g_2(\mu_{\bs k}\Delta_N)
\cos(2\pi k_1 y)\cos(2\pi k_2 z)
=
O\biggl(\frac{\widetilde{R}_{1,N}\lor R_{2,N}\lor\Delta_N^{1-\beta/2}}{\delta^3}\biggr).
\end{align*}
\end{enumerate}
Furthermore, it holds that
\begin{align*}
\int_{D_N^{(1)}\cup D_N^{(2)}}f(\theta_2\pi^2|x|^2)\dd x
=O(R_N^{(\beta)}).
\end{align*}
\end{lem}

\begin{proof}
In a similar way to \eqref{eq1-0100}, if (d) holds, then it follows that
\begin{align*}
\int_{F_N}|g_1'g_2(|x|^2)|\dd x \lor \int_{F_N}|g_1g_2'(|x|^2)|\dd x
=O\biggl(
\int_{\Delta_N^{1/2}}^1 s |g_1'g_2(s^2)|\dd s
\lor
\int_{\Delta_N^{1/2}}^1 s |g_1g_2'(s^2)| \dd s
\biggr).
\end{align*}
By the Taylor expansion,  it follows that
\begin{align}
&\Biggl|
\Delta_N\sum_{\bs k\in\mathbb N^2}
g_1(\lambda_{\bs k}\Delta_N)g_2(\mu_{\bs k}\Delta_N)
-\Delta_N\sum_{\bs k\in\mathbb N^2}
g_1(\theta_2\pi^2|\bs k|^2\Delta_N)g_2(\pi^2|\bs k|^2\Delta_N)
\Biggr|
\nonumber
\\
&=
\Biggl|
\Delta_N\sum_{\bs k\in\mathbb N^2}
\biggl\{
g_1(\theta_2\pi^2|\bs k|^2\Delta_N)g_2(\pi^2|\bs k|^2\Delta_N)
\nonumber
\\
&\qquad\qquad
+\int_0^1 
g_1'\bigl(\theta_2\pi^2|\bs k|^2\Delta_N
+u(\lambda_{\bs k}-\theta_2\pi^2|\bs k|^2)\Delta_N\bigr)
g_2(\pi^2|\bs k|^2\Delta_N)
\dd u\Bigl(\frac{\theta_1^2+\eta_1^2}{4\theta_2}-\theta_0\Bigr)\Delta_N
\nonumber
\\
&\qquad\qquad
+\int_0^1 
g_1(\theta_2\pi^2|\bs k|^2\Delta_N)
g_2'\bigl(\pi^2|\bs k|^2\Delta_N
+u(\mu_{\bs k}-\pi^2|\bs k|^2)\Delta_N\bigr)
\dd u\, \mu_0\Delta_N
\biggr\}
\nonumber
\\
&\qquad\qquad
-\Delta_N\sum_{\bs k\in\mathbb N^2}
g_1(\theta_2\pi^2|\bs k|^2\Delta_N)g_2(\pi^2|\bs k|^2\Delta_N)
\Biggr|
\nonumber
\\
&\le
\Delta_N\sum_{\bs k\in\mathbb N^2}
\int_0^1 
\bigl|g_1'\bigl(\theta_2\pi^2|\bs k|^2\Delta_N
+u(\lambda_{\bs k}-\theta_2\pi^2|\bs k|^2)\Delta_N\bigr)
g_2(\pi^2|\bs k|^2\Delta_N)\bigr|
\dd u\Bigl|\frac{\theta_1^2+\eta_1^2}{4\theta_2}-\theta_0\Bigr|\Delta_N
\nonumber
\\
&\qquad
+\Delta_N\sum_{\bs k\in\mathbb N^2}
\int_0^1 
\bigl|g_1(\theta_2\pi^2|\bs k|^2\Delta_N)
g_2'\bigl(\pi^2|\bs k|^2\Delta_N
+u(\mu_{\bs k}-\pi^2|\bs k|^2)\Delta_N\bigr)\bigr|
\dd u|\mu_0|\Delta_N
\nonumber
\\
&=O\Biggl(
\Delta_N
\biggl(
\int_{F_N}|g_1'g_2(|x|^2)|\dd x \lor \int_{F_N}|g_1g_2'(|x|^2)|\dd x
\biggr)\Biggr)
=O(\widetilde{R}_{1,N}),
\label{eq7-0301}
\end{align}
and from (c), we obtain that
\begin{align*}
\Delta_N\sum_{\bs k\in\mathbb N^2}
g_1(\theta_2\pi^2|\bs k|^2\Delta_N)g_2(\pi^2|\bs k|^2\Delta_N)
&=\frac{\Delta_N}{g_2(\theta_2)}\sum_{\bs k\in\mathbb N^2}
g_1(\theta_2\pi^2|\bs k|^2\Delta_N)g_2(\theta_2\pi^2|\bs k|^2\Delta_N)
\\
&=
\frac{\Delta_N}{g_2(\theta_2)}\sum_{\bs k\in\mathbb N^2}
f(\theta_2\pi^2|\bs k|^2\Delta_N).
\end{align*}
In a similar way to the proof of Lemma \ref{lem1}, 
it holds from  \eqref{eq7-0301} that 
\begin{align*}
\Delta_N\sum_{\bs k\in\mathbb N^2}
f(\theta_2\pi^2|\bs k|^2\Delta_N)
=\frac{1}{4\pi\theta_2}
\int_0^\infty f(s)\dd s
-\int_{D_N^{(1)}\cup D_N^{(2)}}f(\theta_2\pi^2|x|^2) \dd x
+O(\widetilde{R}_{1,N}\lor \Delta_N^{1-\beta/2})
\end{align*}
and therefore we obtain (1). 
Similarly, it holds that
\begin{align*}
&\Delta_N\sum_{\bs k\in\mathbb N^2}
g_1(\lambda_{\bs k}\Delta_N)g_2(\mu_{\bs k}\Delta_N)
\cos(2\pi k_1 y)
\nonumber
\\
&=
\frac{\Delta_N}{g_2(\theta_2)}\sum_{\bs k\in\mathbb N^2}
f(\theta_2\pi^2|\bs k|^2\Delta_N)\cos(2\pi k_1 y)
+O(\widetilde{R}_{1,N}),
\end{align*}
\begin{align*}
&\Delta_N\sum_{\bs k\in\mathbb N^2}
g_1(\lambda_{\bs k}\Delta_N)g_2(\mu_{\bs k}\Delta_N)
\cos(2\pi k_1 y)\cos(2\pi k_2 z)
\nonumber
\\
&=
\frac{\Delta_N}{g_2(\theta_2)}\sum_{\bs k\in\mathbb N^2}
f(\theta_2\pi^2|\bs k|^2\Delta_N)\cos(2\pi k_1 y)\cos(2\pi k_2 z)
+O(\widetilde{R}_{1,N}),
\end{align*}
and therefore from Lemma \ref{lem1}, we obtain (2) and (3). 
\end{proof}

The functions $f_\alpha$ and $f_{\alpha,\tau}$ in \eqref{functions} 
can be expressed as
\begin{equation*}
f_\alpha(s)=\frac{1-\ee^{-s}}{s^{1+\alpha}}
=\frac{1-\ee^{-s}}{s}\cdot\frac{1}{s^\alpha}
=:g_1(s)g_{2,\alpha}(s),
\end{equation*}
\begin{equation*}
f_{\alpha,\tau}(s)=\frac{(1-\ee^{-s})^2}{s^{1+\alpha}}\ee^{-\tau s}
=\frac{(1-\ee^{-s})^2}{s}\ee^{-\tau s}\cdot\frac{1}{s^\alpha}
=:g_{1,\tau}(s)g_{2,\alpha}(s),
\end{equation*}
and 
satisfy (c) and (d). Hence, it follows 
from Lemma \ref{lem5} and an analogous proof of Lemma \ref{lem2} that
\begin{align}
&\Delta_N
\sum_{\bs k\in\mathbb N^2}
g_1(\lambda_{\bs k}\Delta_N)g_{2,\alpha}(\mu_{\bs k}\Delta_N)e_{\bs k}^2(y,z)
=\frac{\Gamma(1-\alpha)}{4\pi\alpha\theta_2^{1-\alpha}}
\ee^{-\frac{\theta_1}{\theta_2}y}\ee^{-\frac{\eta_1}{\theta_2}z}
+O(\delta^{-3}\Delta_{N}^{1-\alpha}),
\label{eq7-0302}
\\
&\Delta_N
\sum_{\bs k\in\mathbb N^2}
g_{1,\tau}(\lambda_{\bs k}\Delta_N)g_{2,\alpha}(\mu_{\bs k}\Delta_N)e_{\bs k}^2(y,z)
\nonumber
\\
&=
\frac{\Gamma(1-\alpha)}{4\pi\alpha\theta_2^{1-\alpha}}
\bigl(
-\tau^\alpha+2(1+\tau)^\alpha-(2+\tau)^\alpha
\bigr)
\ee^{-\frac{\theta_1}{\theta_2}y}\ee^{-\frac{\eta_1}{\theta_2}z}
+O(\delta^{-3}\Delta_{N}^{1-\alpha}),
\label{eq7-0303}
\end{align}
\begin{align}
\sum_{\bs k\in\mathbb N^2}
\frac{1-\ee^{-\lambda_{\bs k}\Delta_N}}{\lambda_{\bs k}\mu_{\bs k}^\alpha}
&=
\Delta_N^{1+\alpha}
\sum_{\bs k\in\mathbb N^2}
f_\alpha(\lambda_{\bs k}\Delta_N)
\biggl(\frac{\lambda_{\bs k}}{\mu_{\bs k}}\biggr)^\alpha
=O(\Delta_N^\alpha),
\label{eq7-0304}
\\
\sum_{\bs k\in\mathbb N^2}
\frac{(1-\ee^{-\lambda_{\bs k}\Delta_N})^2}{\lambda_{\bs k}\mu_{\bs k}^\alpha}
&=
\Delta_N^{1+\alpha}
\sum_{\bs k\in\mathbb N^2}
f_{\alpha,0}(\lambda_{\bs k}\Delta_N)
\biggl(\frac{\lambda_{\bs k}}{\mu_{\bs k}}\biggr)^\alpha
=O(\Delta_N^\alpha).
\label{eq7-0305}
\end{align}

Set 
\begin{align*}
B_{1,i,\bs k}^{Q_2}
&=
-\frac{\sigma(1-\ee^{-\lambda_{\bs k}\Delta_N})}{\mu_{\bs k}^{\alpha/2}}
\int_0^{(i-1)\Delta_N}\ee^{-\lambda_{\bs k}((i-1)\Delta_N-s)}\dd w_{\bs k}(s),
\\
B_{2,i,\bs k}^{Q_2}
&=
\frac{\sigma}{\mu_{\bs k}^{\alpha/2}}
\int_{(i-1)\Delta_N}^{i\Delta_N}
\ee^{-\lambda_{\bs k}(i\Delta_N-s)}\dd w_{\bs k}(s),
\end{align*}
and $B_{i,\bs k}^{Q_2}=B_{1,i,\bs k}^{Q_2}+B_{2,i,\bs k}^{Q_2}$.
The increment $\Delta_i x_{\bs k}^{Q_2}$ can be expressed as 
$\Delta_i x_{\bs k}^{Q_2}=A_{i,\bs k}+B_{i,\bs k}^{Q_2}$.

From \eqref{eq7-0302}-\eqref{eq7-0305}, 
the following lemma holds by an analogous proof of Lemma \ref{lem4}.
\begin{lem}\label{lem6}
It holds that uniformly in $(y,z)\in D_\delta$,
\begin{equation*}
\sum_{\bs k,\bs \ell\in\mathbb N^2}
\EE[B_{i,\bs k}^{Q_2}B_{i,\bs \ell}^{Q_2}]e_{\bs k}(y,z)e_{\bs \ell}(y,z)
=\Delta_N^{\alpha}\frac{\sigma^2\Gamma(1-\alpha)}{4\pi\alpha\theta_2^{1-\alpha}}
\ee^{-\frac{\theta_1}{\theta_2}y}\ee^{-\frac{\eta_1}{\theta_2}z}
+O(\Delta_N).
\end{equation*}
Moreover, it holds that uniformly in $(y,z)\in D_\delta$,
\begin{equation*}
\sup_{i\neq j}\sum_{j=1}^N
\sum_{\bs k,\bs \ell\in\mathbb N^2}
\Bigl|
\EE[B_{i,\bs k}^{Q_2}B_{j,\bs \ell}^{Q_2}]e_{\bs k}(y,z)e_{\bs \ell}(y,z)
\Bigr|
=O(1).
\end{equation*}
\end{lem}

\begin{proof}[\bf Proof of Proposition \ref{prop2}]
It can be shown in the same way as the proof of Proposition \ref{prop1}
by using Lemmas \ref{lem3} and \ref{lem6}.
\end{proof}

\begin{proof}[\bf Proof of Theorem \ref{th3}]
It is shown similarly to the proof of Theorem \ref{th1}.
\end{proof}

\begin{proof}[\bf Proof of Theorem \ref{th4}]
(a) In the same way as the proof of Theorem \ref{th2}, 
it holds that under $\frac{n^{1-\alpha}}{mN^{2\gamma}}\to0$ and
$\frac{n^{1-\alpha+\epsilon}}{M_1^{2\epsilon} \land M_2^{2\epsilon}}\to0$,
\begin{equation}\label{eq7-0401}
\check\sigma^2 \pto (\sigma^*)^2,
\end{equation}
and that under $\frac{n^{2-\alpha}}{mN^{2\gamma}}\to0$ and
$\frac{n^{2-\alpha+\epsilon}}{M_1^{2\epsilon} \land M_2^{2\epsilon}}\to0$,
\begin{equation}\label{eq7-0402}
\sqrt n(\check\sigma^2-(\sigma^*)^2) \dto N(0,2(\sigma^*)^4 ).
\end{equation}

(1) We obtain from \eqref{eq7-0401} and Theorem \ref{th3} that 
$\check\theta_1\pto\theta_1^*$, $\check\eta_1\pto\eta_1^*$ and 
$\check\theta_2\pto\theta_2^*$.

(2) Since it follows from Theorem \ref{th3} that
under 
$\frac{n^{2-\alpha}}{mN^{2\gamma}}\to0$ and
$\frac{n^{2-\alpha+\epsilon}}{M_1^{2\epsilon} \land M_2^{2\epsilon}}\to0$,
\begin{align}
\sqrt n
\bigl(\check s^{\frac{-1}{1-\alpha}}-(s^*)^{\frac{-1}{1-\alpha}}\bigr)
&=\frac{\sqrt n}{m^{1/2} N^\gamma}
\int_0^1\frac{-1}{1-\alpha}
\bigl(s^*+u(\check s-s^*)\bigr)^{\frac{-1}{1-\alpha}-1}\dd u\,
m^{1/2} N^\gamma(\check s-s^*)
\nonumber
\\
&=o_p(1),
\nonumber
\\
\sqrt n\bigl(
\check\kappa \check s^{\frac{-1}{1-\alpha}}-\kappa^* (s^*)^{\frac{-1}{1-\alpha}}
\bigr)
&=\check\kappa \sqrt n\bigl(
\check s^{\frac{-1}{1-\alpha}}-(s^*)^{\frac{-1}{1-\alpha}}
\bigr)
+\frac{\sqrt n}{m^{1/2} N^\gamma} (s^*)^{\frac{-1}{1-\alpha}}
m^{1/2} N^\gamma(\check\kappa -\kappa^*)
\nonumber
\\
&=o_p(1),
\nonumber
\\
\sqrt n\bigl(
\check\eta \check s^{\frac{-1}{1-\alpha}}-\eta^* (s^*)^{\frac{-1}{1-\alpha}}
\bigr)
&=\check\eta \sqrt n\bigl(
\check s^{\frac{-1}{1-\alpha}}-(s^*)^{\frac{-1}{1-\alpha}}
\bigr)
+\frac{\sqrt n}{m^{1/2} N^\gamma} (s^*)^{\frac{-1}{1-\alpha}}
m^{1/2} N^\gamma(\check\eta -\eta^*)
\nonumber
\\
&=o_p(1),
\nonumber
\\
\sqrt n\bigl(\check\sigma^{\frac{2}{1-\alpha}}-(\sigma^*)^{\frac{2}{1-\alpha}}\bigr)
&=\frac{(\sigma^*)^{\frac{2\alpha}{1-\alpha}}}{1-\alpha}
\sqrt n(\check\sigma^2-(\sigma^*)^2)+o_p(1),
\nonumber
\end{align}
one has that
\begin{align}
\sqrt n(\check\theta_2-\theta_2^*)
&=\sqrt n
\bigl(
\check\sigma^{\frac{2}{1-\alpha}}
\check s^{\frac{-1}{1-\alpha}}
-(\sigma^*)^{\frac{2}{1-\alpha}}
(s^*)^{\frac{-1}{1-\alpha}}
\bigl)
\nonumber
\\
&=\check\sigma^{\frac{2}{1-\alpha}}
\sqrt n\bigl(\check s^{\frac{-1}{1-\alpha}}-(s^*)^{\frac{-1}{1-\alpha}}\bigl)
+(s^*)^{\frac{-1}{1-\alpha}}
\sqrt n\bigl(\check\sigma^{\frac{2}{1-\alpha}}-(\sigma^*)^{\frac{2}{1-\alpha}}\bigl)
\nonumber
\\
&=(s^*)^{\frac{-1}{1-\alpha}}
\frac{(\sigma^*)^{\frac{2\alpha}{1-\alpha}}}{1-\alpha}
\sqrt n(\check\sigma^2-(\sigma^*)^2)+o_p(1)
\nonumber
\\
&=\frac{\theta_2^*}{(1-\alpha)(\sigma^*)^2}
\sqrt n(\check\sigma^2-(\sigma^*)^2)+o_p(1),
\label{eq7-0403}
\\
\sqrt n(\check\theta_1-\theta_1^*)
&=\sqrt n
\bigl(
\check\sigma^{\frac{2}{1-\alpha}}
\check\kappa\check s^{\frac{-1}{1-\alpha}}
-(\sigma^*)^{\frac{2}{1-\alpha}}
\kappa^*(s^*)^{\frac{-1}{1-\alpha}}
\bigl)
\nonumber
\\
&=\kappa^*(s^*)^{\frac{-1}{1-\alpha}}
\frac{(\sigma^*)^{\frac{2\alpha}{1-\alpha}}}{1-\alpha}
\sqrt n(\check\sigma^2-(\sigma^*)^2)+o_p(1)
\nonumber
\\
&=\frac{\theta_1^*}{(1-\alpha)(\sigma^*)^2}
\sqrt n(\check\sigma^2-(\sigma^*)^2)+o_p(1),
\label{eq7-0404}
\\
\sqrt n(\check\eta_1-\eta_1^*)
&=\sqrt n
\bigl(
\check\sigma^{\frac{2}{1-\alpha}}
\check\eta\check s^{\frac{-1}{1-\alpha}}
-(\sigma^*)^{\frac{2}{1-\alpha}}
\eta^*(s^*)^{\frac{-1}{1-\alpha}}
\bigl)
\nonumber
\\
&=\eta^*(s^*)^{\frac{-1}{1-\alpha}}
\frac{(\sigma^*)^{\frac{2\alpha}{1-\alpha}}}{1-\alpha}
\sqrt n(\check\sigma^2-(\sigma^*)^2)+o_p(1)
\nonumber
\\
&=\frac{\eta_1^*}{(1-\alpha)(\sigma^*)^2}
\sqrt n(\check\sigma^2-(\sigma^*)^2)+o_p(1),
\label{eq7-0405}
\end{align}
and therefore, we obtain from \eqref{eq7-0402} that
\begin{align*}
\sqrt n
\begin{pmatrix}
\check\theta_1-\theta_1^*
\\
\check\eta_1-\eta_1^*
\\
\check\theta_2-\theta_2^*
\\
\check\sigma^2-(\sigma^*)^2
\end{pmatrix}
&=\frac{1}{(1-\alpha)(\sigma^*)^2}
\begin{pmatrix}
\theta_1^*
\\
\eta_1^*
\\
\theta_2^*
\\
(1-\alpha)(\sigma^*)^2
\end{pmatrix}
\sqrt n(\check\sigma^2-(\sigma^*)^2)+o_p(1)
\\
&=\frac{\bs \nu_{-1}^*}{(1-\alpha)(\sigma^*)^2}
\sqrt n(\check\sigma^2-(\sigma^*)^2)+o_p(1)
\\
&\dto
N\biggl(0,\frac{2}{(1-\alpha)^2}\bs \nu_{-1}^* (\bs \nu_{-1}^*)^\TT\biggr)
=N(0,K).
\end{align*}

(b) As in the proof of Theorem \ref{th2}, 
it holds that under $\frac{n^{1-\alpha}}{mN^{2\gamma}}\to0$ and
$\frac{n^{1-\alpha+\epsilon}}{M_1^{2\epsilon} \land M_2^{2\epsilon}}\to0$,
\begin{equation}\label{eq7-0406}
(\bar\tau_{1,1}^2,\bar\tau_{1,2}^2) 
\pto ((\tau_{1,1}^*)^2, (\tau_{1,2}^*)^2),
\end{equation}
and that under $\frac{n^{2-\alpha}}{mN^{2\gamma}}\to0$ and
$\frac{n^{2-\alpha+\epsilon}}{M_1^{2\epsilon} \land M_2^{2\epsilon}}\to0$,
\begin{equation}\label{eq7-0407}
\sqrt n
\begin{pmatrix}
\bar\tau_{1,1}^2-(\tau_{1,1}^*)^2 
\\
\bar\tau_{1,2}^2-(\tau_{1,2}^*)^2
\end{pmatrix}
\dto 
N\Biggl(0,
\begin{pmatrix}
2(\tau_{1,1}^*)^4 & 0
\\
0 & 2(\tau_{1,2}^*)^4 
\end{pmatrix}
\Biggr).
\end{equation}

(1) It is clear from \eqref{eq7-0406} and Theorem \ref{th3} 
that the estimators are consistent.

(2) Let $\bs \nu^*=(0,\bs \nu_{-1}^*)^\TT$ and
\begin{equation*}
\Phi(x,y)=3\pi^2(x^{1/\alpha}y^{-1/\alpha}-1)^{-1}, \quad
\Psi(x,y)=\{3\pi^2(y^{-1/\alpha}-x^{-1/\alpha})^{-1}\}^\alpha.
\end{equation*}
It then follows that
\begin{align*}
\sqrt n(\bar\mu_0-\mu_0^*)
&=\sqrt n\bigl\{
\Phi(\bar\tau_{1,1}^2,\bar\tau_{1,2}^2)
-\Phi((\tau_{1,1}^*)^2,(\tau_{1,2}^*)^2)
\bigr\}
\\
&=\partial \Phi((\tau_{1,1}^*)^2,(\tau_{1,2}^*)^2)
\sqrt n
\begin{pmatrix}
\bar\tau_{1,1}^2-(\tau_{1,1}^*)^2 
\\
\bar\tau_{1,2}^2-(\tau_{1,2}^*)^2
\end{pmatrix}
+o_p(1),
\\
\sqrt n(\bar\sigma^2-(\sigma^*)^2)
&=\sqrt n\bigl\{
\Psi(\bar\tau_{1,1}^2,\bar\tau_{1,2}^2)
-\Psi((\tau_{1,1}^*)^2,(\tau_{1,2}^*)^2)
\bigr\}
\\
&=\partial \Psi((\tau_{1,1}^*)^2,(\tau_{1,2}^*)^2)
\sqrt n
\begin{pmatrix}
\bar\tau_{1,1}^2-(\tau_{1,1}^*)^2 
\\
\bar\tau_{1,2}^2-(\tau_{1,2}^*)^2
\end{pmatrix}
+o_p(1).
\end{align*}
In the same way as \eqref{eq7-0403}-\eqref{eq7-0405}, it holds that
\begin{align*}
\sqrt n(\bar\theta_2-\theta_2^*)
&=\frac{\theta_2^*}{(1-\alpha)(\sigma^*)^2}
\sqrt n(\bar\sigma^2-(\sigma^*)^2)+o_p(1),
\\
\sqrt n(\bar\theta_1-\theta_1^*)
&=\frac{\theta_1^*}{(1-\alpha)(\sigma^*)^2}
\sqrt n(\bar\sigma^2-(\sigma^*)^2)+o_p(1),
\\
\sqrt n(\bar\eta_1-\eta_1^*)
&=\frac{\eta_1^*}{(1-\alpha)(\sigma^*)^2}
\sqrt n(\bar\sigma^2-(\sigma^*)^2)+o_p(1).
\end{align*}
Therefore, one has from \eqref{eq7-0407} 
that under $\frac{n^{2-\alpha}}{mN^{2\gamma}}\to0$ and
$\frac{n^{2-\alpha+\epsilon}}{M_1^{2\epsilon} \land M_2^{2\epsilon}}\to0$,
\begin{align*}
\sqrt n
\begin{pmatrix}
\bar \mu_0-\mu_0^*
\\
\bar\theta_1-\theta_1^*
\\
\bar\eta_1-\eta_1^*
\\
\bar\theta_2-\theta_2^*
\\
\bar\sigma^2-(\sigma^*)^2
\end{pmatrix}
&=
\biggl\{
\bs e_1 \partial \Phi((\tau_{1,1}^*)^2,(\tau_{1,2}^*)^2)
+\bs \nu^*
\frac{\partial \Psi((\tau_{1,1}^*)^2,(\tau_{1,2}^*)^2)}{(1-\alpha)(\sigma^*)^2}
\biggr\}
\sqrt n
\begin{pmatrix}
\bar\tau_{1,1}^2-(\tau_{1,1}^*)^2 
\\
\bar\tau_{1,2}^2-(\tau_{1,2}^*)^2
\end{pmatrix}
+o_p(1)
\\
&=:B\sqrt n
\begin{pmatrix}
\bar\tau_{1,1}^2-(\tau_{1,1}^*)^2 
\\
\bar\tau_{1,2}^2-(\tau_{1,2}^*)^2
\end{pmatrix}
+o_p(1)
\\
&\dto N(0,L),
\end{align*}
where 
\begin{equation*}
L=B\begin{pmatrix}
2(\tau_{1,1}^*)^4 & 0
\\
0 & 2(\tau_{1,2}^*)^4 
\end{pmatrix}
B^\TT.
\end{equation*}
Since
\begin{equation*}
\partial \Phi(\tau_{1,1}^2,\tau_{1,2}^2)
=\frac{\mu_{1,1}\mu_{1,2}}{3\pi^2\alpha\sigma^2}(-\mu_{1,1}^\alpha,\mu_{1,2}^\alpha),
\quad
\partial \Psi(\tau_{1,1}^2,\tau_{1,2}^2)
=\frac{1}{3\pi^2}(-\mu_{1,1}^{1+\alpha},\mu_{1,2}^{1+\alpha})
\end{equation*}
and $\tau_{k,\ell}^4=\sigma^4\mu_{k,\ell}^{-2\alpha}$, 
it follows that
\begin{equation*}
\partial \Phi((\tau_{1,1}^*)^2,(\tau_{1,2}^*)^2)
\begin{pmatrix}
2(\tau_{1,1}^*)^4 & 0
\\
0 & 2(\tau_{1,2}^*)^4 
\end{pmatrix}
\partial \Phi((\tau_{1,1}^*)^2,(\tau_{1,2}^*)^2)^\TT
=\frac{4(\mu_{1,1}^*)^2(\mu_{1,2}^*)^2}{9\pi^4\alpha^2}
=\frac{2d_1}{9\pi^4},
\end{equation*}
\begin{equation*}
\partial \Phi((\tau_{1,1}^*)^2,(\tau_{1,2}^*)^2)
\begin{pmatrix}
2(\tau_{1,1}^*)^4 & 0
\\
0 & 2(\tau_{1,2}^*)^4 
\end{pmatrix}
\frac{\partial \Psi((\tau_{1,1}^*)^2,(\tau_{1,2}^*)^2)^\TT}{(1-\alpha)(\sigma^*)^2}
=\frac{2\mu_{1,1}^*\mu_{1,2}^*(\mu_{1,1}^*+\mu_{1,2}^*)}{9\pi^4\alpha(1-\alpha)}
=\frac{2d_2}{9\pi^4},
\end{equation*}
\begin{equation*}
\frac{\partial \Psi((\tau_{1,1}^*)^2,(\tau_{1,2}^*)^2)}{(1-\alpha)(\sigma^*)^2}
\begin{pmatrix}
2(\tau_{1,1}^*)^4 & 0
\\
0 & 2(\tau_{1,2}^*)^4 
\end{pmatrix}
\frac{\partial \Psi((\tau_{1,1}^*)^2,(\tau_{1,2}^*)^2)^\TT}{(1-\alpha)(\sigma^*)^2}
=\frac{2\{(\mu_{1,1}^*)^2+(\mu_{1,2}^*)^2\}}{9\pi^4(1-\alpha)^2}
=\frac{2d_3}{9\pi^4}.
\end{equation*}
Hence, it holds that
\begin{align*}
L
&=
\frac{2}{9\pi^4}\Bigl(
d_1\bs e_1\bs e_1^\TT
+d_2\bigl\{\bs e_1 (\bs \nu^*)^\TT+\bs \nu^* \bs e_1^\TT \bigr\}
+d_3\bs \nu^*(\bs \nu^*)^\TT
\Bigr)
\\
&=
\frac{2}{9\pi^4}
\begin{pmatrix}
d_1 & d_2 (\bs \nu_{-1}^*)^\TT
\\
d_2 \bs \nu_{-1}^* & d_3 \bs \nu_{-1}^*(\bs \nu_{-1}^*)^\TT
\end{pmatrix}.
\end{align*}
\end{proof}

\section*{Appendix} \label{appendix}

In order to help us understand the characteristics of the parameters 
of the SPDE \eqref{2d_spde}, 
we can refer some sample paths
with different values of the parameters as follows.

\subsection*{characteristic of $\kappa = \theta_1/\theta_2$}
$\kappa$ affects the variation of the sample path when y is varied.
Figures \ref{fig4}-\ref{fig6} are the sample paths, where $\kappa$ is changed.

Figure \ref{fig4} is the sample path with $\theta=(0,1.5,0.2,0.2,1)$ and $\kappa=7.5$.
Figure \ref{fig4} shows that if $\kappa$ is large, the variation of the sample path  
is large near $y = 0$ and small near $y = 1$.
\begin{figure}[H]
\begin{minipage}{0.32\hsize}
\begin{center}
\includegraphics[width=4.5cm]{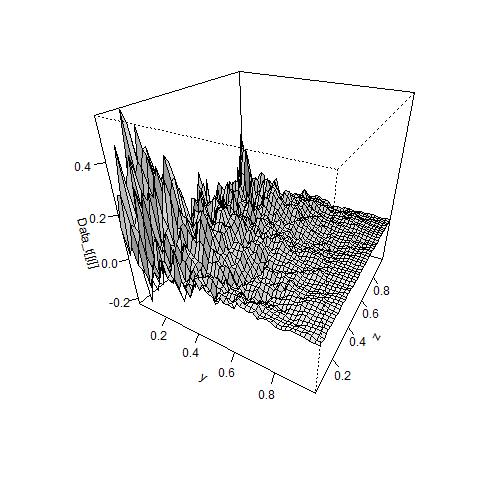}
\captionsetup{labelformat=empty,labelsep=none}
\subcaption{Cross sections of 
the sample path at $t = 0.5$}
\end{center}
\end{minipage}
\begin{minipage}{0.32\hsize}
\begin{center}
\includegraphics[width=4.5cm]{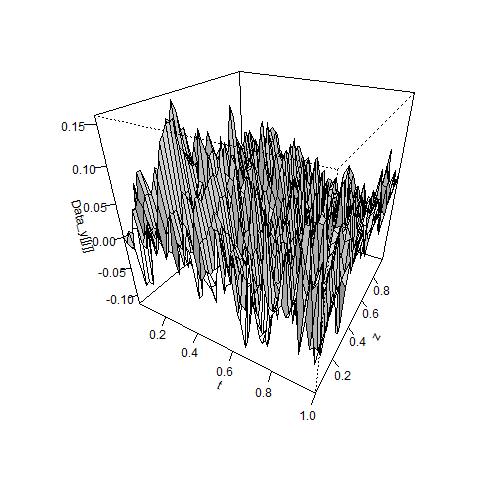}
\captionsetup{labelformat=empty,labelsep=none}
\subcaption{Cross sections of 
the sample path at $y = 0.5$}
\end{center}
\end{minipage}
\begin{minipage}{0.32\hsize}
\begin{center}
\includegraphics[width=4.5cm]{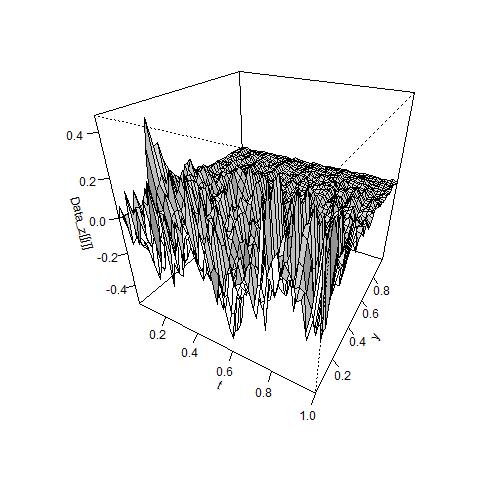}
\captionsetup{labelformat=empty,labelsep=none}
\subcaption{Cross sections of 
the sample path at $z = 0.5$}
\end{center}
\end{minipage}
\caption{Sample path with $\theta=(0,1.5,0.2,0.2,1)$ and $\kappa=7.5$}
\label{fig4}
\end{figure}

Figure \ref{fig5} is the sample path with $\theta=(0,0,0.2,0.2,1)$ and $\kappa=0$.
Figure \ref{fig5} shows that if $\kappa$ is close to $0$, there is not much difference 
between the variation of the sample path near $y=0$ and that near $y=1$.
\begin{figure}[H]
\begin{minipage}{0.32\hsize}
\begin{center}
\includegraphics[width=4.5cm]{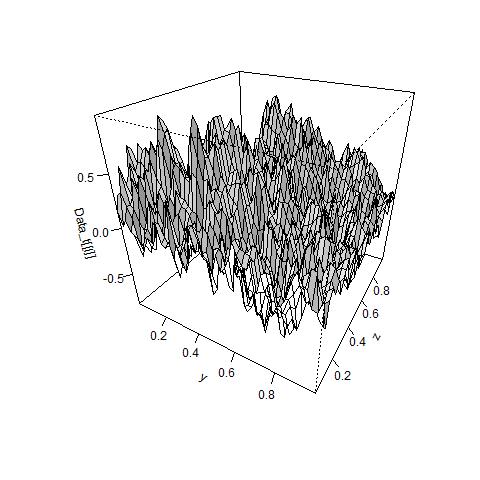}
\captionsetup{labelformat=empty,labelsep=none}
\subcaption{Cross sections of 
the sample path at $t = 0.5$}
\end{center}
\end{minipage}
\begin{minipage}{0.32\hsize}
\begin{center}
\includegraphics[width=4.5cm]{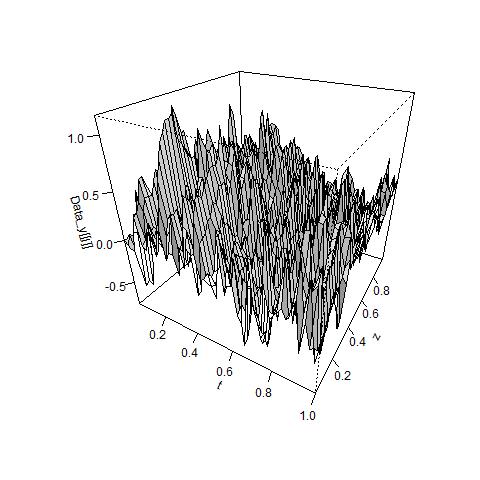}
\captionsetup{labelformat=empty,labelsep=none}
\subcaption{Cross sections of 
the sample path at $y = 0.5$}
\end{center}
\end{minipage}
\begin{minipage}{0.32\hsize}
\begin{center}
\includegraphics[width=4.5cm]{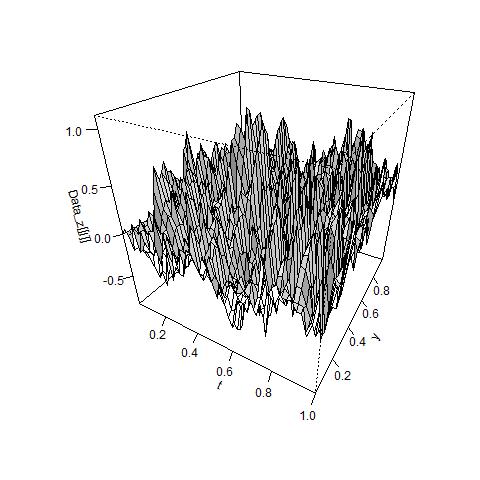}
\captionsetup{labelformat=empty,labelsep=none}
\subcaption{Cross sections of 
the sample path at $z = 0.5$}
\end{center}
\end{minipage}
\caption{Sample path with $\theta=(0,0,0.2,0.2,1)$ and $\kappa=0$}
\label{fig5}
\end{figure}

Figure \ref{fig6} is the sample path with $\theta=(0,-1.5,0.2,0.2,1)$ and $\kappa=-7.5$.
Figure \ref{fig6} shows that if $\kappa$ is small, the variation of the sample path  
is large near $y = 1$ and small near $y = 0$.
\begin{figure}[H]
\begin{minipage}{0.32\hsize}
\begin{center}
\includegraphics[width=4.5cm]{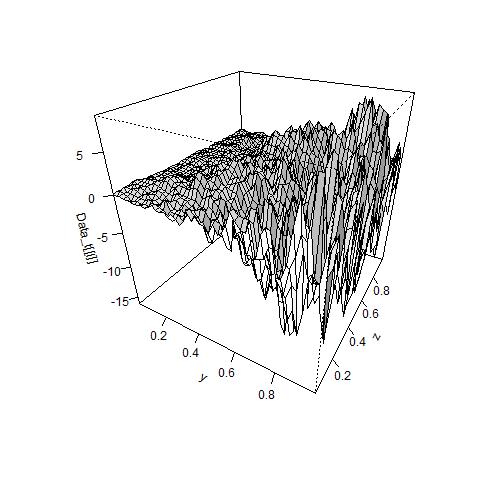}
\captionsetup{labelformat=empty,labelsep=none}
\subcaption{Cross sections of 
the sample path at $t = 0.5$}
\end{center}
\end{minipage}
\begin{minipage}{0.32\hsize}
\begin{center}
\includegraphics[width=4.5cm]{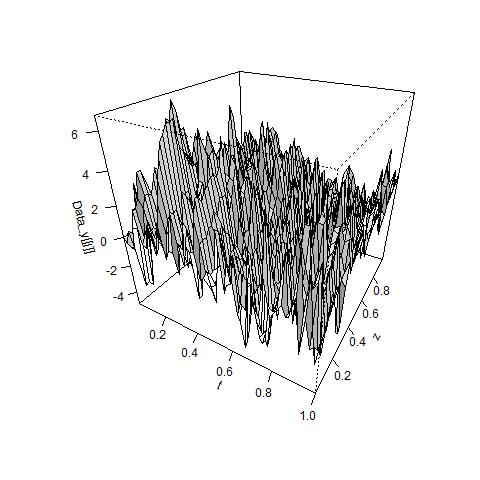}
\captionsetup{labelformat=empty,labelsep=none}
\subcaption{Cross sections of 
the sample path at $y = 0.5$}
\end{center}
\end{minipage}
\begin{minipage}{0.32\hsize}
\begin{center}
\includegraphics[width=4.5cm]{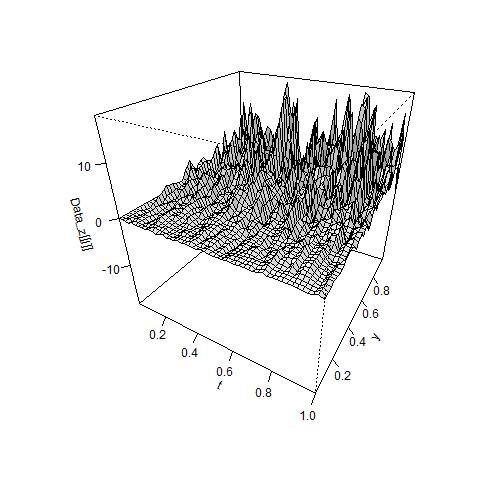}
\captionsetup{labelformat=empty,labelsep=none}
\subcaption{Cross sections of 
the sample path at $z = 0.5$}
\end{center}
\end{minipage}
\caption{Sample path with $\theta=(0,-1.5,0.2,0.2,1)$ and $\kappa=-7.5$}
\label{fig6}
\end{figure}

\subsection*{characteristic of $\eta = \eta_1/\theta_2$}
$\eta$ affects the variation of the sample path when $z$ is varied.
Figures \ref{fig7}-\ref{fig9} are the sample paths, where $\eta$ is changed.

Figure \ref{fig7} is the sample path with $\theta=(0,0.2,1.5,0.2,1)$ and $\eta=7.5$.
Figure \ref{fig7} shows that if $\eta$ is large, the variation of the sample path  
is large near $z = 0$ and small near $z = 1$.
\begin{figure}[H]
\begin{minipage}{0.32\hsize}
\begin{center}
\includegraphics[width=4.5cm]{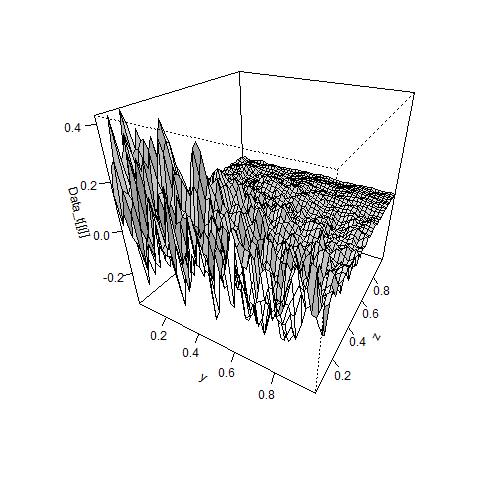}
\captionsetup{labelformat=empty,labelsep=none}
\subcaption{Cross sections of 
the sample path at $t = 0.5$}
\end{center}
\end{minipage}
\begin{minipage}{0.32\hsize}
\begin{center}
\includegraphics[width=4.5cm]{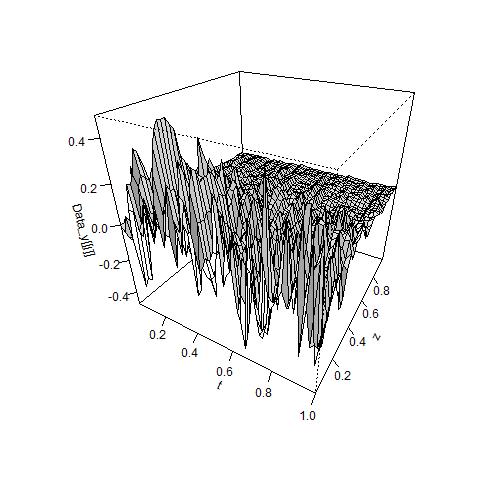}
\captionsetup{labelformat=empty,labelsep=none}
\subcaption{Cross sections of 
the sample path at $y = 0.5$}
\end{center}
\end{minipage}
\begin{minipage}{0.32\hsize}
\begin{center}
\includegraphics[width=4.5cm]{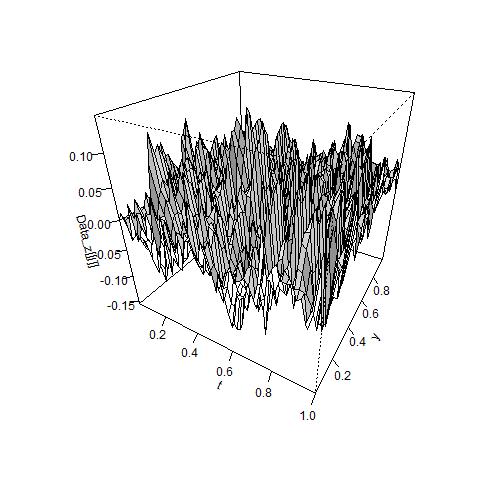}
\captionsetup{labelformat=empty,labelsep=none}
\subcaption{Cross sections of 
the sample path at $z = 0.5$}
\end{center}
\end{minipage}
\caption{Sample path with $\theta=(0,0.2,1.5,0.2,1)$ and $\eta=7.5$}
\label{fig7}
\end{figure}

Figure \ref{fig8} is the sample path with $\theta=(0,0.2,0,0.2,1)$ and $\kappa=0$.
Figure \ref{fig8} shows that if $\eta$ is close to $0$, there is not much difference 
between the variation of the sample path near $z=0$ and that near $z=1$.
\begin{figure}[H]
\begin{minipage}{0.32\hsize}
\begin{center}
\includegraphics[width=4.5cm]{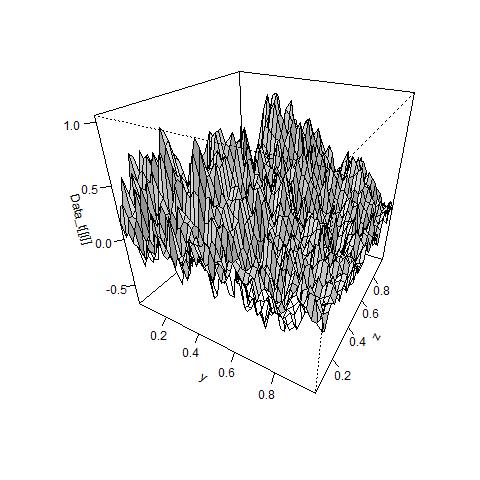}
\captionsetup{labelformat=empty,labelsep=none}
\subcaption{Cross sections of 
the sample path at $t = 0.5$}
\end{center}
\end{minipage}
\begin{minipage}{0.32\hsize}
\begin{center}
\includegraphics[width=4.5cm]{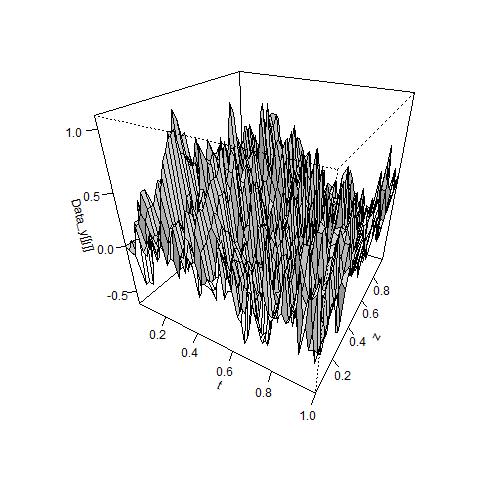}
\captionsetup{labelformat=empty,labelsep=none}
\subcaption{Cross sections of 
the sample path at $y = 0.5$}
\end{center}
\end{minipage}
\begin{minipage}{0.32\hsize}
\begin{center}
\includegraphics[width=4.5cm]{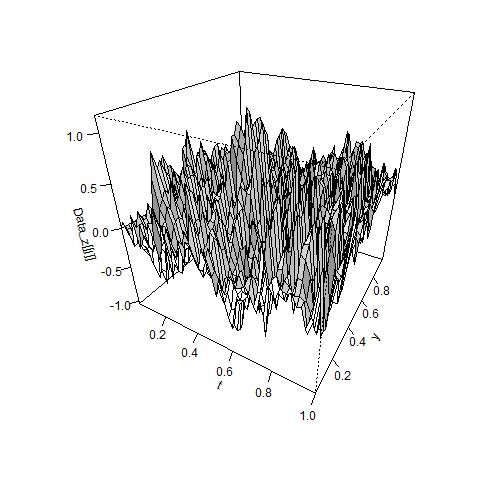}
\captionsetup{labelformat=empty,labelsep=none}
\subcaption{Cross sections of 
the sample path at $z = 0.5$}
\end{center}
\end{minipage}
\caption{Sample path with $\theta=(0,0.2,0,0.2,1)$ and $\eta=0$}
\label{fig8}
\end{figure}

Figure \ref{fig9} is the sample path with $\theta=(0,0.2,-1.5,0.2,1)$ and $\eta=-7.5$.
Figure \ref{fig9} shows that if $\eta$ is small, the variation of the sample path  
is large near $z = 1$ and small near $z = 0$.
\begin{figure}[H]
\begin{minipage}{0.32\hsize}
\begin{center}
\includegraphics[width=4.5cm]{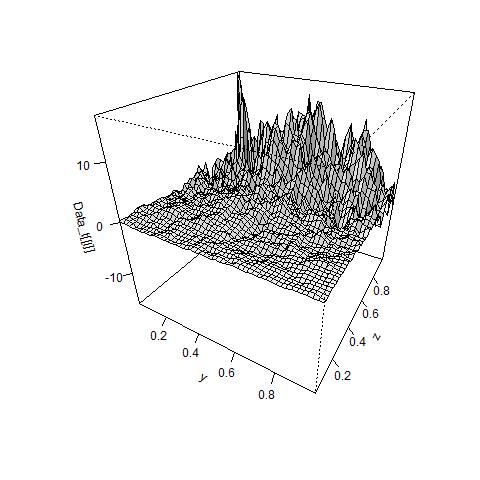}
\captionsetup{labelformat=empty,labelsep=none}
\subcaption{Cross sections of 
the sample path at $t = 0.5$}
\end{center}
\end{minipage}
\begin{minipage}{0.32\hsize}
\begin{center}
\includegraphics[width=4.5cm]{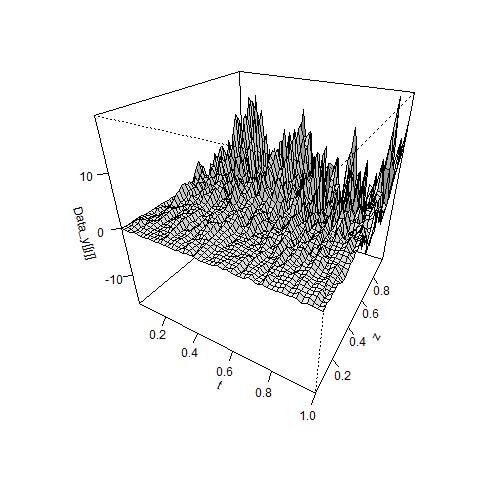}
\captionsetup{labelformat=empty,labelsep=none}
\subcaption{Cross sections of 
the sample path at $y = 0.5$}
\end{center}
\end{minipage}
\begin{minipage}{0.32\hsize}
\begin{center}
\includegraphics[width=4.5cm]{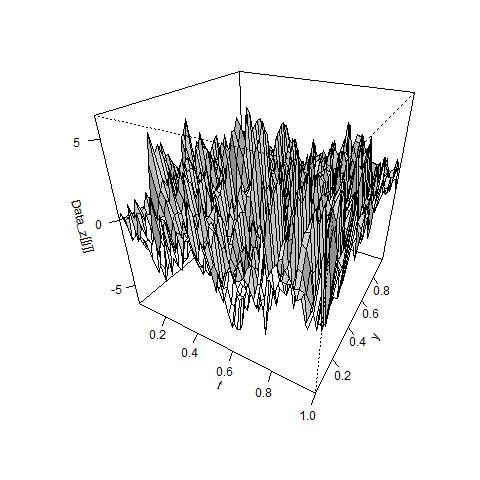}
\captionsetup{labelformat=empty,labelsep=none}
\subcaption{Cross sections of 
the sample path at $z = 0.5$}
\end{center}
\end{minipage}
\caption{Sample path with $\theta=(0,0.2,-1.5,0.2,1)$ and $\kappa=-7.5$}
\label{fig9}
\end{figure}


\end{document}